\pdfoutput=1
\RequirePackage{ifpdf}
\ifpdf 
\documentclass[pdftex]{sigma}
\else
\documentclass{sigma}
\fi

\usepackage{ytableau}
\ytableausetup{aligntableaux=top,mathmode,boxsize=1em}
\def\young(#1){\ytableaushort{#1}}
\def\yng(#1){\tiny {\ydiagram{#1}}}

\numberwithin{equation}{section}

\newtheorem{Theorem}{Theorem}[section]
\newtheorem*{Theorem*}{Theorem}
\newtheorem{Corollary}[Theorem]{Corollary}
\newtheorem{Lemma}[Theorem]{Lemma}
\newtheorem{Proposition}[Theorem]{Proposition}

 { \theoremstyle{definition}
\newtheorem{Definition}[Theorem]{Definition}

\newtheorem{Example}[Theorem]{Example}
\newtheorem{Remark}[Theorem]{Remark} }

\newcommand{\bbC}{\mathbb{C}}

\newcommand{\bbZ}{\mathbb{Z}}

\newcommand{\beq}{\begin{equation}}
\newcommand{\eeq}{\end{equation}}
\newcommand{\beqa}{\begin{eqnarray}}
\newcommand{\eeqa}{\end{eqnarray}}

\newcommand{\floor}[1]{\lfloor#1\rfloor}

\newcommand{\Nk}{{\mathsf N}}
\newcommand{\mm}{\mathsf{m}}
\newcommand{\nn}{\mathsf{n}}
\newcommand{\qg}{U_{\mathsf v}}

\newcommand{\ignore}[1]{}

\def\cb#1#2#3{\begin{bmatrix}#1\\#2\end{bmatrix}_{#3}}
\def\ctwo#1{\left({#1 \atop 2}\right)}

\def\SS{{\mathcal{H}_{\mathrm S}}}
\def\Bor{\mathcal{B}}
\def\dfrac#1#2{{\displaystyle\frac{#1}{#2}}}

\begin{document}

\allowdisplaybreaks

\newcommand{\arXivNumber}{2309.15364}

\renewcommand{\PaperNumber}{077}
	
\FirstPageHeading

\ShortArticleName{Non-Stationary Difference Equation and Affine Laumon Space II}

\ArticleName{Non-Stationary Difference Equation \\ and Affine Laumon Space II: \\
 Quantum Knizhnik--Zamolodchikov Equation}

\Author{Hidetoshi~AWATA~$^{\rm a}$, Koji~HASEGAWA~$^{\rm b}$, Hiroaki~KANNO~$^{\rm ac}$,
Ryo~OHKAWA~$^{\rm de}$, \newline
Shamil~SHAKIROV~$^{\rm fg}$, Jun'ichi~SHIRAISHI~$^{\rm h}$ and Yasuhiko~YAMADA~$^{\rm i}$}

\AuthorNameForHeading{H.~Awata et~al.}

\Address{$^{\rm a)}$~Graduate School of Mathematics, Nagoya University, Nagoya 464-8602, Japan}
\EmailD{\href{mailto:awata@math.nagoya-u.ac.jp}{awata@math.nagoya-u.ac.jp}}

\Address{$^{\rm b)}$~Mathematical Institute, Tohoku University, Sendai 980-8578, Japan}
\EmailD{\href{mailto:kojihas2@gmail.com}{kojihas2@gmail.com}}

\Address{$^{\rm c)}$~Kobayashi-Maskawa Institute,
Nagoya University, Nagoya 464-8602, Japan}
\EmailD{\href{mailto:kanno@math.nagoya-u.ac.jp}{kanno@math.nagoya-u.ac.jp}}

\Address{$^{\rm d)}$~Osaka Central Advanced Mathematical Institute, Osaka Metropolitan University,\\
\hphantom{$^{\rm d)}$}~Osaka 558-8585, Japan}
\EmailD{\href{mailto:ohkawa.ryo@omu.ac.jp}{ohkawa.ryo@omu.ac.jp}}

\Address{$^{\rm e)}$~Research Institute for Mathematical Sciences, Kyoto University, Kyoto 606-8502, Japan}
\EmailD{\href{mailto:ohkawa@kurims.kyoto-u.ac.jp}{ohkawa@kurims.kyoto-u.ac.jp}}

\Address{$^{\rm f)}$~University of Geneva, Switzerland}
\EmailD{\href{mailto:shakirov.work@gmail.com}{shakirov.work@gmail.com}}

\Address{$^{\rm g)}$~Institute for Information Transmission Problems, Moscow, Russia}

\Address{$^{\rm h)}$~Graduate School of Mathematical Sciences, University of Tokyo,\\
\hphantom{$^{\rm h)}$}~Komaba, Tokyo 153-8914, Japan}
\EmailD{\href{mailto:shiraish@ms.u-tokyo.ac.jp}{shiraish@ms.u-tokyo.ac.jp}}

\Address{$^{\rm i)}$~Department of Mathematics, Kobe University, Rokko, Kobe 657-8501, Japan}
\EmailD{\href{mailto:yamaday@math.kobe-u.ac.jp}{yamaday@math.kobe-u.ac.jp}}

\ArticleDates{Received November 06, 2023, in final form August 07, 2024; Published online August 22, 2024}

\Abstract{We show that Shakirov's non-stationary difference equation, when it is truncated, implies the quantum Knizhnik--Zamolodchikov ($q$-KZ) equation for $U_{\mathsf v}\bigl(A_1^{(1)}\bigr)$ with generic spins. Namely, we can tune mass parameters so that the Hamiltonian acts on the space of finite Laurent polynomials. Then the representation matrix of the Hamiltonian agrees with the $R$-matrix, or the quantum $6j$ symbols. On the other hand, we prove that the $K$ theoretic Nekrasov partition function from the affine Laumon space is identified with the well-studied Jackson integral solution to the $q$-KZ equation. Combining these results, we establish that the affine Laumon partition function gives a solution to Shakirov's equation, which was a conjecture in our previous paper. We also work out the base-fiber duality and four-dimensional limit in relation with the $q$-KZ equation.}


\Classification{14H70; 81R12; 81T40; 81T60}

\section{Introduction}

The quantum (or $q$-deformed) Knizhnik--Zamolodchikov ($q$-KZ) equation typically arises as a~linear difference equation
satisfied by matrix elements of the products of intertwining operators of highest weight representations
of quantum affine algebra \cite{EFK, FR-q-KZ}.
It roughly takes the following form:
\begin{align}
\Psi(u_1,\dots, pu_j,\dots, u_n)
&{}= R_{V_j V_{j-1}} \biggl(\frac{pu_j}{u_{j-1}}\biggr) \cdots R_{V_j V_{1}} \biggl(\frac{pu_j}{u_{1}}\biggr)
D_j \nonumber\\
&\quad{}\times
R_{V_n V_j}^{-1}\biggl(\frac{u_n}{u_{j}}\biggr) \cdots R_{V_{j+1} V_j}^{-1}\biggl(\frac{u_{j+1}}{u_{j}}\biggr) \Psi(u_1,\dots, u_n),
\end{align}
where $\Psi(u_1,\dots, u_n)$ is an appropriate matrix element of the intertwiners with the spectral parameters $u_i$
and $R_{V_k V_\ell}$ denotes the $R$ matrix on the tensor product $V_k \otimes V_\ell$ of highest weight representations.
$D_j$ 
is a diagonal operator acting on $V_j$. 
In the case of \smash{$\qg\bigl(A^{(1)}\bigr)$} with level $k$, the shift parameter is fixed to be $p={\sf v}^{2(k+2)} = q^{k+2}$.\footnote{
The convention used in our papers \cite{Awata:2022idl,Shakirov:2021krl} is different from
the standard convention for the deformation parameter of the quantum affine algebra.
This is the reason why we denote $\qg\bigl(A_1^{(1)}\bigr)$ with ${\sf v}^2 = q$.
According to the AGT dictionary for the current algebra \cite{Alday:2010vg},
$k+2 = -\epsilon_2/\epsilon_1$, where
$(\epsilon_1, \epsilon_2)$ are the $\Omega$-background parameters on the gauge theory side.
Hence, we have $p=e^{(k+2)\epsilon_1} = e^{-\epsilon_2} = t$.}
The case $n=2$ is of our interest in this paper.
By the scaling of the spectral parameters, which fixes the origin and the infinity,
we can consider the two point function $\Phi(u) := \Psi(1,u)$, for which the $q$-KZ equation is
\begin{equation}\label{eq:qKZ-2pt}
%
\Phi(pu) = R_{V_2 V_{1}} (pu) D_2 \Phi(u).
\end{equation}

For the $q$-difference equations including \eqref{eq:qKZ-2pt},
various important results using the Jackson integrals have been developed based on the fundamental theory by K.~Aomoto.
In particular, the Jackson integral of symmetric Selberg type relevant for the equation \eqref{eq:qKZ-2pt}
was thoroughly studied by Aomoto, Kato, Matsuo, Mimachi and Ito
around early 1990's (see, e.g., \cite{AK,MIto} and references therein).
Around the same time, the relation to Bethe vectors was found by Reshetikhin \cite{Reshe-Bethe} and the method
was further developed by Tarasov, Varchenko, etc.
These works play a crucial role in our study.

In \cite{Awata:2022idl}, we show that the non-stationary difference equation proposed in \cite{Shakirov:2021krl}
is transformed into a quantization of the discrete Painlev\'e VI equation.
The original form of the non-stationary difference equation is
\begin{align}\label{Shamil-Eq}
\mathsf{U} (t \Lambda,x)=\mathcal{A}_1(\Lambda,x) \cdot \Bor \cdot
\mathcal{A}_2(\Lambda,x) \cdot \Bor \cdot \mathcal{A}_3 (\Lambda,x) \mathsf{U}\biggl(\Lambda,{x\over t q Q}\biggr),
\end{align}
where $\Bor$ is the $q$-Borel transformation on a formal Laurent series in $x$:
\begin{equation}\label{q-Borel}
\Bor\biggl(\sum_{n} c_n x^n\biggr)=\sum_n q^{n(n+1)/2} c_n x^n,
\end{equation}
and $\mathcal{A}_i(\Lambda,x)$ are products of infinite products $(x;q)_\infty$
and $(x;q,t)_\infty$ defined by \eqref{inf-prod} and \eqref{double-inf-prod}, respectively.
(See \cite{Shakirov:2021krl} and \cite{Awata:2022idl} for explicit forms of $\mathcal{A}_i(\Lambda,x)$).
In Appendix~\ref{App:framing}, we explain a~relation of the operator $\Bor$ and the refined Chern--Simons theory \cite{aganagic2011knot},
which gave the original motivation for introducing $\Bor$ in \cite{Shakirov:2021krl}.
The non-stationary difference equation \eqref{Shamil-Eq} was found by looking at the instanton partition function of
five-dimensional ${\rm SU}(2)$ gauge theory with a~surface defect.
The parameters $(q,t)$ are equivariant parameters of the torus action $(z_1, z_2) \longrightarrow \bigl(q z_1, t^{-1} z_2\bigr)$
on $\mathbb{C}^2$ and the defect is at $z_2=0$, which is the set of fixed points
of $t^{-1}$-action.\footnote{We can regard $q$ and $t$ as the $\Omega$-background along the defect and transverse to the defect, respectively.}
The function $\mathsf{U} (\Lambda,x)$ is supposed to be a formal double power series in $x$ and $\Lambda/x$:
\begin{equation}\label{intro:expand}
\mathsf{U} (\Lambda,x) = \sum_{k,\ell \geq 0} c_{k,\ell} x^k (\Lambda/x)^\ell,
\end{equation}
where the coefficients $c_{m,n}$ are functions of the ${\rm SU}(2)$ Coulomb parameter $Q$, mass parameters~$T_i$
and the equivariant parameters $(q,t)$.
Physically, $\Lambda$ is the instanton expansion parameter and~$x$ is identified with the position
of the degenerate field $\phi_{2,1}(x)$ that corresponds to the defect according
to the AGT correspondence \cite{Alday:2009fs,Alday:2009aq}.

In the case of the Virasoro algebra which corresponds to four-dimensional $\mathcal{N}=2$ supersymmetric gauge theories,
the degenerate field $\phi_{2,1}(x)$ has a null state at level two, which implies
a~differential equation of the second order in $x$ as the BPZ equation.
The non-stationary difference equation \eqref{Shamil-Eq} may be regarded as a $q$-deformed version of the BPZ equation.
However, if we try to derive it, we encounter the problem, since the natural coproduct is not known for the deformed Virasoro algebra~\cite{Awata:2009ur,Shiraishi:1995rp}.
Hence, we cannot define intertwiners and consequently cannot
use the methods of representation theory. We can overcome such a difficulty by noticing the following fact.
Namely, the instanton partition function with a surface defect allows another description
in terms of the affine Laumon space \cite{Alday:2010vg,Laumon1,Laumon2}.
In this case the instanton partition function gives a conformal block of the affine Kac--Moody algebra and
its $q$ deformation, the quantum affine algebra is a Hopf algebra!
In \cite{Awata:2022idl} (and implicitly in \cite{Shakirov:2021krl}),
it was conjectured that the instanton partition function coming from the equivariant character of
the affine Laumon space solves the gauge transformed version of \eqref{Shamil-Eq};
See \eqref{to-qKZ} below.
More concretely, in \cite{Awata:2022idl}
we have shown that the non-stationary difference equation \eqref{Shamil-Eq} can be factorized into a coupled system,
which is in accord with the formula of Hamiltonian of the discrete Painlev\'e VI equation
in terms of the extended affine Weyl group of \smash{$D_5^{(1)}$}.
Then the conjecture is that a pair of the affine Laumon partition functions whose parameters are related
by a simple transformation provides a solution to the coupled system, see \cite[Conjecture 6.4]{Awata:2022idl}.
In this paper, we prove this conjecture by using a connection through the $q$-KZ equation.

In order to see the relation to the $q$-KZ equation, we first make a gauge transformation
of the ``wave function'' $\mathsf{U} (\Lambda,x) \to \Psi(\Lambda,x)$ to recast \eqref{Shamil-Eq}
into the following form \cite{Awata:2022idl}:
\begin{equation}\label{to-qKZ}
\Psi(\Lambda,x) = \SS T_{qtQ,x}^{-1} T_{t,\Lambda}^{-1} \cdot \Psi(\Lambda,x).
\end{equation}
An explicit form of the Hamiltonian $\SS$ is given by \eqref{Shakirov}.
$\SS$ involves ``renormalized'' mass parameters $d_i$, $1 \leq i \leq 4$ which are monomials in the original mass parameters $T_i$, $Q$ and
$(q,t)$. It is amusing that we can eliminate the double infinite product $(x;q,t)_\infty$ appeared in $\mathcal{A}_i(\Lambda,x)$ from $\SS$.
It is also remarkable that the parameters $Q$ and $t$ appear explicitly only through the shift parameters for $x$ and $\Lambda$.
By tuning two of mass parameters, say $d_2=q^{-m}$ and $d_3=q^{-n}$, the Hamiltonian acts on the $(n+m+1)$-dimensional space of
Laurent polynomials with a basis~$x^j$, $-n \leq j \leq m$.
Then, the representation matrix of $\SS$ agrees with the $R$-matrix of \smash{$\qg\bigl(A_1^{(1)}\bigr)$},
which allows us to identify \eqref{to-qKZ} as a $q$-KZ equation. Namely, for the matrix $r_{i,j}$ defined by
\begin{equation}
\SS x^{i}=\sum_{j=-n}^{m} r_{i,j}(\Lambda) x^j, \qquad
-n\leq i\leq m,
\end{equation}
we have the following statement.

\begin{Theorem}[Theorem \ref{r-matrix}]
Set $N=m+n$. The matrix $[r_{i,j}(\Lambda)]_{-n\leq i,j\leq m}$ coincides with
the $(N+1) \times (N+1)$ block
in the weight decomposition of the infinite-dimensional $R$-matrix for \smash{\raisebox{-0.5pt}{$\qg\bigl(A_1^{(1)}\bigr)$}}
for evaluation representations of generic weights.
\end{Theorem}
This is one of the fundamental results in this paper.

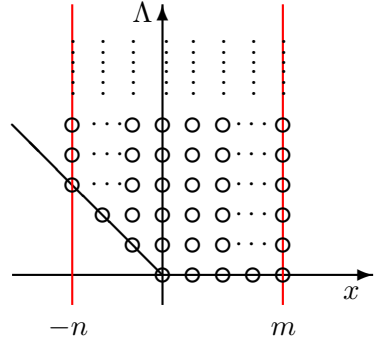
\begin{figure}[t]\centering
\vspace*{23mm}
\begin{picture}(50,40)
\setlength{\unitlength}{0.8mm}
\thicklines
\put(-25,0){\vector(1,0){60}}
\put(0,-5){\vector(0,1){50}}
\put(0,0){\line(-1,1){25}}
\put(-15,-5){{\color{red}\line(0,1){50}}}
\put(20,-5){{\color{red}\line(0,1){50}}}
\put(0,0){\circle{2}}
\put(0,5){\circle{2}}
\put(0,10){\circle{2}}
\put(0,15){\circle{2}}
\put(0,20){\circle{2}}
\put(0,25){\circle{2}}
\put(5,0){\circle{2}}
\put(5,5){\circle{2}}
\put(5,10){\circle{2}}
\put(5,15){\circle{2}}
\put(5,20){\circle{2}}
\put(5,25){\circle{2}}
\put(10,0){\circle{2}}
\put(10,5){\circle{2}}
\put(10,10){\circle{2}}
\put(10,15){\circle{2}}
\put(10,20){\circle{2}}
\put(10,25){\circle{2}}
\put(15,0){\circle{2}}
\put(12,4){$\cdots$}
\put(12,9){$\cdots$}
\put(12,14){$\cdots$}
\put(12,19){$\cdots$}
\put(12,24){$\cdots$}
\put(20,0){\circle{2}}
\put(20,5){\circle{2}}
\put(20,10){\circle{2}}
\put(20,15){\circle{2}}
\put(20,20){\circle{2}}
\put(20,25){\circle{2}}
\put(-5,5){\circle{2}}
\put(-5,10){\circle{2}}
\put(-5,15){\circle{2}}
\put(-5,20){\circle{2}}
\put(-5,25){\circle{2}}
\put(-10,10){\circle{2}}
\put(-12,14){$\cdots$}
\put(-12,19){$\cdots$}
\put(-12,24){$\cdots$}
\put(-15,15){\circle{2}}
\put(-15,20){\circle{2}}
\put(-15,25){\circle{2}}
\put(0,30){$\vdots$}
\put(4.5,30){$\vdots$}
\put(9.5,30){$\vdots$}
\put(14.5,30){$\vdots$}
\put(19.5,30){$\vdots$}
\put(-5.5,30){$\vdots$}
\put(-10.5,30){$\vdots$}
\put(-15.5,30){$\vdots$}
\put(0,35){$\vdots$}
\put(4.5,35){$\vdots$}
\put(9.5,35){$\vdots$}
\put(14.5,35){$\vdots$}
\put(19.5,35){$\vdots$}
\put(-5.5,35){$\vdots$}
\put(-10.5,35){$\vdots$}
\put(-15.5,35){$\vdots$}

\put(-5,42){$\Lambda$}
\put(30,-4){$x$}
\put(18,-10){$m$}
\put(-19,-10){$-n$}
\end{picture}
\vspace{5mm}

\caption{After the mass truncation $d_2=q^{-m}$, $d_3=q^{-n}$, the double series $\Psi(\Lambda,x)$ becomes
a Laurent polynomial in $x$, while it is still a formal power series in~$\Lambda$.
The circles represent the positions of allowed terms in the $(x, \Lambda)$-lattice.}\label{Fig:mass-truncation}
\end{figure}

As we mentioned above, the Jackson integral of symmetric Selberg type solves the $q$-KZ equation.
Recall that the Jackson integral is a pairing of the integration cycle $\xi$ and the cocycle function $\phi(z)$
\begin{equation}\label{intro;Jackson}
\langle \phi(z), \xi \rangle
=(1-t)^N \sum_{\nu \in \bbZ^N} \biggl[\phi(z)\Phi(z) \prod_{1\leq i< j \leq N} (z_i-z_j) \biggr]_{z_i=\xi_i t^{\nu_i}}.
\end{equation}
In this paper, we take the weight function $\Phi(z)$ which depends on four parameters $(a_1,a_2, b_1, b_2)$;
See Section \ref{sec:duality} for details.
On the other hand, the instanton partition function $\mathcal{Z}_{\rm {AL}}$
from the equivariant character of the affine Laumon space is defined by \eqref{fixedpoint}.
Hence, if we can show the relation of the partition function $\mathcal{Z}_{\rm {AL}}$
to the Jackson integral \eqref {intro;Jackson}, it gives a proof of the conjecture in our previous paper.
Under the same tuning conditions for mass parameters as above, the partition function $\mathcal{Z}_{\rm {AL}}$
becomes a Laurent polynomial in $x$ (see Figure~\ref{Fig:mass-truncation} and Appendix~\ref{App:Tuning}).
By carefully examining combinatorial structure of the orbifolded Nekrasov factor in~\eqref{fixedpoint} (see Appendix~\ref{App:Laumon}),
we can identify the coefficients of $\mathcal{Z}_{\rm {AL}}$ in the expansion parameter $x$
as the Jackson integral \eqref{intro;Jackson} with an appropriate choice of the cycle $\xi$.
In this way, the affine Laumon partition function $\mathcal{Z}_{\rm {AL}}$
agrees with the Jackson integral for the $q$-KZ equation term by term,
if we choose suitable basis of (co-)cycles.
More concretely, we have the following.

\begin{Theorem}[Theorem \ref{AL=J}]
Under the identification of the parameters
\begin{gather}
d_1=\frac{1}{q^{m-1}a_1b_1}, \qquad
d_2=q^{-m}, \qquad
d_3=q^{-n}, \qquad
d_4=\frac{1}{q^{n-1}a_2b_2}, \qquad
Q=\frac{q^{m-n}a_1}{t a_2},\!\!
\end{gather}
the affine Laumon partition function $\mathcal{Z}_{\rm {AL}}$ coincides with the Jackson integral \eqref{intro;Jackson}
with an appropriate choice of $\xi$ and a suitable basis of $\phi(z)$.
\end{Theorem}
 A definition of the basis and the choice of $\xi$ are the subjects of Section~\ref{sec:duality}.
In particular, we show that the combinatorial structure of $\mathcal{Z}_{\rm {AL}}$ as a sum over the pair of Young diagrams
is in a~nice correspondence with the cohomology basis called Matsuo basis in the literature.

By combining above two theorems under the specialization $d_2=q^{-m}$, $d_3=q^{-n}$,
we can show that the affine Laumon partition function $\mathcal{Z}_{\rm {AL}}$
is a solution to the gauge transformed version~\eqref{to-qKZ} of the Shakirov's equation.
The result can be extended to generic mass parameters $d_i$ also by the following reasons.
We invoke the fact that the equation \eqref{Shamil-Eq} inductively determines the expansion coefficients $c_{i,j}$
in \eqref{intro:expand} uniquely up to an overall normalization, which we fix by~${c_{0,0}=1}$
(see \cite[Proposition 2.7]{Awata:2022idl}).
Moreover, from the recursion relations among $c_{i,j}$ we can see that for generic mass parameters $d_i$
they are polynomials in $d_i$ with coefficients in the field of rational functions of $q$, $t$ and $Q$.\footnote{
If we expand the Hamiltonian $\SS$ in $x$ and $\Lambda/x$, the leading term does not involve $d_i$.}
Hence, if $\mathcal{Z}_{\rm {AL}}$ is a solution under the specialization $d_2 = q^{-m}$, $d_3 = q^{-n}$
for any $m,n \in \mathbb{Z}_{\geq 0}$, it is valid for generic values of $d_2$ and $d_3$.

\begin{Theorem}\label{Main}
The affine Laumon partition function $\mathcal{Z}_{\rm {AL}}$ is a solution to the non-stationary difference equation
\eqref{to-qKZ}.
\end{Theorem}

In Appendix~\ref{App:coupled}, we recast \eqref{to-qKZ} to a coupled system, which is essentially the same as
the coupled system derived in \cite{Awata:2022idl} based on the structure of the Hamiltonian of the $qq$-Painlev\'e~VI equation.
Hence, our main Theorem~\ref{Main} also proves \cite[Conjecture~6.4]{Awata:2022idl}.
Recall that the original conjecture in \cite{Shakirov:2021krl} was for the degenerate five point $q$-Virasoro conformal block,
which is realized as the Nekrasov partition function of ${\rm U}(2) \times {\rm U}(2)$ linear quiver gauge
theory with the Higgsing condition on mass parameters.
One can express the partition function in question in terms of the topological vertex (intertwiners of the quantum
toroidal $\mathfrak{gl}_1$-algebra) and confirm that it is similarly related to the Jackson integral of the same type
as the present paper. This is regarded as a $q$-deformed analogue of
the correspondence between the current block of $\mathfrak{sl}_2$ and the degenerate conformal block.
Details will be reported elsewhere.

It is known that the Jackson integral relevant for $q$-KZ equation actually satisfies two kinds of difference equations
(see, e.g., \cite{MIto}).
For the case in our hand, the second difference equation involves the shift of a ``renormalized''
Coulomb modulus $(qtQ)^{-1}$ of ${\rm SU}(2)$ gauge theory.\footnote{Recall that $(qtQ)^{-1}$ is the shift parameter of $x$ in \eqref{Shamil-Eq}.}
Two difference equations are exactly the same under the exchange
\begin{equation}\label{base-fiber-duality}
\Lambda \longleftrightarrow \frac{q^2}{d_1d_2} (qtQ)^{-1}.
\end{equation}
In the geometric engineering of ${\rm SU}(2)$ gauge theory by local $\mathbb{F}_0 = \mathbb{P}^1 \times \mathbb{P}^1$, $\Lambda$ and $Q$ are identified with
K\"ahler parameters of $\mathbb{P}^1 \times \mathbb{P}^1$. Hence, this is a manifestation of the base-fiber duality in topological strings.

We will also work out the four-dimensional limit of the gauge transformed equation \eqref{to-qKZ}.
Then we can check that it agrees with the KZ equation, which has been derived from the regularity of the fractional fundamental
$qq$-character of type $A_1$ \cite{Nekrasov:2021tik}.

Since the KZ equation can be regarded as a quantization of the Schlesinger system \cite{Reshe},
the relation of the conformal block of the current algebra to the quantum monodromy preserving deformation
has been believed as a folklore.
The equation \eqref{to-qKZ} is quite interesting object as an explicit example realizing such correspondence.
In recent papers \cite{CV,TV}, the quantum cohomology and (quantum) monodromy preserving deformation
has also been discussed.

It is very likely that the $q$-Borel transformation $\Bor$ appears
in some representation of double affine Hecke algebras (DAHA) on polynomials
in $x_1, x_2, \dots$ \cite{Shakirov:2021krl}. From such a speculation, it is interesting
that in connection with the quantum $Q$ system \cite{DFK1,DFK2},
an operator very close to $\Bor$ is employed as Cherednik's Gaussian operator
representing the Dehn twist on polynomial representations of spherical DAHA.

The present paper is organized as follows.
In Section~\ref{sec:q-KZ}, we introduce a truncation of the non-stationary difference equation \eqref{Shamil-Eq} by tuning two of mass parameters among four.
We will call it mass truncation. After the mass truncation the Hamiltonian of \eqref{Shamil-Eq} acts on the space of Laurent polynomials in $x$.
We find that the matrix elements $r_{i,j}(\Lambda)$ satisfy
the same relation as the quantum $6j$-symbol derived by Rosengren \cite{Rosengren:2003nee}.
In Section~\ref{sec:duality}, following the work by Ito~\cite{MIto}, we first review the $q$-KZ equations satisfied by the Jackson integral \eqref{intro;Jackson},
which is a paring of an integration cycle~$\xi$ and a cocycle function $\varphi(z)$. We also study a special integration cycle~$\xi$
which reduces the bilateral sum over the lattice in the definition of the Jackson integral to a~sum over a positive cone of the lattice.
This truncation (which we call lattice truncation) has been considered in literature (see, e.g., \cite{ItoForrester,ItoNoumi} and references therein).
The relevant cycles are called the characteristic cycles or the $\alpha$-stable or $\alpha$-unstable cycles by Aomoto, and play crucial roles in the study of asymptotic behavior.
The lattice truncation is crucial, when we make an identification of the Jackson integral with the Nekrasov partition function
which is expressed as a sum over a pair of Young diagrams.
In Section~\ref{sec:Matsuo}, we consider the Nekrasov partition function with a surface defect, which can be obtained from the equivariant character
at fixed points of the toric action on the affine Laumon space. The partition function is expanded in $x$ and $\Lambda/x$ and we show that
the expansion in $x$ is in accord with a basis, which is called Matsuo basis, of the cocycle functions in the Jackson integral representation of
solutions to the $q$-KZ equation. In~Section~\ref{sec:4d}, we compute the four-dimensional limit of the $R$-matrix that appears in the truncation
of the non-stationary difference equation and show that the limit correctly reproduces the KZ equation for
the conformal block of the $\mathfrak{sl}_2$ current algebra.
There are many appendices in this paper, where 
some helpful clarifications and technical details are presented. 

We will use the following notations throughout the paper \cite{hypergeometric}:
\begin{gather}
\label{inf-prod}
\varphi(x):= (x;q)_\infty = \prod_{n=0}^\infty (1-x q^n)
=\exp\Biggl(-\sum_{n=1}^\infty{1\over n} {1\over 1-q^n} x^n \Biggr), \qquad |x| < 1, \quad |q|<1, \\
(x;q,t)_\infty = \prod_{n,m=0}^\infty (1-x q^nt^m)
=\exp\Biggl(-\sum_{n=1}^\infty{1\over n} {1\over (1-q^n)(1-t^n)} x^n \Biggr), \nonumber \\
 \hphantom{(x;q,t)_\infty =}{} |x| <1, \quad |q|, |t| < 1 .\label{double-inf-prod}
\end{gather}
The $q$-shifted factorial is defined by
\begin{equation}
(a;q)_n = \frac{(a;q)_\infty}{(aq^n ; q)_\infty} = \prod_{i=0}^{n-1} \bigl(1-a q^i\bigr).
\end{equation}
It is convenient to introduce the notation
\begin{equation}
(a_1, a_2, \dots, a_k ;q)_n := (a_1;q)_n (a_2;q)_n \cdots (a_k;q)_n.
\end{equation}
We frequently use the short-hand notation $(a)_n:=(a;q)_n$, when there is no risk of confusion.
The following formulas are useful:
\begin{equation}\label{a_n-inversion}
(a)_n=(-a)^n q^{n(n-1)/2} \biggl(\frac{1}{a q^{n-1}}\biggr)_n, \qquad
(a)_{k+\ell}=(a)_k\bigl(a q^k\bigr)_{\ell}.
\end{equation}
The $q$-exponential function is
\begin{equation}
e_q(x) := \sum_{n=0}^\infty \frac{x^n}{(q;q)_n} = \frac{1}{(x;q)_\infty}.
\end{equation}
Finally, the $q$-binomial coefficient is defined by
\begin{equation}\label{q-binom}
\cb{n}{k}{q}:= \frac{(q;q)_n}{(q;q)_{k} (q;q)_{n-k}}.
\end{equation}

\section[Shakirov's equation as a q-KZ equation with generic spins]{Shakirov's equation as a $\boldsymbol{q}$-KZ equation with generic spins}\label{sec:q-KZ}

In \cite{Awata:2022idl}, we show that by an appropriate gauge transformation we can recast the original equation~\eqref{Shamil-Eq}
to the form $\SS T_{qtQ,x}^{-1}T_{t,\Lambda}^{-1} \cdot \Psi(\Lambda,x)=\Psi(\Lambda,x)$ with the Hamiltonian
\begin{align}
\SS&{} = \frac{1}{\varphi(qx)\varphi(\Lambda/x)} \cdot \Bor \cdot
\frac{\varphi(\Lambda)\varphi\bigl(q^{-1} d_1d_2d_3d_4\Lambda\bigr)}{\varphi(-d_1x)\varphi(-d_2x)\varphi(-d_3 \Lambda/x)\varphi(-d_4 \Lambda/x)}
\nonumber \\
&\quad{}\times \Bor\cdot \frac{1}{\varphi\bigl(q^{-1}d_1d_2x\bigr)\varphi(d_3d_4 \Lambda/x)},\label{Shakirov}
\end{align}
where $\varphi(z) := (z;q)_\infty$.
We have made the following change of variables:\footnote{In \eqref{Shakirov}, $'$ on $x$ and $\Lambda$ are deleted for simplicity.}
\begin{equation}\label{d-variables1}
\Lambda' = t T_2T_3 \Lambda, \qquad x'= T_2 q^{-1/2} t^{1/2} x,
\end{equation}
and mass parameters:\footnote{We made an exchange $d_1\leftrightarrow d_2$
in \cite[equation~(5.5)]{Awata:2022idl}. We also exchanged $T_1$ and $T_2$, which corresponds to the action of a Weyl reflection
on the Painlev\'e side, namely \eqref{Shakirov} is derived from \cite[equation~(5.2)]{Awata:2022idl}
by the change of variables \eqref{d-variables1}
and \eqref{d-variables2}.}
\begin{alignat}{3}
&d_1 = T_1 q^{1/2} t^{-1/2}, \qquad&&
d_2 = T_2^{-1} q^{1/2} t^{-1/2} Q^{-1},& \nonumber\\
&d_3 = T_3^{-1} q^{1/2} t^{-1/2}, \qquad&&
d_4 = T_4 q^{1/2} t^{1/2} Q.&\label{d-variables2}
\end{alignat}
The Hamiltonian $\SS$ is manifestly symmetric under the exchange $d_1\leftrightarrow d_2$, $d_3 \leftrightarrow d_4$.
Let $\vartheta_x := x \partial_x$ be the Euler operator in $x$.
We have the following commutation relations among $x, p=q^{\vartheta_x}$ and $\Bor = q^{\vartheta_x(\vartheta_x+1)/2}$:
\begin{equation}
px = qxp, \qquad \Bor p = p \Bor, \qquad \Bor x = px \Bor.
\end{equation}

\begin{Lemma}\label{truncation}
For $m \in \bbZ_{\geq 0}$, we have
\begin{gather}
\SS \cdot x^m \bbC\biggl[\biggl[\frac{1}{x}\biggr]\biggr] \subset x^m \bbC\biggl[\biggl[\frac{1}{x}\biggr]\biggr] \qquad{\rm for} \ d_1=q^{-m}\ {\rm or} \ d_2=q^{-m},\notag \\
\SS \cdot x^{-m} \bbC[[x]] \subset x^{-m}\bbC[[x]] \qquad{\rm for} \ d_3=q^{-m}\ {\rm or} \ d_4=q^{-m}.\label{eq:HS-trunc}
\end{gather}
\end{Lemma}
\begin{proof}
Due to the relation 
\[
\SS(d_1, d_2, d_3, d_4, \Lambda, x) = \ q^2 x \ \SS(q d_1, q d_2, d_3/q , d_4/q,\Lambda, x)\ p^2 x^{-1},
\]
the statements inductively reduce to the case $m=0$.
Hence, it is enough to show the lemma for $m=0$. Note that one can omit the factors of the form
$\varphi(\ast \Lambda/x)$ [$\varphi(\ast x)$ (resp.)] in the first [second (resp.)] relation in \eqref{eq:HS-trunc},
since they do not affect the consequence.
By omitting these factors, the relations \eqref{eq:HS-trunc} follow from the identities ($k\geq 0$)
\begin{gather*}
\frac{1}{\varphi(-aq x)}\Bor \frac{1}{\varphi(a x)} \cdot x^{-k}=x^{-k} \prod_{j=0}^{k-1}\bigl(q^j+ax\bigr),\\
\frac{1}{\varphi( -a/x)} \Bor \frac{1}{\varphi(a/x)} \cdot x^k=x^k \prod_{j=1}^k\biggl(q^j+\frac{a}{x}\biggr),
\end{gather*}
which are a special case of the formulas \eqref{lem-1} in Appendix \ref{App:Proofs}.
\end{proof}

Due to Lemma \ref{truncation}, we have the following action of $\SS$ with finite support:
\begin{equation}\label{eq:SSr}
\SS x^{i}=\sum_{j=-n}^{m} r_{i,j}(\Lambda) x^j, \qquad
-n\leq i\leq m,
\end{equation}
when the parameters are specialized as $d_2=q^{-m}$, $d_3=q^{-n}$, $m,n \in \bbZ_{\geq 0}$.\footnote
{Due to the symmetry $d_1\leftrightarrow d_2$, $d_3 \leftrightarrow d_4$, the cases $d_1=q^{-m}$, $d_4=q^{-n}$ etc.
have the similar truncation.}
The coefficient~$r_{i,j}(\Lambda)$ is a rational function in the spectral parameter $\Lambda$.
It also depends on the remaining mass parameters $d_1$ and $d_4$, which are to be identified with
the highest weights of the in-state and the out-state, when we consider a matrix element (two point function)
of the intertwiners.

\begin{Example}
For $(m,n)=(1,0)$, $(2,0)$, we have
\begin{equation}
\begin{bmatrix}
r_{0,0}&r_{0,1}\\
r_{1,0}&r_{1,1}
\end{bmatrix}=
\begin{bmatrix}
\frac{(\frac{d_1 \Lambda }{q})_1 }{(\frac{\Lambda }{q})_1 } & -\frac{(d_1)_1 }{(\frac{\Lambda }{q})_1 } \vspace{1mm}\\
-\frac{\Lambda q (\frac{d_4}{q})_1 }{(\frac{\Lambda }{q})_1 } & \frac{q^2 (\frac{d_4 \Lambda }{q^2})_1 }{(\frac{\Lambda }{q})_1 }\\
\end{bmatrix}
\end{equation}
and
\begin{equation}
\begin{bmatrix}
r_{0,0}&r_{0,1}&r_{0,2}\\
r_{1,0}&r_{1,1}&r_{1,2}\\
r_{2,0}&r_{2,1}&r_{2,2}
\end{bmatrix}
=
\begin{bmatrix}
\frac{(\frac{d_1 \Lambda }{q^2})_1 (\frac{d_1 \Lambda }{q})_1 }{(\frac{\Lambda }{q^2})_1 (\frac{\Lambda }{q})_1}
& -\frac{(1+q) (d_1)_1 (\frac{d_1 \Lambda }{q})_1}{q (\frac{\Lambda }{q^2})_1 (\frac{\Lambda }{q})_1}
& \frac{(d_1)_1 (d_1 q)_1 }{q (\frac{\Lambda }{q^2})_1 (\frac{\Lambda }{q})_1 } \vspace{1mm}\\
-\frac{\Lambda q (\frac{d_4}{q})_1 (\frac{d_1 \Lambda }{q})_1 }{(\frac{\Lambda }{q^2})_1 (\frac{\Lambda }{q})_1}
& \frac{q^2 (\frac{d_1 \Lambda }{q})_1 (\frac{d_4 \Lambda}{q^2})_1 +\Lambda q (d_1 q)_1
(\frac{d_4}{q^2})_1}{(\frac{\Lambda}{q^2})_1 (\frac{\Lambda}{q})_1}
& -\frac{q^2 (d_1 q)_1 (\frac{d_4 \Lambda}{q^3})_1 }{(\frac{\Lambda }{q^2})_1 (\frac{\Lambda }{q})_1} \vspace{1mm}\\
 \frac{\Lambda^2 q^3 (\frac{d_4}{q^2})_1 (\frac{d_4}{q})_1 }{(\frac{\Lambda }{q^2})_1 (\frac{\Lambda }{q})_1 }
& -\frac{\Lambda q^4 (1+q) (\frac{d_4}{q^2})_1 (\frac{d_4 \Lambda }{q^3})_1 }{(\frac{\Lambda }{q^2})_1 (\frac{\Lambda }{q})_1}
& \frac{q^6 (\frac{d_4 \Lambda }{q^4})_1 (\frac{d_4 \Lambda }{q^3})_1 }{(\frac{\Lambda }{q^2})_1 (\frac{\Lambda }{q})_1}
\end{bmatrix},
\end{equation}
where $r_{i,j}=r_{i,j}(\Lambda)$, $(x)_1=(1-x)$.
\end{Example}

\begin{Theorem}\label{r-matrix}
Set $N=m+n$. The matrix $[r_{i,j}(\Lambda)]_{-n\leq i,j\leq m}$ coincides with
the $(N+1) \times (N+1)$ block
in the weight decomposition of the infinite-dimensional $R$-matrix for $\qg\bigl(A_1^{(1)}\bigr)$
for evaluation representations of generic weights {\rm \cite{Bosnjak:2016oze,Mangazeev:2014gwa}}.\footnote{See Remark \ref{explicitR} below and computations in Section~\ref{subsec:r-matrix} for more precise identification.}
\end{Theorem}

\begin{proof}
The defining relation \eqref{eq:SSr} of $r_{i,j}(\Lambda)$ can be written as
\begin{gather}
\frac{\varphi(\Lambda)\varphi\bigl(q^{-1} d_1d_2d_3d_4\Lambda\bigr)}{\varphi(-d_1x)\varphi(-d_2x)\varphi(-d_3 \Lambda/x)\varphi(-d_4 \Lambda/x)}
\Bor \frac{1}{\varphi(q^{-1}d_1d_2x)\varphi(d_3d_4 \Lambda/x)} x^i \nonumber\\
\qquad{}=
\Bor^{-1} \varphi(qx)\varphi\biggl(\frac{\Lambda}{x}\biggr)
\sum_{j=-n}^{m} r_{i,j}(\Lambda) x^j .\label{eq:SSr2}
\end{gather}
Using the relations (see Appendix \ref{App:Proofs})
\begin{gather}
\Bor \frac{1}{\varphi\bigl(q^{-1}d_1d_2 x\bigr)\varphi(d_3d_4\frac{\Lambda}{x})}x^i=
\frac{\varphi\bigl(-q^i d_1d_2 x\bigr)\varphi\bigl(-q^{-i}d_3d_4\frac{\Lambda}{x}\bigr)}{\varphi\bigl(q^{-1}d_1d_2d_3d_4 \Lambda\bigr)}q^{i(i+1)/2}x^i, \nonumber\\
\Bor^{-1} \varphi(qx)\varphi\biggl(\frac{\Lambda}{x}\biggr) x^j= q^{-j(j+1)/2}x^j
\frac{\varphi(\Lambda)}{\varphi\bigl(-q^{-j}x\bigr)\varphi\bigl(-q^j \frac{\Lambda}{x}\bigr)},
\end{gather}
we can rewrite \eqref{eq:SSr2} as
\begin{gather}
\frac{\varphi\bigl(-q^i d_1d_2 x\bigr)\varphi\bigl(-q^{-i}d_3d_4\frac{\Lambda}{x}\bigr)}{\varphi(-d_3 \Lambda/x)\varphi(-d_4 \Lambda/x)}q^{i(i+1)/2}x^i \nonumber\\
\qquad{}=\sum_{j=-n}^m r_{i,j}(\Lambda) q^{-j(j+1)/2}x^j
\frac{\varphi(-d_1x)\varphi(-d_2x)}{\varphi\bigl(-q^{-j}x\bigr)\varphi\bigl(-q^j \frac{\Lambda}{x}\bigr)}.
\end{gather}
Since $d_2=q^{-m}$, $d_3=q^{-n}$, $m,n \in \bbZ_{\geq 0}$,
this equation can be simplified to
\begin{gather}
 q^{\frac{1}{2} i (i+1)} x^i \bigl(-d_1 q^{i-m}x\bigr)_{m-i} \biggl(-d_4 q^{-i-n}\frac{\Lambda}{x}\biggr)_{i+n} \nonumber \\
\qquad{}= \sum_{j=-n}^m r_{i,j}(\Lambda)
q^{-\frac{1}{2} j (j+1)} x^j (-q^{-m}x)_{m-j} \biggl(-q^{-n}\frac{\Lambda}{x}\biggr)_{j+n},\label{eq:rdef}
\end{gather}
where $(a)_n=(a;q)_n$.
This relation is essentially the same as the equation for the $q$-$6j$ symbols~\cite{Rosengren:2003nee}.
%
Hence, the coefficients $r_{i,j}(\Lambda)$ is given by the $q$-$6j$ (i.e., $R$-matrix) for \smash{\raisebox{-0.5pt}{$\qg\bigl(A_1^{(1)}\bigr)$}}.
\end{proof}

\begin{Remark}\label{explicitR}
Some explicit formulas for $\qg\bigl(A_1^{(1)}\bigr)$ $R$-matrix are known \cite{Bosnjak:2016oze,Mangazeev:2014gwa}.
One of such expressions is related to the hypergeometric series ${}_4\phi_3$
with the base $q$,\footnote{The parameter $q$ in \cite{Bosnjak:2016oze,Mangazeev:2014gwa}
is denoted by ${\sf v}=q^{\frac{1}{2}}$ in this paper.}
\begin{equation}\label{eq:RinHGS}
R^{\rm HG}_{i,j}=\beta^{-j}\frac{(q)_N \bigl(\frac{\alpha}{z}\bigr)_{N-i}\bigl(\frac{1}{\beta}\bigr)_{N-j}\bigl(\frac{\beta}{z}\bigr)_{j}}
{(q)_j \bigl(q\bigr)_{N-j}\bigl(\frac{1}{z}\bigr)_{N}\bigl(\frac{1}{\beta}\bigr)_{N-i}}
\sum_{k=0}^{j}
\frac{\bigl(q^{-j}\bigr)_k\bigl(q^{i-N}\bigr)_k\bigl(q^{1-N}z\bigr)_k\bigl(\frac{z}{\alpha\beta}\bigr)_k}{(q)_k\bigl(q^{-N}\bigr)_k\bigl(\frac{q^{1+i-N}z}{\alpha}\bigr)_k\bigl(\frac{q^{1-j}z}{\beta}\bigr)_k}q^k.
\end{equation}
For $d_2=q^{-m}$, $d_3=q^{-n}$ and $i, j \in [-n,m]$, one can show (see Section~\ref{subsec:r-matrix})
\begin{equation}
\label{r=R}
\begin{array}{l}
r_{i,j}(\Lambda)=d_1^{m-i} q^{(m+1)i}
R^{\rm HG}_{i+n,j+n}\big|_{\{N=m+n, z=\frac{\Lambda}{q}, \alpha=\frac{q^n}{d_1}, \beta=\frac{q^m}{d_4}\}},
\end{array}
\end{equation}
hence, the matrices $\{r_{i,j}(\Lambda)\}_{i,j \in [-n,m]}$ with fixed $N=m+n$ give essentially the same $(N+1) \times (N+1)$ $R$-matrix.
\end{Remark}


The above result implies that the equation
\begin{equation}\label{Shamil-Eq2}
\Psi(\Lambda,x)=\SS \Psi\biggl(\frac{\Lambda}{t}, \frac{x}{qtQ}\biggr),
\end{equation}
for $d_2=q^{-m}$, $d_3=q^{-n}$, $N=m+n$, is written as the $t$-difference equation with $(N+1) \times (N+1)$ matrix coefficient:
\begin{equation}
\psi_j(\Lambda)=\sum_{i=-n}^{m} \psi_i\biggl(\frac{\Lambda}{t}\biggr)r_{i,j}(\Lambda)(qtQ)^{-i},
\end{equation}
where \smash{$\Psi(\Lambda,x)=\sum_{i=-n}^m x^i \psi_i(\Lambda)$}.
This is exactly the $q$-KZ equation for the four point function of $\qg\bigl(A_1^{(1)}\bigr)$
in $N$-down spin sector.\footnote{The parameter $Q$ specifies the intermediate state of the correlation
function $\langle j_{4} | \Phi_{j_3}(1)\Phi_{j_2}(\Lambda) |j_{1} \rangle$.
The variable~$x$ is merely a book-keeping parameter for the spin direction.}
The cases $(m,n)$ with $m+n=N$ correspond to an equivalent equation of size $N+1$.
In Section~\ref{sec:Matsuo}, we will show that the specializations of the affine Laumon partition function
corresponding to the condition $d_2=q^{-m}$, $d_3=q^{-n}$ give a set of $N+1$ solutions which form the fundamental solutions.
In this sense, the Laumon partition function can be viewed as the ``universal fundamental solution'' for $q$-KZ equation.


\subsection[Computation of r\_\{i,j\}]{Computation of $\boldsymbol{r_{i,j}}$}
\label{subsec:r-matrix}

We will explain how to compute $(r_{i,j})_{i,j=-n}^{m}$ from \eqref{eq:rdef}.
By using the formula \eqref{a_n-inversion}, the relation~\eqref{eq:rdef} can be written as
\begin{equation}
q^{-in}(\Lambda d_4)^{i+n}U_i(x)=
\sum_{j=-n}^{m}r_{i,j} q^{-j}\Lambda^{j+n} V_j(x),
\end{equation}
where
\begin{equation}
U_i(x)=\bigl(-d_1 q^{i-m}x\bigr)_{m-i}\biggl(-\frac{qx}{\Lambda d_4}\biggr)_{i+n}, \qquad
V_j(x)=(-q^{-m}x)_{m-j}\biggl(-\frac{q^{1-j}x}{\Lambda}\biggr)_{j+n}.
\end{equation}
Since both $\{U_i(x)\}_{i=-n}^{m}$ and $\{V_j(x)\}_{j=-n}^{m}$ are bases of polynomials of degree $m+n$,
our task is to compute the transition matrix between these two bases. 
For this purpose it is convenient to introduce an intermediate basis $\{W_k(x)\}_{k=-n}^{m}$ defined by\footnote{Another choice
$\tilde{W}_k(x)=U_k(x){|}_{d_4\mapsto \frac{q^{m+1}}{\Lambda}}$ also works as well,
which gives different decomposition of the same matrix $(r_{i,j})$.}
\begin{equation}
W_k(x)=U_k(x)\Big|_{d_1 \mapsto \frac{q^{n+1}}{\Lambda}}=\biggl(-\frac{q^{n+1+k-m}}{\Lambda}x\biggr)_{m-k}\biggl(-\frac{qx}{\Lambda d_4}\biggr)_{k+n}.
\end{equation}
The polynomial $W_k(x)$ is defined so that the zeros of $U_i(x)$, $W_k(x)$, $V_j(x)$ overlap with each other as follows:
\begin{gather}
U_i(x)\colon \ -x=\underbrace{\frac{q}{d_1}, \dots, \frac{q}{d_1} q^{m-i-1}}_{m-i},\,
\underbrace{\frac{d_4 \Lambda}{q^{n+i}}, \dots, \frac{d_4 \Lambda}{q}}_{n+i}, \nonumber\\
W_k(x)\colon \ -x=\underbrace{\frac{\Lambda}{q^{n}}, \dots, \frac{\Lambda}{q^{n}}q^{m-k-1}}_{m-k}, \,
\underbrace{\frac{d_4 \Lambda}{q^{n+k}}, \dots, \frac{d_4 \Lambda}{q}}_{n+k},\nonumber\\
V_j(x)\colon \ -x=\underbrace{\frac{\Lambda}{q^{n}}, \dots, \frac{\Lambda}{q^{n}}q^{n+j-1}}_{n+j}, \,
\underbrace{q^{j+1}, \dots, q^m}_{m-j}.\label{commonzero}
\end{gather}

Let us compute the transition coefficients $r^{UW}_{i,k}$ and $r^{WV}_{k,j}$ defined by
\begin{equation}
\label{eq:UWV}
U_i(x) =\sum_{k=-n}^{m} r^{UW}_{i,k} W_k(x), \qquad
W_k(x) =\sum_{j=-n}^{m} r^{WV}_{k,j} V_j(x).
\end{equation}
Due to the cancellation of common factors arising from \eqref{commonzero}, the first relation of (\ref{eq:UWV}) can be written as
\begin{gather}\label{eq:rUWcoef}
\bigl(-d_1 q^{i-m}x\bigr)_{m-i}=\sum_{k=i}^{m} r^{UW}_{i,k} \biggl(-\frac{q^{n+1+k-m}}{\Lambda}x\biggr)_{m-k}\biggl(-\frac{q^{i+n+1}x}{\Lambda d_4}\biggr)_{k-i}.
\end{gather}
Note that we may delete the coefficient of $W_k(x)$ which does not have the common zeros with~$U_i(x)$.
Similarly, from the second equation of \eqref{eq:UWV}, we have
\begin{gather}\label{eq:rWVcoef}
\biggl(-\frac{qx}{\Lambda d_4}\biggr)_{k+n}=
\sum_{j=m-k-n}^{m} r^{WV}_{k,j} (-q^{-m}x)_{m-j}\biggl(-\frac{q^{1-j}x}{\Lambda}\biggr)_{j+n-m+k}.
\end{gather}

\begin{Proposition}
The nonzero coefficients of
$r^{UW}_{i,k}$, $i\leq k$, can be determined as
\begin{gather}
r^{UW}_{i,k}=q^{\frac{1}{2}(i-k)(2m+1+i-k)}(-d_4)^{k-i}\dfrac{(q)_{m-i}}{(q)_{k-i}(q)_{m-k}} \nonumber\\ \hphantom{r^{UW}_{i,k}=}{}
\times
\dfrac{\bigl(d_1 \Lambda q^{i-k-n}\bigr)_{k-i}\bigl(d_1d_4\Lambda q^{-m-n-1}\bigr)_{m-k}}{(d_4 q^{-m})_{m-i}}.\label{eq:rUWre}
\end{gather}
Hence, $\bigl(r^{UW}_{i,k}\bigr)$ is upper triangular matrix with factorized elements.
Similarly, the nonvanishing coefficients of $r^{WU}_{k,j}$, $N-(k+n) \leq j+n$, are
\begin{gather}
r^{WV}_{k,j}=q^{\frac{1}{2}(j-k-m-n-1)(j+k-m+n)}(-\Lambda)^{j+k-m+n} \dfrac{(q)_{k+n}}{(q)_{k+j-m+n}(q)_{m-j}} \nonumber \\ \hphantom{r^{WV}_{k,j}=}{}
\times{}\dfrac{\bigl(\frac{q^{m+1}}{d_4\Lambda}\bigr)_{k+j-m+n} \bigl(\frac{q^{j+1}}{d_4}\bigr)_{m-j}}{\bigl(q^{-k-n}\Lambda\bigr)_{k+n}}.\label{eq:rWVre}
\end{gather}
The matrix $\bigl(r^{WV}_{k,j}\bigr)$ is ``lower triangular along anti-diagonal'' with factorized elements.
\end{Proposition}

\begin{proof}
We follow the method in \cite[Section~3]{Rosengren:2003nee}.
By the change of the parameters
\begin{equation}
N= m-i, \qquad r= m-k, \qquad (a,b,c) = \biggl( -d_1 q^{i-m}, -\frac{q^{i+n+1}}{d_4 \Lambda }, -\frac{q^{n+1}}{\Lambda}\biggr)
\end{equation}
for \eqref{eq:rUWcoef}, and
\begin{equation}
N= k+n, \qquad r= j+k-m+n, \qquad (a,b,c) = \biggl(-\frac{q}{d_4 \Lambda }, -q^{-m}, -\frac{q^{k-m+n+1}}{\Lambda }\biggr)
\end{equation}
for \eqref{eq:rWVcoef}, respectively, both equations are written as the binomial expansion
\begin{equation}
(a x)_N=\sum_{r=0}^N C_{N,r}(b x)_{N-r} (c q^{-r} x)_r.
\end{equation}
This implies that the coefficients $C_{N,r}$ are given by
\begin{equation}\label{eq:Csol}
C_{N,r}=q^{r(r+1)/2}\biggl(-\frac{b}{c}\biggr)^r \frac{(q)_N}{(q)_r(q)_{N-r}}
\frac{\bigl(\frac{a}{b}\bigr)_r \bigl(q^{r+1}\frac{a}{c}\bigr)_{N-r}}{\bigl(\frac{b q}{c}\bigr)_N}.
\end{equation}
To see this, note that
\begin{gather}
 1-q^N a x=A_r \bigl(1-q^{N-r} b x\bigr)+B_r\bigl(1-q^{-r-1} c x\bigr),\nonumber\\
A_r=\dfrac{c-a q^{N+1+r}}{c-b q^{N+1}}, \qquad
B_r=\dfrac{b-a q^r}{b-c q^{-N-1}}.
\end{gather}
Hence,
\begin{equation}
\frac{(a x)_{N+1}}{(a x)_N}=A_r\frac{(b x)_{N+1-r}}{(b x)_{N-r}}+B_r \frac{\bigl(c q^{-r-1}x\bigr)_{r+1}}{(c q^{-r}x)_r}.
\end{equation}
From this, the coefficients $C_{N,r}$ are uniquely determined by ``Pascal's triangle''
\begin{equation}\label{eq:pascal}
C_{N+1,r}=A_r C_{N,r}+B_{r-1}C_{N,r-1},
\end{equation}
with the boundary conditions
$C_{0,0}=1$, $C_{N,-1}=C_{N,N+1}=0$.
 For $C_{N,r}$ given by \eqref{eq:Csol}, the boundary conditions are obvious.
The relation \eqref{eq:pascal} follows from
\begin{equation}
\frac{C_{N+1,r}}{C_{N,r}}=\frac{\bigl(1-q^{N+1}\bigr) \bigl(1-\frac{a q^{N+1}}{c}\bigr)}{\bigl(1-q^{N+1-r}\bigr) \bigl(1-\frac{b q^{N+1}}{c}\bigr)},
\qquad
\frac{C_{N,r-1}}{C_{N,r}}= -\frac{c q^{-r} (1-q^r) \bigl(1-\frac{a q^r}{c}\bigr)}{b \bigl(1-q^{N+1-r}\bigr)
\bigl(1-\frac{a q^{r-1}}{b}\bigr)}
\end{equation}
and the identity
\begin{gather}
\bigl(1-q^{N+1}\bigr) \biggl(1-\frac{a}{c} q^{N+1}\biggr)
-q^{N+1-r}(1-q^r) \biggl(1-\frac{a}{c} q^r\biggr) \nonumber\\
\qquad{}-\bigl(1-q^{N+1-r}\bigr) \biggl(1-\frac{a}{c} q^{N+1+r}\biggr)=0.
\end{gather}
Hence, \eqref{eq:Csol} is proved.
Back to the original parameters, we obtain the desired results.
\end{proof}

 Combining \eqref{eq:rUWre} and \eqref{eq:rWVre}, we have
 \begin{gather}
 r_{i,j}=q^{-in}(\Lambda d_4)^{i+n}q^{j}\Lambda^{-j-n} r^{UV}_{i,j}=
 q^{-in}(\Lambda d_4)^{i+n}q^{j}\Lambda^{-j-n} \sum_{k=-n}^{m} r^{UW}_{i,k} r^{WV}_{k,j}\nonumber\\ \hphantom{r_{i,j}}{}
= \sum_{k=m-n-j}^{m} q^{-in}q^{j}
 q^{\frac{1}{2}(i-k)(2m+1+i-k)}q^{\frac{1}{2}(j-k-m-n-1)(j+k-m+n)}(-1)^{-m+n-i+j}\nonumber\\ \hphantom{r_{i,j}=}{}
\times d_4^{k+n} \Lambda^{i+n+k-m} \dfrac{(q)_{m-i}}{(q)_{k-i}(q)_{m-k}}
\dfrac{\bigl(d_1 \Lambda q^{i-k-n}\bigr)_{k-i}\bigl(d_1d_4\Lambda q^{-m-n-1}\bigr)_{m-k}}{(d_4 q^{-m})_{m-i}} \nonumber\\ \hphantom{r_{i,j}=}{}
\times \dfrac{(q)_{k+n}}{(q)_{k+j-m+n}(q)_{m-j}}
\dfrac{\bigl(\frac{q^{m+1}}{d_4\Lambda}\bigr)_{k+j-m+n} \bigl(\frac{q^{j+1}}{d_4}\bigr)_{m-j}}{(q^{-k-n}\Lambda)_{k+n}}.
\end{gather}
To show \eqref{r=R}, we make the change of variables
\begin{gather}
z=\frac{\Lambda}{q}, \qquad \alpha=\frac{q^n}{d_1}, \qquad \beta=\frac{q^m}{d_4}.
\end{gather}
Using $(a)_k=(-a)^k q^{k(k-1)/2} \bigl(\frac{1}{a q^{k-1}}\bigr)_k$ and $(a)_{k+l}=(a)_k\bigl(a q^k\bigr)_{l}$,
we compute
\begin{gather}
 r_{i,j} = d_1^{m-i} \beta^{-n-j} q^{(m+1)i} \dfrac{\bigl(\frac{\alpha}{z}\bigr)_{m-i} \bigl(\frac{1}{\beta}\bigr)_{m-j}
 \bigl(\frac{\beta}{z}\bigr)_{j+n}}{\bigl(\frac{1}{z}\bigr)_{m+n}\bigl(\frac{1}{\beta}\bigr)_{m-i}}
 \sum_{k=m-n-j}^{m} q^{(m-k)(i+n)}
\nonumber \\ \hphantom{r_{i,j} =}{}
\times \biggl(\frac{\beta}{z}\biggr)^{m-k}\!
 \dfrac{(q)_{m-i}}{(q)_{k-i}(q)_{m-k}} \dfrac{(q)_{k+n}}{(q)_{k+j-m+n}(q)_{m-j}}
\dfrac{\bigl(\frac{z}{\alpha\beta}\bigr)_{m-k}}{\bigl(\frac{z}{\alpha}q^{1+i-m}\bigr)_{m-k}}
\dfrac{\bigl(z q^{1-N}\bigr)_{m-k}}{\bigl(\frac{\beta}{z}q^{j+n+k-m}\bigr)_{m-k}}.\!\!\!\!
\end{gather}
One can check these coefficients $r_{i,j}$ coincide with the results in Remark \ref{explicitR}.

\section[Jackson integral representation of solutions to the q-KZ equation]{Jackson integral representation of solutions \\ to the $\boldsymbol{q}$-KZ equation}
\label{sec:duality}

\subsection{Jackson integral of symmetric Selberg type}

Solutions to the $q$-KZ equation allow a representation by the Jackson integral. In this subsection we
review the results of \cite{MIto} which are needed later.

In the construction of solutions to the $q$-KZ equation for $\qg\bigl(A_1^{(1)}\bigr)$ by M.~Ito \cite{MIto},
the $N$-tuple Jackson integral with the integration variables $z=(z_1, \dots, z_N)$
\begin{equation}\label{Jackson-pairing}
\langle \phi(z), \xi \rangle
=(1-t)^N \sum_{\nu \in \bbZ^N} [ \phi(z)\Phi(z) \Delta(1,z)]_{z_i=\xi_i t^{\nu_i}}
\end{equation}
is considered.\footnote{In view of the application in the next section, we exchange $q$ and $t$ in the original paper \cite{MIto}.}
Here,
\begin{equation}\label{Delta-q}
\Delta(q,z) :=\prod_{1\leq i< j \leq N} \bigl(z_i-q^{-1} z_j\bigr),
\end{equation}
and
\begin{equation}\label{common}
\Phi(z)=\prod_{i=1}^N z_i^{\alpha} \dfrac{\bigl(t a_1^{-1} z_i,t a_2^{-1} z_i; t\bigr)_{\infty}}{(b_1 z_i, b_2 z_i; t)_{\infty}}
\cdot
\prod_{1\leq i<j\leq N} z_i^{2(\log_t q)-1}\frac{\bigl(t q^{-1} z_j/z_i;t\bigr)_{\infty}}{(q z_j/z_i;t)_{\infty}}
\end{equation}
is a common weight factor with four parameters $a_1$, $a_2$, $b_1$ and $b_2$.
Note that $\Delta(1,z)$ is nothing but the Vandermonde determinant.
In the pairing \eqref{Jackson-pairing}, $\xi=(\xi_1, \dots, \xi_N)$ has the meaning of parameters for an integration cycle.
On the other hand, $\phi(z)$ defines a cocycle function.
In order to give a basis of the space of cohomology classes, in \cite{MIto} the symmetric polynomial
\begin{equation}
\tilde{E}_{k,i}(a,b;z)=\dfrac{1}{\Delta(1,z)}{{\mathcal A}(E_{k,i}(a,b;z))}
={\mathcal S} \dfrac{E_{k,i}(a,b;z)}{\Delta(1,z)}, \qquad a,b \in \bbC^{\times},
\end{equation}
was employed by introducing
\begin{equation}
E_{k,i}(a,b;z)=z_1\cdots z_k~\Delta(q,z)\prod_{j=1}^{N-i}(1-b z_j) \prod_{j=N-i+1}^N (1-a^{-1}z_j).
\end{equation}
Here ${\mathcal S}$/${\mathcal A}$ denotes the symmetrization/anti-symmetrization of the variables $(z_1, \dots, z_N)$, respectively.
In particular, $e_i(a,b;z)=\tilde{E}_{0,i}(a,b;z)$, $0\leq i\leq N$, gives a basis which is called Matsuo basis \cite{AMatsuo,AMatsuo2} (see also \cite{Mimachi,VarchenkoCMP}).

The function \smash{$c(x)=x^{\log_t (\frac{a}{b})}\frac{\vartheta(ax;t)}{\vartheta(bx;t)}$} with
$\vartheta(x;t)=(x;t)_\infty(t/x;t)_\infty$ satisfies $c(tx)=c(x)$. Such a function is called pseudo constant.
The following lemma shows that the function $\Phi(z)\Delta(1,z)$ in \eqref{Jackson-pairing}
can be written in a symmetric form, up to some pseudo constant factor.
Due to this property, the integral is sometimes called Jackson integral of symmetric Selberg type.
\begin{Lemma}
Let $\tau=\log_t q$ and $i,j \in [N] :=\{1,2,\dots,N\}$. We have
\begin{equation}\label{eq:diff-prod}
\prod_{i \neq j} \frac{\bigl(t q^{-1} z_j/z_i;t\bigr)_{\infty}}{(t z_j/z_i;t)_{\infty}}=C(z)
\Delta(1,z) \prod_{i=1}^{N} z_i^{-\tau(N-1)}
 \prod_{i< j} z_i^{2 \tau-1}\frac{\bigl(t q^{-1} z_j/z_i;t\bigr)_{\infty}}{(q z_j/z_i;t)_{\infty}},
\end{equation}
where
$C(z)=\prod_{i< j} \bigl(\frac{z_j}{z_i}\bigr)^{\tau}\frac{\vartheta(t q^{-1} z_i/z_j;t)}{\vartheta(t z_i/z_j;t)}$
is a pseudo constant for each variable $z_i$.
\end{Lemma}
\begin{proof} The left-hand side is computed as follows:
\begin{align*}
\prod_{i \neq j} \frac{\bigl(t q^{-1} z_j/z_i;t\bigr)_{\infty}}{(t z_j/z_i;t)_{\infty}}
=&~\prod_{i< j} \frac{\bigl(t q^{-1} z_i/z_j;t\bigr)_{\infty}(t q^{-1} z_j/z_i;t)_{\infty}}
{(t z_i/z_j;t)_{\infty}(t z_j/z_i;t)_{\infty}}\\
=&~\prod_{i< j} \frac{\vartheta\bigl(t q^{-1} z_i/z_j;t\bigr)(z_j/z_i;t)_{\infty}\bigl(t q^{-1} z_j/z_i;t\bigr)_{\infty}}
{\vartheta(t z_i/z_j;t)(q z_j/z_i)_{\infty}(t z_j/z_i;t)_{\infty}}\\
=&~C(z) \prod_{i< j} \Bigl(\frac{z_j}{z_i}\Bigr)^{-\tau} (1-z_j/z_i)\frac{\bigl(t q^{-1} z_j/z_i;t\bigr)_{\infty}}{(q z_j/z_i;t)_{\infty}}\\
 =&~C(z) \Delta(1,z) \prod_{i=1}^{N} z_i^{-\tau(N-1)}
 \prod_{i< j} z_i^{2 \tau-1}\frac{\bigl(t q^{-1} z_j/z_i;t\bigr)_{\infty}}{(q z_j/z_i;t)_{\infty}}.
\end{align*}
The last line follows from
\[\prod_{i< j} \Bigl(\frac{z_j}{z_i}\Bigr)^{-\tau}z_i^{-1}
=\prod_{i< j} (z_jz_i)^{-\tau}z_i^{2\tau-1}
=\prod_{i=1}^{N} z_i^{-\tau(N-1)}\prod_{i<j}z_i^{2\tau-1}.
\tag*{\qed}
\]
\renewcommand{\qed}{}
\end{proof}

For the cocycle function $\phi(z)$ in the pairing \eqref{Jackson-pairing},
we will take the Matsuo basis $e_0(a_2,b_1), \dots,\allowbreak e_N(a_2,b_1)$.\footnote{Since the common factor \eqref{common} is
symmetric under the exchanges $a_1 \leftrightarrow a_2$ and $b_1 \leftrightarrow b_2$, there are four choices of the Matsuo basis.
Here we follow the choice in \cite{MIto}.}
Namely, we consider the following $\bbC^{N+1}$-valued function
\begin{equation}\label{take-Matsuo}
\Psi= [\Psi_{-n},\dots, \Psi_{m}] =
[ \langle e_N(a_2,b_1), \xi \rangle,\dots, \langle e_0(a_2,b_1), \xi \rangle ].
\end{equation}
In order to identify the instanton partition function to be introduced in the next section with the Jackson integral,
it is important to keep the factorized structure of the integrand as far as possible,
hence we use the following expression:
\begin{align}
\label{eq:e-factor}
{\hat e}_k(a,b;z)&{}=e_{N-k}(a,b;z) \\
&{}=[k]_{q^{-1}}![N-k]_{q^{-1}}!\sum_{I \sqcup J=[N], |J|=k} \prod_{i \in I}\biggl(1-\frac{z_i}{a}\biggr)\prod_{j \in J}(1-b z_j)
\prod_{i \in I,j \in J}\frac{z_j-q^{-1} z_i}{z_j-z_i}, \nonumber
\end{align}
where the sum is taken over the $\bigl( {N \atop k} \bigr)$ terms corresponding
to disjoint union $I \sqcup J= [N]=\{1,2,\dots, N\}$ with $|I|=N-k$, $|J|=k$.
The formula (\ref{eq:e-factor}) follows from Proposition \ref{Shuffle} in~Appendix \ref{App:symmetrization} for the case $s=2$
with $f_1(x)=1-b_1 x$ and \smash{$f_2(x)=1-\frac{x}{a_2}$}.\footnote{As noted in \cite{Reshe-Bethe},
these cocycle functions naturally arise as the Bethe vector.
See \cite{AOF} for the developments based on the geometric method.}
Thus the summation in Jackson integral is written as the sum over the cone (\ref{eq:cone}) and
additional $2^N=\sum_{k=0}^N \bigl( {N \atop k} \bigr)$ sums. As we will see in the next section,
in the relevant instanton partition function, the first sum corresponds to the sum over the two Young diagrams
with even lengths for all the columns and the second sum comes from their even/odd variants depending on the number of
columns with odd length.

The expression \eqref{eq:e-factor} looks like rational, but it is a polynomial in $z=(z_1, \dots, z_N)$.
In fact, the base $e_k(a,b;z)$ can be characterized as the linear combination of the elementary symmetric functions in
$z$ having the following specialization \cite{MIto,ItoNoumi}:
\begin{equation}\label{eq:matsuo-sp}
e_k\bigl(a,b; \bigl(x,xq \dots,xq^{N-1}\bigr)\bigr)=
[N]_{q^{-1}}! \prod_{i=1}^{N-k}\bigl(1-q^{i} b x\bigr) \prod_{i=1}^{k-1}\biggl(1-q^i \frac{x}{a}\biggr).
\end{equation}

In \cite{MIto}, the following fact is proved.
\begin{Theorem}[{\cite{MIto}}] \label{thm:Ito}
With respect to the shifts $T_{\alpha}\colon \alpha \to \alpha+1$ and $T_i=T_{t,b_i}^{-1}T_{t,a_i}$, $i=1,2$,
the Jackson integral $\Psi$ satisfies the following system of difference equations:
\begin{align}
T_\alpha \Psi= \Psi K_0, \qquad T_1 \Psi= \Psi K_1,
\qquad T_2 \Psi= \Psi K_2,
\end{align}
where
\begin{gather}
K_0=R^{-1}AR=D_2 A D_2^{-1}, \qquad
K_1=R^{-1} D_1,\qquad
K_2=D_2 (T_2 R), \nonumber\\
D_1=\bigl(\bigl(t^\alpha q^{N-1}\bigr)^{N-i}\delta_{ij}\bigr)_{0\leq i,j\leq N}, \qquad
D_2=\bigl(\bigl(t^\alpha q^{N-1}\bigr)^{i}\delta_{ij}\bigr)_{0\leq i,j\leq N},
\end{gather}
and $\delta_{ij}$ is the Kronecker symbol.
An explicit form of the matrix $R$ is given by\footnote{The matrix $R$ in this section is slightly different from the $R$ matrix in the previous section.
See Appendix~\ref{App:List-R} for a summary of various $R$-like matrices used in this paper.}
\begin{gather}
 R=~L_R D_R U_R, \nonumber\\
L_R = \bigl(l^R_{ij}\bigr)_{0\leq i,j\leq N}, \qquad D_R =\bigl(\delta_{ij}d^R_{j}\bigr)_{0\leq i,j\leq N}, \qquad
U_R = \bigl(u^R_{ij}\bigr)_{0\leq i,j\leq N},\label{eq:RGauss}
\end{gather}
where
\begin{gather}
l^R_{ij}= \cb{N - j}{N - i}{q^{-1}}
\dfrac{(-1)^{i - j} q^{-\ctwo{i - j}}(a_2 b_2 q^j; q)_{i - j}}{\bigl(a_1^{-1} a_2 q^{-(N - 2 j - 1)};q\bigr)_{i - j}}, \qquad N \ge i\geq j \geq 0, \nonumber\\
d^R_{j}=\dfrac{\bigl(a_1 a_2^{-1} q^{-j};q\bigr)_{ N - j}(a_2 b_1;q)_j}{(a_1 b_2;q)_{N - j}\bigl(a_1^{-1} a_2 q^{-(N - j)}; q\bigr)_ j}, \qquad 0\leq j \leq N, \nonumber \\
u^R_{ij}=\cb{j}{i}{q^{-1}}\dfrac{\bigl(a_1 b_1 q^{N - j};q\bigr)_{j - i}}{\bigl(a_1 a_2^{-1} q^{N - i - j};q\bigr)_{j - i} }, \qquad 0 \leq i\leq j \leq N,
\end{gather}
and other elements are zero. $\left[\begin{smallmatrix} n \\ k \end{smallmatrix}\right]_{q^{-1}}$ denotes $q^{-1}$-binomial coefficient $($see \eqref{q-binom}$)$.
The matrix~$A$ is also given in a Gauss decomposed form as
\begin{gather}
A=L_A D_A U_A, \nonumber\\
 L_A = \bigl(l^A_{ij}\bigr)_{0\leq i,j\leq N}, \qquad D_A =\bigl(\delta_{ij}d^A_{j}\bigr)_{0\leq i,j\leq N}, \qquad
U_A = \bigl(u^A_{ij}\bigr)_{0\leq i,j\leq N},\label{eq:AGauss}
\end{gather}
where the nonzero elements are
\begin{gather}
l^A_{ij}= (-1)^{i - j} q^{\ctwo{N - i} - \ctwo{N - j}}\cb{N - j}{N - i}{q}
\dfrac{\bigl(a_2 b_2 q^j; q\bigr)_{ i - j}}{\bigl(t^\alpha a_2 b_2 q^{2 j};q\bigr)_{i - j}}, \qquad N \ge i\geq j \geq 0, \nonumber\\
d^A_{j} =a_1^{N - j} a_2^j q^{\ctwo{j} + \ctwo{N - j}} \dfrac{(t^\alpha;q)_{j} \bigl(t^\alpha a_2 b_2 q^{2 j}; q\bigr)_{N - j}}
{\bigl(t^\alpha a_2 b_2 q^{j - 1}; q\bigr)_{j} \bigl(t^\alpha a_1 a_2 b_1 b_2 q^{n + j - 1};q\bigr)_{N - j}}, \qquad 0\leq j \leq N, \nonumber\\
u^A_{ij}= \bigl(-t^\alpha a_1^{-1} a_2 \bigr)^{j - i} q^{\ctwo{j} - \ctwo{i}} \cb{j}{i}{q}
\dfrac{\bigl(a_1 b_1 q^{N - j}; q\bigr)_{ j - i}}{\bigl(t^\alpha a_2 b_2 q^{2 i}; q\bigr)_{ j - i}},
\qquad 0 \leq i\leq j \leq N.
\end{gather}
\end{Theorem}

\begin{Remark}
The difference equations for the shift $T_i = T_{t,b_i}^{-1} T_{t,a_i}$ are nothing but the traditional (original) $q$-KZ equation.
The components of the $R$ matrix appearing in this type of $q$-KZ equation are nothing but the connection coefficients
between two Matsuo bases $\{e_i(a,b)\}$ with different parameters $a_i$ and $b_i$.
Due to the specialization given in equation~\eqref{eq:matsuo-sp},
the computation can be reduced to the connection problem of single variable polynomials as in \cite{Rosengren:2003nee}.
On the other hand, the matrix $A$ for the shift of $T_{\alpha}\colon \alpha \to \alpha+1$
is obtained in a different manner and the basis $E_{k,i}$ was introduced for this purpose in \cite{MIto}.
It is the $q$-KZ equation for $T_{\alpha}$ that is related to the Shakirov's equation.
\end{Remark}

The following relation between the matrices $R$ and $A$ seems to be noticed long ago at least among specialists
 (e.g., \cite{MIto1997}).
 We give a proof since the relation plays an important role in Section~\ref{subsBF}.
\begin{Proposition}
We have
\begin{equation}
A=s T(R) D,
\end{equation}
where $s$ is a scalar and $D$ is a diagonal matrix given by
\begin{equation}
s=q^{N(N-1)/2}(a_1a_2 b_2)^N \dfrac{(t^\alpha;q)_N}{\bigl(t^\alpha q^{N-1}a_1a_2b_1b_2;q\bigr)_N}, \qquad
D=\bigl((a_1b_2)^{-i}\delta_{ij}\bigr)_{0\leq i,j\leq N}.
\end{equation}
$T$ is a shift operator acting only on the parameters $a_2$, $b_2$ as
\begin{equation}
T=\bigl\{a_2 \to a_2 \bigl(t^\alpha q^{N-1}a_1b_2\bigr), \, b_2 \to b_2\bigl(t^\alpha q^{N-1}a_1b_2\bigr)^{-1}\bigr\}.
\end{equation}
\end{Proposition}

\begin{proof}
Using the relation $\bigl(x^{-1};q\bigr)_k=(-1)^k q^{k(k-1)/2}x^{-k}\bigl(x q^{1-k};q\bigr)_k$
and the explicit forms \eqref{eq:RGauss} and~\eqref{eq:AGauss}, one can check
\begin{equation}
L_A=T(L_R), \qquad
D_A=s T(D_R) D, \qquad
U_A=D^{-1} T(U_R) D.
\end{equation}
Then we obtain
\begin{equation}
A=L_A D_A U_A=T(L_R) s T(D_R) D D^{-1} T(U_R) D=s T(L_R D_R U_R)D=s T(R) D,
\end{equation}
as desired.
\end{proof}

The compatibility of the dual pair of difference equations in Theorem \ref{thm:Ito} implies $RD_2 A = AR D_2$.
In Appendix \ref{App:matrix-inversion}, we give a direct check of the compatibility based on the matrix inversion formula of Andrews and Bressoud.
It is amusing that the compatibility condition follows from Bailey's transformation formula for terminating very-well-poised balanced
series ${}_{10} W_9$.


\subsection{Lattice truncation by a choice of the cycle}\label{sec3.2}

The sum in the Jackson integral \eqref{Jackson-pairing} considered in \cite{MIto} is bilateral, namely it is taken over
$\{\nu_i \}\in \bbZ^N$. On the other hand, the Nekrasov partition function involves a sum over Young diagrams.
For the function $\Psi$, the discrepancy is remedied by an appropriate choice of the cycle $\xi=(\xi_1, \dots, \xi_N)$.
The suitable cycle is already known (see \cite[equation~(3.19)]{ItoNoumi} and \mbox{\cite[equations~(4.4) and (4.5)]{ItoForrester}}).
In fact, we have the following.

\begin{Proposition}
For $\xi$ given by
\begin{equation}\label{lattice-truncate}
\xi=\xi_{m,n}=\bigl(\underbrace{a_2,a_2 q, \dots, a_2 q^{n-1}}_{n}, \underbrace{a_1, a_1 q, \dots, a_1 q^{m-1}}_{m}\bigr), \qquad m+n=N,
\end{equation}
the lattice summation over $\bigl\{z_i=\xi_i t^{\nu_i}\bigr\}$ is truncated to a cone
\begin{align}
&0\leq \nu_1\leq \nu_2 \leq \dots \leq \nu_{n-1} \leq \nu_{n}, \notag \\
&0 \leq \nu_{n+1} \leq \nu_{n+2}\leq \dots \leq \nu_{N-1} \leq \nu_{N},\label{eq:cone}
\end{align}
and we can normalize the $\bbC^{N+1}$-valued function \eqref{take-Matsuo} as
\begin{equation}\label{eq:psi-tri}
\Psi^{T}=
\begin{bmatrix}
\Psi_{m}\\
\vdots\\
\Psi_{0}\\
\vdots\\
\Psi_{-n}
\end{bmatrix}=
\begin{bmatrix}
*&*&*&*&\cdots\\
 \vdots&\vdots&\vdots&\vdots&\cdots \\
 1&*&*&*&\cdots\\
&\ddots&\ddots&\ddots&\cdots\\
O& &*&*&\cdots
 \end{bmatrix}
\begin{bmatrix}
1\\ \Lambda\\ \Lambda^2\\ \Lambda^3\\ \vdots
\end{bmatrix},
\end{equation}
where $\Lambda=t^\alpha$.
\end{Proposition}

When we impose the mass truncation condition $d_2=q^{-m}$, $d_3=q^{-n}$ the partition function becomes a Laurent
polynomial in $x$. With the normalization \eqref{eq:psi-tri} the component $\Psi_{i}$, ${-n \leq i \leq m}$,
is to be identified with the coefficient of $x^i$-term in the Laurent polynomial (see also Figure~\ref{Fig:mass-truncation}
for the structure of the partition function).

\begin{proof}
The condition \eqref{eq:cone} easily follows from the explicit form \eqref{common} of the function $\Phi(z)$.
We will show \eqref{eq:psi-tri}. From the expression \eqref{eq:e-factor} for $e_k=e_k(a_2,b_1,z)$ with
\[
z=\bigl(\underbrace{a_2 t^{\nu_1}, a_2 q t^{\nu_2},\dots, a_2 q^{n-1} t^{\nu_n}}_{n},
\underbrace{a_1 t^{\nu_{n+1}}, a_1 q t^{\nu_{n+2}},\dots, a_1 q^{m-1} t^{\nu_{n+m}}}_{m}\bigr),
\]
we see that the leading term in $\Lambda$-expansion of each component of $\Psi$ has a single contribution
arising from a specific $\{\nu_i\}$ and $J\subset [N]$, $|J|=k$.
Explicitly, for $0\leq k\leq m$, we have
\[
e_k=g_k+O\bigl(\Lambda^1\bigr), \qquad
g_k=\frac{\bigl(q^{-1};q^{-1}\bigr)_m(b_1a_2;q)_n\bigl(q^{-N+k};q\bigr)_n\bigl(q^kb_1a_1;q\bigr)_{m-k}\bigl(q^{-n}\frac{a_1}{a_2},q\bigr)_k}{\bigl(1-q^{-1}\bigr)^N},
\]
arising from $J=\{n+1, n+2, \dots, n+k\}$, $(\nu_i)=\bigl(0^N\bigr)$,
and for $0\leq l \leq n$, we have
\begin{gather*}
e_{m+l}=h_{l} \Lambda^{l}+O\bigl(\Lambda^{l+1}\bigr),\\
h_{l}=\frac{\bigl(q^{-1};q^{-1}\bigr)_{n-l}\bigl(q^{-1};q^{-1}\bigr)_{m+l}(t;q)_l(b_1a_2;q)_{n-l}\bigl(q^{l-n}\frac{a_1}{a_2};q\bigr)_{m}}{\bigl(1-q^{-1}\bigr)^N},
\end{gather*}
arising from $J=\{n-l+1, n-l+2, \dots, N\}$, $(\nu_i)=\bigl(0^l,1^{n-l}, 0^m\bigr)$.
Dividing by the scalar factor
\[
h_{0}=\frac{\bigl(q^{-1};q^{-1}\bigr)_{n}\bigl(q^{-1};q^{-1}\bigr)_{m}(b_1a_2;q)_{n}\bigl(q^{-n}\frac{a_1}{a_2};q\bigr)_{m}}{\bigl(1-q^{-1}\bigr)^N},
\]
we have the expression \eqref{eq:psi-tri}.
\end{proof}
%

\newcommand{\tiD}{{\tilde D}}
\newcommand{\tiR}{R}
\newcommand{\tiA}{{\tilde A}}

For the function $\Psi$ normalized as above, the two types of $q$-KZ equations in Theorem \ref{thm:Ito} can be written as
\begin{equation}
\bigl(T^{-1}_{t,a_2}T_{t,b_2}\Psi\bigr) \tiD_1 \tiR=\Psi, \qquad
T_{t,\Lambda} \Psi=\Psi \tiA \tiD_2,
\end{equation}
where the matrices $\tiA$ and $\tiR$ are given as the following connection matrices\footnote{Contrary to the matrix $\tiR$,
the fact that the matrix $\tiA$ is obtained as the connection matrix of Matsuo base is not obvious from the definition.}
\begin{align*}
&[e_0(a_2,b_1), \dots, e_N(a_2,b_1)]=
[e_N(a_1,b_2), \dots, e_0(a_1,b_2)] \tiR,\\
&[e_N(a, b), \dots, e_0(a,b)]_{\substack{a=a_1 b_1\Lambda \\ b=q^{N-1} a_2 b_2}}=
[e_0(c,d), \dots, e_N(c,d)]_{\substack{c=q^{1-N} \\ d=\Lambda ^{-1}}} \tiA,
\end{align*}
and $\tiD_1$, $\tiD_2$ are diagonal matrices given by
\begin{gather*}
\tiD_1=(-1)^N q^\frac{(N-1)(n-m)}{2} \biggl(\frac{a_2}{a_1}\biggr)^N \prod_{i=0}^{n-1} \frac{1-q^{m-n+1+i}\frac{a_1}{a_2}}{1-q^i a_2
b_1}\\ \hphantom{\tiD_1=}{}
\times\prod_{i=0}^{m-1}\frac{1-q^i a_1 b_2}{1-q^{n-m+1+i}\frac{a_2}{a_1}}
\cdot {\rm diag} \bigl(\bigl\{\bigl(\Lambda q^{N-1}\bigr)^{m-i}\bigr\}_{i=0}^N\bigr), \\
\tiD_2=\biggl(\frac{q^m a_1}{a_2}\biggr)^n \prod_{i=0}^{N-1}\frac{1-q^i \Lambda }{1-q^{N-1+i} a_1a_2b_1b_2 \Lambda } \cdot {\rm diag}\bigl(\bigl\{\bigl(a_2 b_1\bigr)^{N-i}\bigr\}_{i=0}^{N}\bigr).
\end{gather*}


\subsection[Base-fibre duality of the q-KZ equation]{Base-fibre duality of the $\boldsymbol{q}$-KZ equation}\label{subsBF}

Let us look at the two types of the $q$-KZ equation explicitly in the case at our hand.
It turns out that these $q$-KZ equations are related by the duality which exchanges
the instanton expansion parameter $\Lambda$ and the Coulomb parameter $Q$.
The $q$-KZ equation arising from Shakirov's equation was given by
\begin{equation}\label{eq:qKZ-Sh}
\psi_j(\Lambda)=\sum_{i=-n}^{m} \psi_i\biggl(\frac{\Lambda}{t}\biggr) \bigl(Q^{\vee}\bigr)^i \ r_{i,j}(\Lambda),
\qquad j=-n,\dots, m,
\end{equation}
where $Q^{\vee}:=(q t Q)^{-1}$ is the shift parameter for $x$ and
$[r_{i,j}(\Lambda)]_{i,j =-n}^{m}$
is the $R$-matrix obtained by the specialization $d_2=q^{-m}$, $d_3=q^{-n}$.
Since the leading term $r_{i,j}(0)$ of the expansion $r_{i,j}(\Lambda)=r_{i,j}(0)+O(\Lambda)$ is upper triangular
with $r_{i,i}(0)=q^{i(i+1)}$,
the equation \eqref{eq:qKZ-Sh} has a set of fundamental solutions of the form
\begin{equation}
\psi^{(i)}_{j}(\Lambda)=\Lambda^{\rho_i} Y_{i,j}(\Lambda), \qquad Y_{i,j}(\Lambda)=Y_{i,j}(0)+O(\Lambda),
\qquad i,j=-n, \dots, m,
\end{equation}
where $t^{\rho_i}=\bigl(Q^{\vee}\bigr)^i r_{i,i}(0)=\bigl(Q^{\vee}\bigr)^i q^{i(i+1)}$ and the leading coefficients $Y_{i,j}(0)$ are upper triangular.
One can normalize them as $Y_{i,i}(0)=1$.
Then the equation for $Y_{i,j}(\Lambda)$ is written as
\begin{equation}\label{eq:Lambda-KZ}
Y_{k,j}(\Lambda)=\sum_{i=-n}^{m} \bigl(Q^{\vee}\bigr)^{-k} q^{-k(k+1)} Y_{k,i}\biggl(\frac{\Lambda}{t}\biggr) \bigl(Q^{\vee}\bigr)^i r_{i,j}(\Lambda),
\qquad k,j=-n,\dots, m.
\end{equation}

Since $r_{i,j}(\Lambda)$ is independent of $Q^{\vee}$, this equation depends on the parameter $Q^{\vee}$ only through simple power factors.
The following proposition\footnote{Proposition \ref{another-qKZ} is a consequence of Theorems \ref{thm:Ito} and~\ref{AL=J}.
The proof is omitted, since it is not used in other parts of the paper.} shows that the fundamental solution $Y(\Lambda)=Y\bigl(\Lambda, Q^{\vee}\bigr)$ satisfies another $q$-KZ equation
with respect to the parameter $Q^{\vee}$.
\begin{Proposition}\label{another-qKZ}
The fundamental solution $Y\bigl(\Lambda,Q^{\vee}\bigr)$ satisfies the following equation:
\begin{equation}\label{eq:a-KZ}
\sum_{j=-n}^m Y_{i,j}\biggl(\Lambda, \frac{Q^{\vee}}{t}\biggr) \biggl(\frac{\Lambda d_1}{q^{m+2}}\biggr)^{j} \tilde{r}_{j,k}\bigl(Q^{\vee}\bigr)
=v_{i}\bigl(\Lambda,Q^{\vee}\bigr) Y_{i,k}\bigl(\Lambda,Q^{\vee}\bigr), \qquad -n \leq i,k \leq m,
\end{equation}
where
\begin{equation}
\tilde{r}_{i,j}\bigl(Q^{\vee}\bigr)=r_{i,j}(\Lambda){|}_{\Lambda \mapsto \frac{q^{m+2}}{d_1} Q^{\vee}}
\quad
\end{equation}
and
\begin{equation}
v_i\bigl(\Lambda,Q^{\vee}\bigr)=q^{i(i+1)}\biggl(\frac{\Lambda d_1}{q^{m+2}}\biggr)^i
\frac{\bigl(Q^{\vee} q^{2+2i};q\bigr)_{m-i} \bigl(d_4 Q^{\vee} q^{1-n};q\bigr)_{n+i}}{\bigl(d_1^{-1}Q^{\vee} q^{2+i};q\bigr)_{m-i}\bigl(Q^{\vee} q^{1-n+i};q\bigr)_{n+i}}.
\end{equation}
\end{Proposition}

The equations \eqref{eq:Lambda-KZ} and \eqref{eq:a-KZ} look very similar.
To make this similarity more explicit we make the gauge transformation
$Y_{i,j}\bigl(\Lambda,Q^{\vee}\bigr)=G_i\bigl(Q^{\vee}\bigr) \tilde{Y}_{i,j}\bigl(\Lambda,Q^{\vee}\bigr)$ with
\begin{equation}
G_i\bigl(Q^{\vee}\bigr)=\prod_{k=1}^{\infty}\frac{\bigl(Q^{\vee} q^{2+2i}t^k;q\bigr)_{m-i} \bigl(d_4 Q^{\vee} q^{1-n}t^k;q\bigr)_{n+i}}
{\bigl(d_1^{-1}Q^{\vee} q^{2+i}t^k;q\bigr)_{m-i}\bigl(Q^{\vee} q^{1-n+i}t^k;q\bigr)_{n+i}}.
\end{equation}
Since
\[
G_i^{-1}\biggl(\frac{Q^{\vee}}{t}\biggr) G_i\bigl(Q^{\vee}\bigr) v_i\bigl(\Lambda,Q^{\vee}\bigr)=q^{i(i+1)}\biggl(\frac{\Lambda d_1}{q^{m+2}}\biggr)^i,
\]
one can rewrite the equation \eqref{eq:a-KZ} for ${-n \leq i,k \leq m}$ as
\begin{gather}
\sum_{j=-n}^m {\tilde Y}_{i,j}\biggl(\Lambda, \frac{Q^{\vee}}{t}\biggr) \biggl(\frac{\Lambda d_1}{q^{m+2}}\biggr)^{j}
\tilde{r}_{j,k}\bigl(Q^{\vee}\bigr)=q^{i(i+1)}\biggl(\frac{\Lambda d_1}{q^{m+2}}\biggr)^i {\tilde Y}_{i,k}\bigl(\Lambda,Q^{\vee}\bigr).
\end{gather}
This equation is exactly the same form as \eqref{eq:Lambda-KZ} under the replacement
\begin{equation}
\Lambda \mapsto \frac{q^{m+2}}{d_1} Q^{\vee}, \qquad
Q^{\vee} \mapsto \frac{d_1}{q^{m+2}} \Lambda.
\end{equation}

\begin{figure}[t]
\vspace{30mm}
\begin{center}
\begin{picture}(50,40)
\setlength{\unitlength}{0.9mm}
\thicklines
\put(0,15){\line(1,0){20}}
\put(0,35){\line(1,0){20}}
\put(0,15){\line(0,1){20}}
\put(20,15){\line(0,1){20}}
\put(0,15){\line(-1,-1){10}}
\put(0,35){\line(-1,1){10}}
\put(20,35){\line(1,1){10}}
\put(20,15){\line(1,-1){10}}

\put(8,10){$\Lambda$}
\put(23,24){$Q^{\vee}$}
\end{picture}
\caption{Geometric engineering of ${\rm U}(2)$ gauge theory by local $\mathbb{P}^1 \times \mathbb{P}^1$.
$\Lambda$ is the K\"ahler parameter of the base $\mathbb{P}^1$ and $Q^{\vee}$ is that of the fibre.
We can introduce four matter hypermultiplets by blow-ups.}
\label{toric-diagram}
\end{center}
\end{figure}
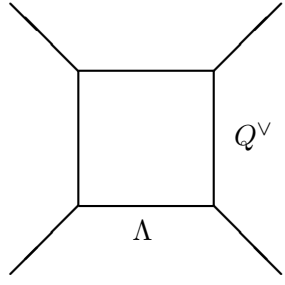

\begin{Remark}
For the five-dimensional quantum Seiberg--Witten curve
with coefficient matrix~${\mathcal D}$ or~${\mathcal D}_3$ in the proof of Proposition 5.1 in~\cite{Awata:2022idl},
the positions of the external lines (tentacles of the corresponding amoeba) are given by
\[
(x,p)=(0,d_3), (0, d_4), \biggl(\infty,\frac{1}{d_1}\biggr), \biggl(\infty, \frac{1}{d_2}\biggr),
(1,0), (d_3d_4 \Lambda \mu,0), \biggl(\frac{\mu q}{d_1d_2}, \infty\biggr), \biggl(\frac{\Lambda}{q}, \infty\biggr)
\]
or
\[
(x,p)=(0,d_3), \biggl(0, \frac{1}{\mu d_4}\biggr), \biggl(\infty,\frac{1}{d_1}\biggr), \biggl(\infty, \frac{d_2}{q^2 \mu}\biggr),
\biggl(\frac{q}{d_2},0\biggr), (d_3 \Lambda,0), \biggl(\frac{1}{d_1}, \infty\biggr), \biggl(\frac{d_4 \Lambda}{q}, \infty\biggr),
\]
where $\mu=Q^{\vee} $.
Then, the exchange $\Lambda \leftrightarrow Q^{\vee}$ can be seen as the exchange of two external lines at $p=\infty$ (for ${\mathcal D}$)
or as the exchange of $x$ and $p$ (for ${\mathcal D_3}$) respectively.
\end{Remark}

\begin{Example}
For $d_2=q^{-m}$, $d_3=q^{-n}$ with $(m,n)=(1,0)$, the solution $Z=Z(\Lambda,x)$ of Shakirov's equation
\begin{align}\label{eq:Sh-01}
&T_{t,\Lambda}Z(\Lambda,x)=\SS T_{qtQ,x}^{-1} Z(\Lambda,x)
\end{align}
is explicitly given by the Heine's series ${}_{2}\phi_{1} \bigl[{a,b \atop c};t,z \bigr]$ with base $t$ as
\begin{align}
Z&={}_2\phi_{1}\left[ {\frac{1}{d_1},\frac{Q t}{d_4} \atop \frac{Q t}{q}};t,\frac{d_1 d_4 \Lambda }{q^2} \right]
-\frac{1-d_1}{1-\frac{q}{Q t}} \ {}_2\phi_{1} \left[ {\frac{t}{d_1},\frac{Q t}{d_4} \atop \frac{Q t^2}{q}};
t,\frac{d_1 d_4 \Lambda }{q^2} \right]x \nonumber \\
&={}_2\phi_1 \left[ {a,z_2 \atop b z_2};t, z_1 \right]+
\frac{b z_2\bigl(1-\frac{1}{a}\bigr)}{1-b z_2} \ {}_2\phi_1 \left[{ta,z_2 \atop t b z_2};t, z_1 \right]x,
\end{align}
where we put $\Lambda = \frac{a q z_1}{b}$,
$Q = \frac{b q z_2}{t}$,
$d_1 =\frac{1}{a}$,
$d_4 = b q$
for simplicity.
Then the coefficients $y_0$, $y_1$ of $Z=y_0+y_1 x$ satisfy
\begin{align}
&\biggl(1-\frac{a z_1}{b}\biggr)T_{t,z_1} Y=Y M(z_1), \label{eq:Ydual1}\\
&\biggl(1-\frac{t}{b z_2}\biggr)T_{t,z_2}^{-1} Y=Y M\biggl(\frac{t}{z_2}\biggr), \label{eq:Ydual2}
\end{align}
where
\begin{equation}
Y=[y_0, y_1], \qquad
M(u)=\begin{bmatrix}1&\\&\frac{a z_1}{b z_2}\end{bmatrix}
\begin{bmatrix}1-\frac{u}{b}&1-\frac{1}{a}\\[1.5mm]1-\frac{1}{b}&1-\frac{1}{au}\end{bmatrix}
\begin{bmatrix} 1&\\&-1\end{bmatrix}.
\end{equation}
The equation \eqref{eq:Ydual1} is the truncated form of \eqref{eq:Sh-01}, while
the equation \eqref{eq:Ydual2} follows from
\begin{align}\label{eq:Ydual3}
&b^{-1}\biggl(1-\frac{t}{z_2}\biggr)T_{t,z_2}^{-1} {\tilde Y}={\tilde Y} M\biggl(\frac{t}{z_2}\biggr),
\end{align}
where ${\tilde Y}=[{\tilde y}_0, {\tilde y}_1]$, ${\tilde Z}={\tilde y}_0,+{\tilde y}_1 x$,
and
\begin{align}\label{eq:Z-Heine}
\tilde{Z}:= {}_2\phi_1 \biggl[{b,z_1 \atop a z_1};t, z_2 \biggr]+
\frac{b z_2(1-\frac{1}{a})}{1-a z_1} {}_2\phi_1 \biggl[{tb,z_1 \atop t a z_1};t, z_2\biggr] x \dfrac{(b z_2,z_1;t)_{\infty}}{(a z_1,z_2;t)_{\infty}} Z.
\end{align}
For $x=0$, the last relation in \eqref{eq:Z-Heine} is Heine's transformation. Note that the ratio ${\tilde Z}/Z$ is independent of $x$,
which explains how a single function can satisfy the dual pair of equations~\eqref{eq:Ydual1} and \eqref{eq:Ydual2}.
\end{Example}

\section{Nekrasov partition function as Jackson integral}
\label{sec:Matsuo}

Let us show that the $K$-theoretic Nekrasov partition function from the affine Laumon space agrees with the Matsuo bases.
When $n=2$ the general formula \eqref{t-formula} derived in Appendix \ref{App:Laumon} for the orbifolded Nekrasov factor
reduces to
\begin{gather}\label{N-factor-even}
\mathsf{N}_{\lambda, \mu}^{(0\vert 2)} (u \vert q,t) =
\prod_{i,j=1}^\infty
\frac{\Bigl[uq^{j-i} t^{1+ \floor{\frac{\mu_{i}^\vee -\lambda_j^\vee}{2}}} ; t\Bigr]_\infty}
{\Bigl[uq^{j-i-1} t^{1+ \floor{\frac{\mu_{i}^\vee -\lambda_j^\vee}{2}}} ; t\Bigr]_\infty}
\frac{\bigl[uq^{j-i-1} t; t\bigr]_\infty} {\bigl[uq^{j-i} t; t\bigr]_\infty},
\\
\label{N-factor-odd}
\mathsf{N}_{\lambda, \mu}^{(1\vert 2)} (u \vert q,t) =
\prod_{i,j=1}^\infty
\frac{\Bigl[uq^{j-i} t^{\frac{1}{2} + \floor{\frac{\mu_{i}^\vee -\lambda_j^\vee +1}{2}}} ; t\Bigr]_\infty}
{\Bigl[uq^{j-i-1} t^{\frac{1}{2} + \floor{\frac{\mu_{i}^\vee -\lambda_j^\vee +1}{2}}} ; t\Bigr]_\infty}
\frac{\bigl[uq^{j-i-1} t^{\frac{1}{2}}; t\bigr]_\infty} {\bigl[uq^{j-i} t^{\frac{1}{2}}; t\bigr]_\infty},
\end{gather}
where $\mu_i^\vee$ and $\lambda_j^\vee$ denote the transpose of the Young diagrams.
For the definition of $[u;t]_\infty$, we refer to \eqref{sinh-reg}.
When one of the partitions is empty, the formula simplifies to
\begin{gather}
\mathsf{N}_{\lambda, \varnothing}^{(0\vert 2)} (u \vert q, \kappa)
 =
\prod_{i \geq 1} \bigl[uq^{i-1}; \kappa^2\bigr]
_{\floor{\frac{\lambda_{i}^\vee + 1}{2}}},
\qquad
\mathsf{N}_{\lambda, \varnothing}^{(1\vert 2)} (u \vert q, \kappa)
=
\prod_{i \geq 1} \bigl[uq^{i-1}\kappa; \kappa^2\bigr]
_{\floor{\frac{\lambda_{i}^\vee}{2}}},
\\
\mathsf{N}_{\varnothing, \mu}^{(0\vert 2)} (u \vert q, \kappa)
 =
\prod_{i \geq 1}\Bigl[uq^{-i}\kappa^{-2\floor{\frac{\mu_i^\vee}{2}}} ; \kappa^2\Bigr]
_{\floor{\frac{\mu_i^\vee}{2}}},
\nonumber\\
\mathsf{N}_{\varnothing, \mu}^{(1\vert 2)} (u \vert q, \kappa)
=
\prod_{i \geq 1}\Bigl[uq^{-i}\kappa^{1 -2\floor{\frac{\mu_i^\vee + 1 }{2}}} ; \kappa^2\Bigr]
_{\floor{\frac{\mu_i^\vee + 1 }{2}}},
\end{gather}
where $t=\kappa^{-2}$.
Note that $\floor{\frac{m+1}{2}} + \floor{\frac{m}{2}} = m$. When $\lambda_{i}^\vee$, $\mu_i^\vee$ are even,
$\mathsf{N}^{(0\vert 2)}$ and $\mathsf{N}^{(1\vert 2)}$ have the same number of factors.
But when they are odd, $\lambda_{i}^\vee$ contributes more to $\mathsf{N}^{(0\vert 2)}$, while $\mu_i^\vee$ does to
$\mathsf{N}^{(1\vert 2)}$. This is due to a difference of the coloring of $\lambda$ and $\mu$ (see Figure~\ref{Fig:coloring}).


\def\dtb#1#2#3#4{{\ytableaushort{}*{#1,#2,#3,#4}*[*(white)]{#1} *[*(orange)]{#1,#2}*[*(white)]{#1,#2,#3}*[*(orange)]{#1,#2,#3,#4}}}
\def\ddtb#1#2#3{{\ytableaushort{}*{#1,#2,#3}*[*(orange)]{#1} *[*(white)]{#1,#2}*[*(orange)]{#1,#2,#3}}}

\begin{figure}[t]\centering
\ytableausetup{boxsize=1.5em}
$$
\bigl(\lambda^{(1)},\lambda^{(2)}\bigr)=\biggl(~~\dtb{4}{2}{1}{1},~~\ddtb{2}{2}{1}~~\biggr), \qquad
\bigl(\lambda^{(1)}\bigr)^\vee=(4,2,1,1), \qquad \bigl(\lambda^{(2)}\bigr)^\vee=(3,2).
$$
\caption{$\mathbb{Z}_2$-coloring of a pair of Young diagrams.}\label{Fig:coloring}
\end{figure}

Recall that by the localization formula the Nekrasov partition function with a surface defect is given by
\begin{align}\label{fixedpoint}
\mathcal{Z}_{\mathrm{AL}}&{}
= 
\mathcal{Z}_{\mathrm{AL}} \left( \left.\left.\begin{array}{c}u_1,u_2 \\v_1,v_2\\w_1,w_2\end{array}
\right| x_1, x_2 \right|q,t\right)\\
&{}= \sum_{\bigl(\lambda^{(1)}, \lambda^{(2)}\bigr)}
\prod_{i,j=1}^2 \frac{\Nk^{(j-i\vert 2)}_{\varnothing,\lambda^{(j)}}(u_i/v_j|q,t)
\Nk^{(j-i\vert 2)}_{\lambda^{(i)},\varnothing}(v_i/w_j|q,t)}
{\Nk^{(j-i\vert 2)}_{\lambda^{(i)},\lambda^{(j)}}(v_i/v_j|q,t)}
\cdot x_1^{|\lambda^{(1)}|_o+|\lambda^{(2)}|_e} x_2^{|\lambda^{(1)}|_e+|\lambda^{(2)}|_o},\nonumber
\end{align}
where $\bigl(\lambda^{(1)}, \lambda^{(2)}\bigr)$ is a fixed point of the toric action on the affine Laumon space.
The weight in the summation over the fixed points is
\smash{\raisebox{-0.5pt}{$x_1^{|\lambda^{(1)}|_o+|\lambda^{(2)}|_e} x_2^{|\lambda^{(1)}|_e+|\lambda^{(2)}|_o}$}},
where $|\lambda|_o = \sum_{k\geq 1} \lambda_{2k-1}$ and $|\lambda|_e = \sum_{k\geq 1} \lambda_{2k}$.
The expansion parameters $(x_1, x_2)$ are related to the physical parameters~$(\Lambda,x)$~by
\begin{gather}\label{expansion-special}
x_1 = -\frac{\sqrt{Qd_1d_2}}{\kappa}x, \qquad x_2 = - \sqrt{\frac{d_3d_4}{q^2Q}}\frac{\Lambda}{x}.
\end{gather}
By the relations \eqref{d-variables1} and \eqref{d-variables2}, this specialization gives
$x_1 = - t^{1/2} T_1^{1/2} T_2^{1/2}x$ and $x_2= - t^{1/2} T_3^{1/2} T_4^{1/2} \frac{\Lambda}{x}$.
We also employ the following specialization of the spectral parameters
with $\kappa= t^{-\frac{1}{2}}$:
\beq\label{Yamada}
u_1 = \frac{qQ}{d_3}, \quad u_2 = \frac{\kappa q}{d_1}, \qquad
v_1 = 1, \quad v_2 = \frac{Q}{\kappa}, \qquad
w_1 = \frac{1}{d_2}, \quad w_2 = \frac{Q}{d_4 \kappa}.
\eeq
One can check that these specializations agree those of $\mathcal{F}^{(1)}$ defined in \cite[Section~6]{Awata:2022idl}.
The overall scaling of parameters by $Q^{1/2}$ is necessary for the matching of the parameters $v_i$
of $\mathcal{Z}_{\mathrm{AL}}$ and $\mathcal{F}^{(1)}$.
After the same scaling the specialization for the function $\mathcal{F}^{(1)}$ is
\begin{equation}
u_1 = q^{1/2} \kappa^{-1} Q T_3, \quad u_2 = q^{1/2} T_1^{-1}, \qquad
w_1 = q^{-1/2} \kappa^{-1} Q T_2, \quad w_2 = q^{-1/2} T_4^{-1}.
\end{equation}
By substituting \eqref{d-variables2}, we can see the agreement with \eqref{Yamada}.

There are four possibilities of a specialization of the spectral parameters $u_i$ and $w_i$
due to the symmetry $d_1\leftrightarrow d_2$, $d_3 \leftrightarrow d_4$ of the Hamiltonian $\SS$.
But the symmetry is broken by choosing the specialization \eqref{Yamada}.
The
 are paired into two ${\rm SU}(2)$ doublets; $(d_1,d_3)$ and $(d_2,d_4)$
in the case of the affine Laumon space.
Such a rearrangement of ${\rm SU}(2)$ doublets of mass parameters between the five point conformal block with a degenerate field insertion and
the four point current block was already observed in the four-dimensional theory (see \cite[equation~(5.33)]{AFKMY}).
From the viewpoint of the orbifold coloring of four mass parameters, the decomposition into the pairs~$(d_1,d_2)$ and $(d_3, d_4)$
is by the parity of $\mathbb{Z}_2$ coloring. On the other hand, the decomposition into~$(d_1,d_3)$ and $(d_2, d_4)$ corresponds to
the fundamental and the anti-fundamental representations of ${\rm SU}(2)$ gauge symmetry.

Let us introduce a notation for the length of the columns of the Young diagrams,
\begin{gather}
\bigl(\lambda^{(1)}\bigr)^\vee = (\ell_1, \ell_2, \dots),
\qquad
\bigl(\lambda^{(2)}\bigr)^\vee = (k_1, k_2, \dots).
\end{gather}
Omitting the normalization factors, which are the last factors in \eqref{N-factor-even} and \eqref{N-factor-odd},
we find the following contributions to the partition function \eqref{fixedpoint}:
\begin{enumerate}\itemsep=0pt
\item
Fundamental and anti-fundamental matter contribution
\begin{gather}
 \prod_{i,j=1}^2
\Nk^{(j-i\vert 2)}_{\varnothing,\lambda^{(j)}}(u_i/v_j|q,t)
\Nk^{(j-i\vert 2)}_{\lambda^{(i)},\varnothing}(v_i/w_j|q,t)
\propto \prod_{i=1}^\infty \frac{\bigl[d_2 q^{i-1} t^{1-\floor{\frac{\ell_i + 1}{2}}} ; t\bigr]_\infty}
{\bigl[d_4^{-1}Q q^{1-i}t^{1+\floor{\frac{\ell_i}{2}}} ; t\bigr]_\infty } \nonumber\\
 \qquad{}\times \prod_{i=1}^\infty
\frac{ \bigl[d_3 q^{i-1} t^{-\floor{\frac{k_i-1}{2}}}; t\bigr]_\infty
\bigl[Q^{-1}d_3 q^{i-1} t^{-\floor{\frac{\ell_i}{2}}} ; t\bigr]_\infty
\bigl[Q d_2 q^{i-1}t^{1-\floor{\frac{k_i}{2}}}; t\bigr]_\infty}
{\bigl[d_1^{-1} Q^{-1} q^{1-i}t^{\floor{\frac{k_i}{2}}} ; t\bigr]_\infty
\bigl[d_1^{-1} q^{1-i}t^{\floor{\frac{\ell_i + 1 }{2}}} ; t\bigr]_\infty
\bigl[d_4^{-1} q^{1-i}t^{1+\floor{\frac{k_i-1}{2}}} ; t\bigr]_\infty}.
\end{gather}
\item
Vector multiplet contribution
\begin{gather}
\Biggl( \prod_{i,j=1}^2 \Nk^{(j-i\vert 2)}_{\lambda^{(i)},\lambda^{(j)}}(v_i/v_j|q,t) \Biggr)^{-1}
\!\!\!\propto
\!\prod_{i,j=1}^\infty
\frac
{\bigl[q^{j-i-1} t^{1+ \floor{\frac{\ell_i - \ell_j}{2}}} ; t\bigr]_\infty}
{\bigl[q^{j-i} t^{1+ \floor{\frac{\ell_i - \ell_j}{2}}} ; t\bigr]_\infty}
\!\prod_{i,j=1}^\infty
\frac
{\bigl[q^{j-i-1} t^{1+ \floor{\frac{k_i - k_j}{2}}} ; t\bigr]_\infty}
{\bigl[q^{j-i} t^{1+ \floor{\frac{k_i -k_j}{2}}} ; t\bigr]_\infty}
\nonumber \\
\qquad{}\times
\prod_{i,j=1}^\infty
\frac
{\bigl[Q^{-1}q^{j-i-1} t^{1+\floor{\frac{k_i -\ell_j -1 }{2}}} ; t\bigr]_\infty}
{\bigl[Q^{-1}q^{j-i} t^{1+\floor{\frac{k_i -\ell_j -1}{2}}} ; t\bigr]_\infty}
\prod_{i,j=1}^\infty
\frac
{\bigl[Qq^{j-i-1} t^{1+ \floor{\frac{\ell_i - k_j +1}{2}}} ; t\bigr]_\infty}
{\bigl[Qq^{j-i} t^{1+ \floor{\frac{\ell_i - k_j +1}{2}}} ; t\bigr]_\infty}.
\end{gather}
\end{enumerate}
We note that for a given pair of Young diagrams $\bigl(\lambda^{(1)}, \lambda^{(2)}\bigr)$, the normalization factor reduces
the formal infinite products of ratio of $[u;t]_\infty$ in the above formulas to finite products due to the cancellations
between the numerator and the denominator for sufficiently large $i$ and $j$. To see it,
note that we have $\ell_i=0$ for \smash{$\mm := \bigl(\lambda^{(1)}\bigr)_1 <i$} and $k_j=0$ for \smash{$\nn := \bigl(\lambda^{(2)}\bigr)_1 <j$}.
Later we will make a tuning of mass parameters so that we have $\mm=m$ and $\nn=n$ uniformly for all $\bigl(\lambda^{(1)}, \lambda^{(2)}\bigr)$
that contribute to the partition function.
In terms of the variables\footnote{Note that
$\floor{\frac{\ell_i}{2}}$ and $\floor{\frac{k_j}{2}}$ count the vertical dominos of length 2
in the Young diagrams $\lambda$ and $\mu$, respectively.}
\begin{equation}
z_i = q^{1-i} t^{\floor{\frac{\ell_i}{2}}}, \qquad w_j = q^{1-j} t^{\floor{\frac{k_j}{2}}},
\end{equation}
the finite product form of the matter contribution is
\beq
\prod_{i=1}^\mm \frac{\bigl[d_2 z_i^{-1} t^{1-(\ell_i)} ; t\bigr]_\infty
\bigl[Q^{-1}d_3 z_i^{-1}; t\bigr]_\infty}
{\bigl[d_4^{-1}Q z_i t; t\bigr]_\infty
\bigl[d_1^{-1} z_i t^{(\ell_i)} ; t\bigr]_\infty}
\prod_{j=1}^\nn
\frac{\bigl[d_3 w_j^{-1}t^{1-(k_j)}; t\bigr]_\infty
\bigl[Q d_2 w_j^{-1} t; t\bigr]_\infty}
{\bigl[d_1^{-1} Q^{-1} w_j ; t\bigr]_\infty
\bigl[d_4^{-1} w_j t^{(k_j)}; t\bigr]_\infty},
\eeq
where we have used the identity
\beq\label{floor-diff}
\biggl\lfloor \frac{K-L}{2}\biggr\rfloor = \biggl\lfloor\frac{K}{2}\biggr\rfloor- \biggl\lfloor\frac{L}{2}\biggr\rfloor + \{ (K)-1 \}\cdot (L), \qquad K,L \in \mathbb{Z}
\eeq
by the specialization $L=\pm1$. Here $(K)$ denotes the parity of an integer $K$; $(K)=1$ for an odd integer and $(K)=0$, otherwise.

For the vector multiplet contribution, we have to make
an appropriate decomposition of $(i,j) \in \mathbb{N} \times \mathbb{N}$ into four regions $R_{\mathrm{I}},\dots, R_{\mathrm{IV}}$
according to the lengths of the first row of the Young diagrams; $\mm = \bigl(\lambda^{(1)}\bigr)_1$ and $\nn = \bigl(\lambda^{(2)}\bigr)_1$.
For example, for a pair $(\ell_i, k_j)$, we define
\begin{alignat}{3}
& R_{\mathrm{I}}= \{ 1 \leq i \leq \mm,\, 1 \leq j \leq \nn \}, &&\qquad R_{\mathrm{II}} = \{ \mm+1 \leq i < \infty ,\, 1 \leq j \leq \nn \},& \nonumber \\
 & R_{\mathrm{III}} = \{ 1 \leq i \leq \mm ,\, \nn+1 \leq j < \infty \},&& \qquad R_{\mathrm{IV}} = \{ \mm+1 \leq i < \infty ,\, \nn+1 \leq j < \infty \},\hspace*{-10mm}& \label{four-regions}
\end{alignat}
and for a pair $(\ell_i, \ell_j)$ we set $\mm=\nn$ in the decomposition \eqref{four-regions}.
The vector multiplet contribution becomes trivial only for $R_{\mathrm{IV}}$.
We find the vector multiplet contributions are
\begin{enumerate}\itemsep=0pt
\item[(I)]
From the region $R_{\mathrm{I}}$,
\begin{gather}
\prod_{i \neq j=1}^\mm
\frac{\bigl[(tz_i/qz_j)~t^{\{(\ell_i)-1\}\cdot(\ell_j)} ; t\bigr]_\infty}
{\bigl[(tz_i/z_j)~t^{\{(\ell_i)-1\}\cdot(\ell_j)} ; t\bigr]_\infty}
\prod_{i \neq j=1}^\nn
\frac{\bigl[(tw_i/qw_j)~t^{\{(k_i)-1\}\cdot(k_j)} ; t\bigr]_\infty}
{\bigl[(tw_i/w_j)~t^{\{(k_i)-1\}\cdot(k_j)} ; t\bigr]_\infty}
\nonumber \\
\qquad{}\times
\prod_{i=1}^\mm \prod_{j=1}^\nn
\frac{\bigl[\bigl(Q^{-1} w_j/qz_i\bigr)~t^{(k_j)\cdot\{1-(\ell_i)\}}; t\bigr]_\infty}
{\bigl[\bigl(Q^{-1} w_j/z_i\bigr)~t^{(k_j)\cdot\{1-(\ell_i)\}} ; t\bigr]_\infty}
\frac{\bigl[(Q tz_i/qw_j)~t^{(\ell_i)\cdot\{1-(k_j)\}} ; t\bigr]_\infty}
{\bigl[(Q tz_i/w_j)~t^{(\ell_i)\cdot\{1-(k_j)\}} ; t\bigr]_\infty}.
\end{gather}
Here and henceforth, we use a shorthand notation $\prod_{i \neq j=1}^\mm$
for the sum over $1 \leq i,j \leq \mm$ without the diagonal part $i=j$.
\item[(II)]
For two semi-infinite regions $R_{\mathrm{II}}$ and $R_{\mathrm{III}}$, only the ``boundary'' contributions remain,
\begin{gather}
\prod_{i=1}^\mm
\frac{\bigl[q^{\mm-i} t^{1+ \floor{\frac{\ell_i}{2}}} ; t\bigr]_\infty}{\bigl[q^{i-\mm-1} t^{1+ \floor{-\frac{\ell_i}{2}}}; t\bigr]_\infty}
\frac{\bigl[Q q^{\nn-i} t^{1+ \floor{\frac{\ell_i +1}{2}}} ; t\bigr]_\infty}{\bigl[Q^{-1} q^{i-\nn-1} t^{\floor{\frac{1-\ell_i}{2}}} ; t\bigr]_\infty} \nonumber \\
\qquad\quad{}\times\prod_{j=1}^\nn
\frac{\bigl[q^{\nn-j} t^{1+ \floor{\frac{k_j}{2}}} ; t\bigr]_\infty}{\bigl[q^{j-\nn-1} t^{1+\floor{\frac{-k_j}{2}}}; t\bigr]_\infty}
\frac{\bigl[Q^{-1} q^{\mm-j} t^{1+\floor{\frac{k_j-1}{2}}} ; t\bigr]_\infty}{\bigl[Q q^{j-\mm-1} t^{1+\floor{\frac{1-k_j}{2}}} ; t\bigr]_\infty} \nonumber \\
\qquad{}=
\prod_{i=1}^\mm
\frac{\bigl[z_i q^{\mm-1} t ; t\bigr]_\infty}{\bigl[z_i^{-1} q^{-\mm} t^{1-(\ell_i)}; t\bigr]_\infty}
\frac{\bigl[Q z_i q^{\nn-1} t^{1+ (\ell_i)} ; t\bigr]_\infty}{\bigl[Q^{-1} z_i^{-1} q^{-\nn}; t\bigr]_\infty} \nonumber \\
\qquad\quad{}\times\prod_{j=1}^\nn
\frac{\bigl[w_j q^{\nn-1} t; t\bigr]_\infty}{\bigl[w_j^{-1} q^{-\nn}t^{1-(k_j)}; t\bigr]_\infty}
\frac{\bigl[Q^{-1} w_j q^{\mm-1} t^{(k_j)}; t\bigr]_\infty}{\bigl[Q w_j^{-1} q^{-\mm}t; t\bigr]_\infty},
\end{gather}
where we have used \eqref{floor-diff}.
\end{enumerate}

In summary, up to a $z_i$, $w_j$ independent normalization factor, we can rewrite the Nekrasov partition function
as a sum over the positive cone of a lattice with the following weight function:%
\begin{gather}
 W_{\mm+\nn}(z) \nonumber \\
=
\prod_{i=1}^\mm
\frac{\bigl[d_2 z_i^{-1} t^{1-(\ell_i)} ; t\bigr]_\infty
\bigl[Q^{-1}d_3 z_i^{-1}; t\bigr]_\infty}
{\bigl[d_4^{-1}Q z_i t; t\bigr]_\infty \bigl[d_1^{-1} z_i t^{(\ell_i)} ; t\bigr]_\infty}
\frac{\bigl[z_i q^{\mm-1} t ; t\bigr]_\infty}{\bigl[z_i^{-1} q^{-\mm} t^{1-(\ell_i)}; t\bigr]_\infty}
\frac{\bigl[Q z_i q^{\nn-1} t^{1+ (\ell_i)} ; t\bigr]_\infty}{\bigl[Q^{-1} z_i^{-1} q^{-\nn}; t\bigr]_\infty} \nonumber \\
\quad{}\times\prod_{j=1}^\nn
\frac{\bigl[Q^{-1} d_3 z_{\mm+j}^{-1}; t\bigr]_\infty
\bigl[d_2 z_{\mm+j}^{-1} t^{(k_i)}; t\bigr]_\infty}
{\bigl[d_1^{-1} z_{\mm+j} t^{1-(k_j)} ; t\bigr]_\infty
\bigl[d_4^{-1} Q z_{\mm+j} t; t\bigr]_\infty}
\frac{\bigl[Q z_{\mm+j} q^{\nn-1} t^{2-(k_j)}; t\bigr]_\infty}{\bigl[Q^{-1} z_{\mm+j}^{-1} q^{-\nn}; t\bigr]_\infty}
\frac{\bigl[z_{\mm+j} q^{\mm-1} t; t\bigr]_\infty}{\bigl[z_{\mm+j}^{-1} q^{-\mm} t^{(k_j)}; t\bigr]_\infty} \nonumber \\
\quad{}\times\prod_{i \neq j=1}^\mm
\frac{\bigl[(tz_i/qz_j)~t^{\{(\ell_i)-1\}\cdot(\ell_j)} ; t\bigr]_\infty}
{\bigl[(tz_i/z_j)~t^{\{(\ell_i)-1\}\cdot(\ell_j)} ; t\bigr]_\infty}
\prod_{i \neq j=1}^\nn
\frac{\bigl[(t z_{\mm+i}/q z_{\mm+j})~t^{(k_i)\cdot\{(k_j)-1\}} ; t\bigr]_\infty}
{\bigl[(t z_{\mm+i}/z_{\mm+j})~t^{(k_i)\cdot\{(k_j)-1\}} ; t\bigr]_\infty}
\nonumber \\
\quad{}\times\prod_{i=1}^\mm \prod_{j=1}^\nn
\frac{\bigl[(tz_{\mm+j}/qz_i)~t^{-(k_j)\cdot(\ell_i)}; t\bigr]_\infty}
{\bigl[(tz_{\mm+j}/z_i)~t^{-(k_j)\cdot(\ell_i)} ; t\bigr]_\infty}
\frac{\bigl[(tz_i/q z_{\mm+j})~t^{\{(\ell_i)-1\}\cdot\{1-(k_j)\}} ; t\bigr]_\infty}
{\bigl[(tz_i/z_{\mm+j})~t^{\{(\ell_i)-1\}\cdot\{1-(k_j)\}} ; t\bigr]_\infty},\label{Jackson-weight}
\end{gather}
where we have defined\footnote{See \eqref{floor-diff}.} $z_{\mm+j} = Q^{-1} w_j t^{(k_j)-1}= Q^{-1} q^{1-j} t^{\floor{\frac{k_j-1}{2}}}$.

In the above computations, we have used the fact that
that $\ell_i=0$ for $\mm<i$ and $k_j=0$ for~${\nn<j}$, where $\mm$ and $\nn$ depend on $\bigl(\lambda^{(1)}, \lambda^{(2)}\bigr)$.
But as is shown in Appendix~\ref{App:Tuning}, by tuning some of mass parameters $d_2=q^{-m}$ and $d_3=q^{-n}$,
we can assume $\mm=m$ and $\nn=n$ for any pair of Young diagrams $\bigl(\lambda^{(1)}, \lambda^{(2)}\bigr)$, because
a pair $\bigl(\lambda^{(1)}, \lambda^{(2)}\bigr)$ with $\mm >m$ or $\nn >n$ does not contribute to the partition function.
Furthermore, the substitution of the mass tuning condition $d_2=q^{-m}$ and $d_3=q^{-n}$ leads nice cancellation in
the weight function \eqref{Jackson-weight}.
By defining $(\ell_{m+j})= 1 - (k_j)$, we can make $W_{m,n}(z,w)$ completely symmetric in $z$ and $w$,
\begin{gather}
W_{m+n}(z)
=\prod_{I=1}^{m+n}
\frac{\bigl[z_I q^{m-1} t ; t\bigr]_\infty \bigl[Q z_I q^{n-1} t^{1+ (\ell_I)} ; t\bigr]_\infty}
{\bigl[d_4^{-1}Q z_I t; t\bigr]_\infty \bigl[d_1^{-1} z_I t^{(\ell_I)} ; t\bigr]_\infty} \nonumber\\ \hphantom{W_{m+n}(z)=}{}
\times{}\prod_{I \neq J=1}^{m+n}
\frac{\bigl[(tz_I/qz_J)~t^{\{(\ell_I)-1\}\cdot (\ell_J)} ; t\bigr]_\infty}
{\bigl[(tz_I/z_J)~t^{\{(\ell_I)-1\}\cdot (\ell_J)} ; t\bigr]_\infty}.\label{symmetric-weight}
\end{gather}
In the expansion of Nekrasov partition function,
the contribution from a fixed point $\bigl(\lambda^{(1)}, \lambda^{(2)}\bigr)$ has
the weight \smash{\raisebox{-1pt}{$x_1^{|\lambda^{(1)}|_o +|\lambda^{(2)}|_e} x_2^{|\lambda^{(1)}|_e +|\lambda^{(2)}|_o}$}},
where
\beq
|\lambda|_o = \sum_{k\geq 1} \lambda_{2k-1} = \sum_{k\geq 1} \biggl\lfloor\frac{\lambda_k^\vee +1}{2}\biggr\rfloor,
\qquad
|\lambda|_e = \sum_{k\geq 1} \lambda_{2k} = \sum_{k\geq 1} \biggl\lfloor\frac{\lambda_k^\vee}{2}\biggr\rfloor.
\eeq
Recall that we have specialized $x_1$ and $x_2$ as \eqref{expansion-special}, which implies that
the dependence of the weight on $\Lambda$ and $x$ is
\beq
\Lambda^{|\lambda^{(1)}|_e+|\lambda^{(2)}|_o} x^{|\lambda^{(1)}|_o- |\lambda^{(1)}|_e
+|\lambda^{(2)}|_e - |\lambda^{(2)}|_o}.
\eeq
Since $|\lambda|_o - |\lambda|_e = \sum_{k\geq 1} \bigl(\lambda_k^\vee\bigr)$,
the power of $x$, which is identified with the ${\rm SU}(2)$ spin variable,~is%
\beq\label{eq:2s}
2s = \sum_{i=1}^m (\ell_i) - \sum_{j=1}^n (k_j) = \sum_{I=1}^{m+n} (\ell_I) -n.
\eeq
Hence, only columns with odd length can produce a non-vanishing power of $x$. We see that
the range of the power $p$ of $x$ is $-n \leq p \leq m$ (see also Figure \ref{Fig:mass-truncation}).
Similarly, the power of $\Lambda$, which we identify as the instanton number, is
\beq
\sum_{i=1}^m \biggl\lfloor\frac{\ell_i}{2}\biggr\rfloor + \sum_{j=1}^n \biggl\lfloor\frac{k_j +1}{2}\biggr\rfloor
= \sum_{i=1}^m \biggl\lfloor\frac{\ell_i}{2}\biggr\rfloor + \sum_{j=1}^n \biggl\lfloor\frac{k_j }{2}\biggr\rfloor + \sum_{j=1}^n (k_j).
\eeq

Now we are ready to prove the following theorem.

\begin{Theorem}\label{AL=J}
Under the identification of the parameters
\begin{gather}
d_1=\frac{1}{q^{m-1}a_1b_1}, \qquad
d_2=q^{-m}, \qquad
d_3=q^{-n}, \qquad
d_4=\frac{1}{q^{n-1}a_2b_2}, \nonumber\\
Q=\frac{q^{m-n}a_1}{t a_2},\label{eq:dQ-ab}
\end{gather}
the affine Laumon partition function $\mathcal{Z}_{\rm AL}$ \eqref{fixedpoint} coincides with
the $\bbC^{N+1}$-valued function \eqref{take-Matsuo} with the integration cycle $\xi$ of \eqref{lattice-truncate}.
More precisely, the $N+1$ components of $\mathcal{Z}_{\rm AL}$ with respect to the expansion parameter $x$
are identified as the Jackson integrals $\langle e_k(a_2, b_1), \xi \rangle$ with respect to the Matsuo basis $e_k$, $0 \leq k \leq N=m+n$.
\end{Theorem}
\begin{proof}
Rewriting the dummy indices in \eqref{symmetric-weight}, we consider
\beq
W_{N}(z)
= \prod_{i=1}^N
\frac{\bigl[z_i q^{m-1} t ; t\bigr]_\infty \bigl[Q z_i q^{n-1} t^{1+ (\ell_i)} ; t\bigr]_\infty}
{\bigl[d_4^{-1}Q t z_i ; t\bigr]_\infty \bigl[d_1^{-1} z_i t^{(\ell_i)} ; t\bigr]_\infty}
\prod_{i \neq j=1}^N
\frac{\bigl[(tz_i/qz_j)~t^{\{(\ell_i)-1\}\cdot(\ell_j)} ; t\bigr]_\infty}
{\bigl[(tz_i/z_j)~t^{\{(\ell_i)-1\}\cdot(\ell_j)} ; t\bigr]_\infty},
\eeq
where $N=m+n$. By the parameter relation \eqref{eq:dQ-ab}, it can be written as
\begin{gather}
W_{N}(z)
= \prod_{i=1}^N
\frac{\bigl[z_i q^{m-1} t ; t\bigr]_\infty \bigl[q^{m-1} \frac{a_1}{a_2} z_i t^{(\ell_i)} ; t\bigr]_\infty}
{\bigl[q^{m-1}a_1b_2 z_i ; t\bigr]_\infty \bigl[q^{m-1}a_1 b_1z_i t^{(\ell_i)} ; t\bigr]_\infty} \nonumber\\ \hphantom{W_{N}(z)=}{}
\times\prod_{i \neq j=1}^N
\frac{\bigl[(tz_i/qz_j) t^{\{(\ell_i)-1\}\cdot(\ell_j)} ; t\bigr]_\infty}
{\bigl[(tz_i/z_j) t^{\{(\ell_i)-1\}\cdot(\ell_j)} ; t\bigr]_\infty}.
\end{gather}
Rescaling the integration variables as $z_i \to z_i /\bigl(q^{m-1}a_1\bigr)$, we have
\beq
W_{N}(z)
= \prod_{i=1}^N
\frac{\bigl[\frac{z_i}{a_1} t ; t\bigr]_\infty \bigl[\frac{z_i}{a_2} t^{(\ell_i)} ; t\bigr]_\infty}
{[b_2 z_i ; t]_\infty \bigl[b_1z_i t^{(\ell_i)} ; t\bigr]_\infty}
\prod_{i \neq j=1}^N
\frac{\bigl[(tz_i/qz_j)~t^{\{(\ell_i)-1\}\cdot(\ell_j)} ; t\bigr]_\infty}
{\bigl[(tz_i/z_j)~t^{\{(\ell_i)-1\}\cdot(\ell_j)} ; t\bigr]_\infty}.
\eeq
We decompose the index set as
\beq
I \cup J = \{ 1, 2,\dots, N \}, \qquad
I= \{ i \vert (\ell_i)=0 \}, \quad J= \{ j \vert (\ell_j)=1 \}.
\eeq
From \eqref{eq:2s}, we see that $2s +n = |J|$. Then
\beq
W_{N}(z) = W_{N}^{(0)}(z) P_J(z),
\eeq
where
\begin{align}
&W^{(0)}_{N}(z)
= \prod_{i=1}^N
\frac{\bigl[\frac{z_i}{a_1} t ; t\bigr]_\infty \bigl[\frac{z_i}{a_2}t ; t\bigr]_\infty}
{[b_2 z_i ; t]_\infty [b_1z_i ; t]_\infty}
\prod_{i \neq j=1}^N
\frac{[tz_i/qz_j ; t]_\infty}
{[tz_i/z_j; t]_\infty},\\
&P_J(z)=\prod_{j \in J} \frac{1 -b_1 z_j}{1- \frac{z_j}{a_2}} \times
\prod_{i \in I} \prod_{j \in J} \frac{z_j - q^{-1} z_i}{z_j -z_i}.
\end{align}
Putting $z_i=\xi_i t^{\nu_i}$, $i=1, \dots, N$,
we can rewrite the sum of $W_N(z)$ over the partitions~$\mu$ and~$\lambda$
as the sum of $W_{N}(z) = W_{N}^{(0)}(z) P_J(z)$ over $\nu_i \in \bbZ_{\geq 0}^{N}$ and $J \subset \{1,2,\dots, N\}$,
where the correspondence is
\begin{align}
\mu_i^{\vee}&=\begin{cases}
2 \nu_i, & i \in J,\\
2 \nu_i-1, &i \notin J,
\end{cases} \qquad i=1, \dots, n, \\
\lambda_{i}^{\vee}&=\begin{cases}
2 \nu_{i+n}+1, & i+n \in J,\\
2 \nu_{i+n}, &i+n \notin J,
\end{cases} \qquad i=1, \dots, m.
\end{align}
Then the latter sum can be identified with the Jackson integral studied in Section \ref{sec:duality}.
In fact, from \eqref{common} and \eqref{eq:diff-prod}, we see that $W^{(0)}(z)= \Delta(z,1)\Phi(z)$.
The sum of $P_J(z)$ for $|J|=k$ and the $k$-th cocycle factor ${\hat e}_k$ \eqref{eq:e-factor} is related by
\begin{equation}
\sum_{\substack{J \subset \{1,2,\dots, N\} \\ |J|=k}} P_J(z)=c_k {\hat e}_k(z) \prod_{i=1}^N \frac{1}{1-\frac{z_i}{a_2}},
\end{equation}
where $c_k=q^{(k-m)(k-n)/2}\bigl[{N \atop k}\bigr]_{q^{-1}}/\bigl[{N \atop m}\bigr]_{q^{-1}}$.
\end{proof}

\section{Four-dimensional limit and KZ equation}
\label{sec:4d}

In \cite{Awata:2022idl}, we computed the four-dimensional limit of the Hamiltonian $\SS$ and
confirmed that Shakirov's non-stationary difference equation has the correct four-dimensional limit.
In Section~\ref{sec:q-KZ}, we have shown the representation matrix of $\SS$ is nothing but the $R$-matrix of
\smash{$\qg\bigl(A_1^{(1)}\bigr)$} with generic spins. Let us work out the four-dimensional limit of the $R$-matrix.

\begin{Proposition}
Under the limit $q=e^h$, $d_i=q^{m_i}$, $h\to 0$, we have
\begin{equation}
R=1+ h R^{(1)}+O\bigl(h^2\bigr),
\end{equation}
where $R^{(1)}$ is a tridiagonal matrix, which is generically infinite-dimensional.
Furthermore, it can be truncated to a finite matrix of size $m+n+1$,
if $m_1$ $(\text{or}\ m_2)=-m$ and $m_3$ $(\text{or}\ m_4)=-n$ for $m,n \in \bbZ_{\geq 0}$.
\end{Proposition}
For example, in case of $m_1=-2$, $m_3=-1$, we have
\[
R^{(1)}=\begin{bmatrix}
 -\frac{3 (\Lambda m_2-\Lambda )}{\Lambda -1} & \frac{3 (m_2-1 )}{\Lambda -1} & 0 & 0 \vspace{1mm}\\
 \frac{\Lambda m_4}{\Lambda -1} & -\frac{2 \Lambda m_2+\Lambda m_4}{\Lambda -1} & \frac{2 m_2}{\Lambda -1} & 0 \vspace{1mm}\\
 0 & \frac{2 \Lambda (m_4-1 )}{\Lambda -1} & -\frac{\Lambda +\Lambda m_2+2 \Lambda m_4-2}{\Lambda -1} & -\frac{-m_2-1}{\Lambda -1} \vspace{1mm}\\
 0 & 0 & \frac{3 (\Lambda m_4-2 \Lambda )}{\Lambda -1} & -\frac{3 (\Lambda m_4-2 )}{\Lambda -1}
\end{bmatrix}.
\]

\begin{proof}
Under the limit $q=e^h$, $h\to 0$, we have
\begin{align}
\frac{(-q^\alpha x;q)_\infty}{(-q^\beta x;q)_\infty}
&=(1+x)^{\beta-\alpha}\biggl\{1-\frac{h}{2}(\alpha-\beta)(\alpha+\beta-1)\frac{x}{x-1}+O\bigl(h^2\bigr)\biggr\}.
\end{align}
Then the defining relation of the matrix $R=(r_{i,j})$
\begin{gather}
q^{i(i+1)/2} x^i\frac{\bigl(-q^i d_1d_2 x;q\bigr)_\infty}{(-d_2 x;q)_\infty}
\frac{\bigl(-q^{-i} d_3d_4 \frac{\Lambda}{x};q\bigr)_\infty}{\bigl(-d_4 \frac{\Lambda}{x};q\bigr)_\infty}\nonumber\\
\qquad{}=\sum_{j} r_{i,j}
q^{-j(j+1)/2}x^j
\frac{(-d_1 x;q)_\infty}{\bigl(-q^{-j} x;q\bigr)_\infty}
\frac{\bigl(-d_3 \frac{\Lambda}{x};q\bigr)_\infty}{\bigl(-q^{j} \frac{\Lambda}{x};q\bigr)_\infty},
\end{gather}
can be written as
\begin{equation}
A_i=\sum_{j}\biggl(\frac{x+\Lambda}{x+1}\biggr)^{j-i} B_jr_{i,j},
\end{equation}
where up to $O\bigl(h^2\bigr)$,
\begin{gather}
A_i=1+\frac{1}{2} h \biggl\{-\frac{\Lambda (i-m_3) (i-m_3-2 m_4+1)}{x+\Lambda}\nonumber\\ \hphantom{A_i=+\frac{1}{2} h \biggl\{-}{}
-\frac{x (i+m_1)(i+m_1+2m_2-1)}{x+1}+i (i+1)\biggr\},
\notag \\
B_j=1+\frac{1}{2} h \biggl\{\frac{\Lambda (j-m_3) (j+m_3-1)}{x+\Lambda}+\frac{x (j-m_1+1)(j+m_1)}{x+1}-j (j+1)\biggr\}.
\end{gather}
We put $r_{i,j}=r^{(0)}_{i,j}+h r^{(1)}_{i,j}+O\bigl(h^2\bigr)$.
From the leading term, we have
\begin{equation}
\sum_{j} \biggl(\frac{x+\Lambda}{x+1}\biggr)^{j-i} r^{(0)}_{i,j}
=1,
\end{equation}
hence
\begin{equation}
r^{(0)}_{i,j}=\delta_{i,j}.
\end{equation}
Then, from $O(h)$ terms, we obtain
\begin{equation}
\sum_{j} \biggl(\frac{x+\Lambda}{x+1}\biggr)^{j-i} r^{(1)}_{i,j}
=i(i+1)+\frac{x(i+m_1)(i+m_2)}{1+x}+\frac{\Lambda(i-m_3)(i-m_4)}{x+\Lambda},
\end{equation}
and the solution is given by
\begin{gather}
r^{(1)}_{i,i-1}=\frac{\Lambda (i-m_3) (i-m_4)}{\Lambda -1},\notag \\
r^{(1)}_{i,i}=-\frac{\Lambda \{(i+m_1) (i+m_2)+(i-m_3) (i-m_4)\}}{\Lambda -1}+i (i+1),\notag \\
r^{(1)}_{i,i+1}=\frac{(i+m_1) (i+m_2)}{\Lambda -1},\notag \\
r^{(1)}_{i,j}=0 \qquad \text{otherwise}.
\end{gather}
The truncation can also be seen from this expression.
\end{proof}


\subsection{Identification with the KZ equation}

The result given above is consistent with the four-dimensional limit of the operator $\SS$ (see \mbox{\cite[Theorem 5.6]{Awata:2022idl}}):
\begin{align}
&\SS=1+h H_{4d}+O\bigl(h^2\bigr),\\
\label{4DSS}
&H_{4d}=\vartheta_x(\vartheta_x+1)+\frac{\Lambda-x}{1-\Lambda}(\vartheta_x+m_1)(\vartheta_x+m_2)+
\frac{\Lambda}{x}\frac{x-1}{1-\Lambda}
(\vartheta_x-m_3)(\vartheta_x-m_4),
\end{align}
where $\vartheta_x=x \frac{\partial}{\partial x}$.
Using \eqref{4DSS} and defining $\varkappa$ and $a$ by
$t=q^{\varkappa}$, $Q=q^{a}$,\footnote{We have scaled the Coulomb parameter $a$ by $\epsilon_1$.}
we can see that the four-dimensional limit of the Shakirov's equation
\begin{equation}
\Psi(\Lambda,x)=\SS \Psi\biggl(\frac{\Lambda}{t},\frac{x}{qt Q}\biggr),
\end{equation}
takes the form that looks like the Knizhnik--Zamolodchikov equation
\begin{align}\label{eq:KZlike}
&\varkappa \frac{\partial}{\partial \Lambda}\Psi(\Lambda,x)
=\frac{H_{4d}-(\varkappa +1+ a)\vartheta_x}{\Lambda}\Psi(\Lambda,x)
=\biggl\{\frac{A_0}{\Lambda}+\frac{A_1}{\Lambda-1} \biggr\}\Psi(\Lambda,x),
\end{align}
where $A_0$, $A_1$ are operators acting on variable $x$:
\begin{align}
&A_0=\vartheta_x (\vartheta_x - \varkappa -a)
-x (\vartheta_x +m_1)(\vartheta_x +m_2),\\
&A_1=(x-1)(\vartheta_x +m_1) (\vartheta_x+m_2)
+\biggl(\frac{1}{x}-1\biggr)(\vartheta_x -m_3)(\vartheta_x -m_4).
\end{align}
Below, we will show the relation with the standard KZ equation more explicitly.

Following \cite{FZ}, we write the $\widehat{\mathfrak{sl}_2}$ KZ equation
for generic complex spins $j_a$ and level $k$, $\varkappa =k+2 =\log t/\log q$, in the form
\begin{gather}
\varkappa \frac{\partial}{\partial z_a} \psi=\sum_{b(\neq a)=1}^n \frac{\Omega_{a,b}}{z_a-z_b} \psi,
\qquad a=1, \dots, n , \label{eq:KZa}
\\
\Omega_{a,b}=-(x_a-x_b)^2 \frac{\partial^2}{\partial x_a \partial x_b}+2(x_a-x_b) \biggl(j_a \frac{\partial}{\partial x_b}-
j_b \frac{\partial}{\partial x_a} \biggr)+2 j_a j_b.
\end{gather}
The number of variables $z_a$, $x_a$ can be reduced to $2(n-3)$, thanks to the
${\rm SL}(2, \bbC) \times {\rm SL}(2, \bbC)$ symmetry for $x_a$ and $z_a$,
\begin{align}\label{eq:SL2cons}
&\sum_{a=1}^n x_a^i \biggl\{x_a \frac{\partial}{\partial x_a}-j_a(i+1)\biggr\} \psi=0, \qquad
\sum_{a=1}^n z_a^i \biggl\{z_a \frac{\partial}{\partial z_a}+h_a(i+1)\biggr\} \psi=0,
\end{align}
where $i=0,\pm 1$ and $h_a=j_a(j_a+1)/\varkappa$.

We will consider the case $n=4$, where the conditions \eqref{eq:SL2cons} can be solved as
\begin{align}
&\psi=u v f_1(z,x),
\qquad z=\frac{z_{12}z_{34}}{z_{13}z_{24}}, \qquad
x=\frac{x_{12}x_{34}}{x_{13}x_{24}}, \\
&u=z_{12}^{-2 h_1}z_{23}^{h_1-h_2-h_3+h_4}z_{24}^{-h_1-h_2+h_3-h_4}z_{34}^{h_1+h_2-h_3-h_4},\notag \\
&v=x_{12}^{2 j_1}x_{23}^{-j_1+j_2+j_3-j_4}x_{24}^{j_1+j_2-j_3+j_4}x_{34}^{-j_1-j_2+j_3+j_4},\notag
\end{align}
where $z_{ab} = z_a - z_b$, $x_{ab} = x_a - x_b$.
Then the equations \eqref{eq:KZa} for $a=1, \dots, 4$ are all equivalent and we obtain a single equation for $f_1(z,x)$.
Through the gauge transformation
\begin{align}
f(z,x)=g f_1(z,x), \qquad g=x^{j_1+j_2}z^{-h_1-h_2}(z-1)^\frac{2 j_1j_4}{\varkappa},
\end{align}
the KZ equation for $f(z,x)$ can be written concisely as
\begin{gather}
\biggl\{ \varkappa \vartheta_z-\vartheta_x(\vartheta_x-1)+\frac{x-z}{1-z}(\vartheta_x-j_1+j_2)(\vartheta_x+j_3-j_4) \notag \\
\qquad{}+\frac{z(1-x)}{x(1-z)}(\vartheta_x+j_1+j_2)(\vartheta_x+j_3+j_4)\biggr\} f(z,x)=0,
\label{eq:KZf}
\end{gather}
where $\vartheta_z=z \frac{\partial}{\partial z}$.

In order to identify \eqref{eq:KZf} with the four-dimensional limit of Shakirov's equation,
we change the parameters as\footnote{Here $\tilde{a}=a+ \varkappa$.}
\begin{equation}\label{spin-mass}
(j_1,j_2, j_3, j_4)= \biggl(\frac{-m_1-m_3}{2},
 \frac{\tilde{a}-1+m_1-m_3}{2},
 \frac{\tilde{a}-1+m_2-m_4}{2},
\frac{-m_2-m_4}{2} \biggr),
\end{equation}
and make a gauge transformation
\begin{equation}
\Psi=x^{-m_3} (z-1)^{\frac{(m_1+m_3) (m_2+m_4)}{2 \varkappa}}
 z^{\frac{(m_3-m_1) \tilde{a}-(m_1-1) m_1-(m_3-1) m_3}{2 \varkappa}} f.
\end{equation}
Then the equation \eqref{eq:KZf} exactly agrees with \eqref{eq:KZlike}
\begin{gather}
\biggl\{ \varkappa \vartheta_z-\vartheta_x(\vartheta_x-\tilde{a})
+\frac{x-z}{1-z}(\vartheta_x+m_1)(\vartheta_x+m_2)\nonumber\\
\qquad{}+\frac{z(1-x)}{x(1-z)}(\vartheta_x-m_3)(\vartheta_x-m_4)\biggr\} \Psi=0\label{eq:4dSW}
\end{gather}
by the identification $z=\Lambda$.

There are $2^4$ such identifications. This is because the equation \eqref{eq:4dSW} (= the quantum $P_{\rm VI}$ equation),
has \smash{$(\bbZ_2)^4 \subset W\bigl(D_4^{(1)}\bigr)$} symmetry
\begin{align}
&\{m_1\rightarrow m_2, m_2 \rightarrow m_1 \} \times \{m_3 \rightarrow m_4, m_4 \rightarrow m_3 \} \times
\{m_1\to 1-\tilde{a}-m_2, m_2\to 1-\tilde{a}-m_1\} \nonumber \\
&\qquad{}\times \{m_3\to 1+ a-m_4, m_4\to 1+ a-m_3\},\label{D_4Weyl}
\end{align}
up to a gauge transformation of the form $g=x^{k_1}(x-1)^{k_2}z^{k_3}(z-1)^{k_4}(x-z)^{k_5}$.
Each $\bbZ_2$ transform in \eqref{D_4Weyl} corresponds to the Weyl reflection with respect to the outer nodes
of $D_4^{(1)}$ Dynkin diagram.

\begin{Remark}
The relation between the KZ equation and the quantum $P_{\rm VI}$ was studied by \cite{Nagoya}.
The derivation of the KZ equation from the gauge theory is given by \cite{Nekrasov:2021tik}.
The equation \eqref{eq:KZlike} was previously obtained from Virasoro five-point functions.
\end{Remark}


\appendix

\section[q-Borel transformation and refined Chern--Simons theory]{$\boldsymbol{q}$-Borel transformation and refined Chern--Simons theory}
\label{App:framing}

In this appendix, we would like to explain a part of motivation for the original non-stationary difference equation of \cite{Shakirov:2021krl}.
Namely, we would like to elucidate the reason for introduction of
the operator $\Bor$ in the form that it was introduced.

A part of motivation to study non-stationary difference equations comes from the most well-studied example,
which is the non-stationary Ruijsenaars--Schneider equation suggested by \cite{langmann2020basic}.
This non-stationary equation is satisfied by the $K$-theoretic character of the affine Laumon space
of type $\widehat{\mathfrak{gl}}_N$. The character in question is a special function generalizing the Macdonald polynomial.

For Macdonald polynomials, a relation to Chern--Simons theory and knot invariants is well-known \cite{aganagic2011knot}.
The relation is obtained through so-called refined Chern--Simons theory.
This theory provides some of the most prominent equations satisfied by Macdonald polynomials,
among them a key role is played by the integral identity,
\begin{gather}
\oint_{|x_1|=1}\dfrac{{\rm d}x_1}{x_1} \dots \oint_{|x_N|=1}\dfrac{{\rm d}x_N}{x_N}
 \exp \Biggl(- \displaystyle{\sum_{i=1}^{N}} \dfrac{x_i^2}{2g}\Biggr)
 P_{\lambda}(x_1, \dots, x_N) \prod\limits_{i \neq j =1}^{N} \prod\limits_{m = 0}^{\infty}
 \dfrac{(1 - q^m x_i/x_j)}{(1 - t q^m x_i/x_j)} \nonumber \\
\qquad{}= \mbox{const}\, T_{\lambda}~P_{\lambda}(t^{\rho_1}, \dots, t^{\rho_N}).
\end{gather}
Here $\rho_i = (N+1)/2 - i$ is the Weyl vector of $\mathfrak{gl}_N$, the partition $\lambda$ is arbitrary, and $T_\lambda$
is a certain product of $q$, $t$ to the power of quadratic form of~$\lambda$.

In \cite[Appendix B]{aganagic2011knot}, it is explained that changing the variables via $u_i = \log(x_i)$,
it is possible to trade the Gaussian exponents ${\rm e}^{-x_i^2/2g}$ under the integration sign
for theta functions $\sum_n q^{n^2/2} u_i^n$.
This can be understood simply as the result of action on the rest of the integrand by a product of operators
\begin{align}
\Bor^{\prime} u^n = q^{n^2/2} u^n,
\end{align}
or equivalently (see \eqref{q-Borel}),
\begin{align}
\Bor u^n = q^{n^2/2+n/2} u^n
\end{align}
in each variable $u_i$, and setting all $u_1=\dots=u_N=1$ afterwards.
The linear difference between operators $\Bor^{\prime}$ and $\Bor$ is inessential and amounts to a shift of variable.

In our view, equations for more complicated functions generalizing the Macdonald polynomials
 (such as the $K$-theoretic character of the affine Laumon space) are most likely far-going generalizations
 of the elementary equations for Macdonald polynomials. For this reason, when considering the non-stationary difference
 equation for degenerate five-point conformal blocks of $q$-Virasoro algebra in~\cite{Shakirov:2021krl},
 we paid attention to the following interesting fact. In~\cite{langmann2020basic}, the main part of the non-stationary
 difference equation is the operator $q^{\Delta/2}$, where $\Delta$ is a certain second order differential operator.
 For the case of rank~1, i.e., $\mathfrak{gl}_2$, the action of $q^{\Delta/2}$ on a monomial of the form $(x_1/x_2)^{n}$
 is proportional (up to factors of the form $q$, $t$ to the linear power in $n$) to $q^{n^2}$.

From the perspective described above, it is natural to expect that this operator should be
understood as a product of two operators, each of which contributes $q^{n^2/2}$.
This is natural both from the perspective of Macdonald theory, where this operator ($\Bor^{\prime}$ or $\Bor$, equivalently)
arises in the integral identities, as well as from the perspective of Ruijsenaars--Schneider system,
where~$\Delta$ is a sum of $N$ contributions, i.e., 2 contributions in the $\mathfrak{gl}_2$ case.
This was the principal reason why in \cite{Shakirov:2021krl} we chose to look for the Hamiltonian operator
on the right-hand side of the non-stationary equation in the form (see \eqref{Shamil-Eq})
\[
{\mathcal A}_1 \Bor {\mathcal A}_2 \Bor {\mathcal A}_3,
\]
where the two operators $\Bor$ are not joined together but separate.
It eventually turned out to be the correct ansatz.


\section[Lemma on the q-Borel transformation]{Lemma on the $\boldsymbol{q}$-Borel transformation}
\label{App:Proofs}

A proof of the following lemma is given in \cite[Appendix B]{Awata:2022idl}.
Here we provide another more direct proof. In the following we use a shorthand notation
for the infinite product $\varphi(z) := (z;q)_\infty$.

\begin{Lemma}[\cite{Shakirov:2021krl}]\label{Shamil-lemma}
For $n\in \mathbb{Z}$,
we have
\begin{align}
&\mathcal{B}\cdot {1\over \varphi(\alpha x)\varphi(\beta \Lambda/x)}x^n
={ \varphi\bigl(-q^{1+n}\alpha x\bigr)\varphi(-q^{-n}\beta \Lambda/x)\over \varphi(\alpha \beta \Lambda)}
q^{n(n+1)/2}x^n,\label{lem-1}\\
&\mathcal{B}^{-1}\cdot {\varphi(\alpha x)\varphi(\beta \Lambda/x)}x^n
={ \varphi\bigl(q^{-1}\alpha \beta \Lambda\bigr)\over
\varphi\bigl(-q^{-1-n}\alpha x\bigr)\varphi(-q^{n}\beta \Lambda/x)}
q^{-n(n+1)/2}x^n.\label{lem-2}
\end{align}
\end{Lemma}

\begin{proof}
We will show the first relation (the second relation follows from the first one).
By a~rescaling of $\Lambda$, $x$, it is enough to show
\begin{equation}
f(\Lambda,x):=\frac{1}{\varphi(-q x)\varphi\bigl(-\frac{\Lambda}{x}\bigr)}\Bor\frac{1}{\varphi(x)\varphi\bigl(\frac{\Lambda}{x}\bigr)}
=\frac{1}{\varphi(\Lambda)}.
\end{equation}
Recall that $\vartheta_x= x\frac{\partial}{\partial x},~\Bor = q^{\vartheta_x(\vartheta_x+1)/2}$ and $p=q^{\vartheta_x}$.
First we note that, for any function $F(x)$,
\begin{align}
\nonumber
\Bor \cdot x^n F(x)&{}= (px)^n F(px)
=q^{n(n+1)/2}x^n p^n F(p x)\\
&{}=q^{n(n+1)/2}x^n F(p q^n x)
=q^{n(n+1)/2}x^n \Bor \cdot F(q^n x).
\end{align}
Using this identity, we obtain the following difference equation:
\begin{align}
\nonumber
f(q \Lambda,x)&{}=
\frac{1}{\varphi(-q x)\varphi\bigl(-\frac{q\Lambda}{x}\bigr)}\Bor\frac{1-\frac{\Lambda}{x}}{\varphi(x)\varphi\bigl(\frac{\Lambda}{x}\bigr)}\\ \nonumber
&{}=\frac{1}{\varphi(-q x)\varphi\bigl(-\frac{q\Lambda}{x}\bigr)}\Bor\frac{1}{\varphi(x)\varphi\bigl(\frac{\Lambda}{x}\bigr)}
-\frac{1}{\varphi(-q x)\varphi\bigl(-\frac{q\Lambda}{x}\bigr)}\Bor\frac{\frac{\Lambda}{x}}{\varphi(x)\varphi\bigl(\frac{\Lambda}{x}\bigr)}\\ \nonumber
&{}=\biggl(1+\frac{\Lambda}{x}\biggr)f(\Lambda,x)-\frac{1+x}{\varphi(-x)\varphi\bigl(-\frac{q\Lambda}{x}\bigr)}\frac{\Lambda}{x}\Bor\frac{1}{\varphi\bigl(q^{-1}x\bigr)\varphi\bigl(\frac{q\Lambda}{x}\bigr)}\\
&{}=\biggl(1+\frac{\Lambda}{x}\biggr)f(\Lambda,x)-(1+x)\frac{\Lambda}{x} f\bigl(\Lambda,q^{-1}x\bigr).\label{eq:fdiff1}
\end{align}
Similarly, we have
\begin{align}
\nonumber
f(q \Lambda, q x)&{}=
\frac{1}{\varphi\bigl(-q^2 x\bigr)\varphi\bigl(-\frac{\Lambda}{x}\bigr)}\Bor\frac{1-x}{\varphi(x)\varphi\bigl(\frac{\Lambda}{x}\bigr)}\\ \nonumber
&{}=\frac{1}{\varphi\bigl(-q^2 x\bigr)\varphi\bigl(-\frac{\Lambda}{x}\bigr)}\Bor\frac{1}{\varphi(x)\varphi\bigl(\frac{\Lambda}{x}\bigr)}
-\frac{1}{\varphi\bigl(-q^2 x\bigr)\varphi\bigl(-\frac{\Lambda}{x}\bigr)}\Bor\frac{x}{\varphi(x)\varphi\bigl(\frac{\Lambda}{x}\bigr)}\\ \nonumber
&{}=(1+qx)f(\Lambda,x)-\frac{1+\frac{\Lambda}{qx}}{\varphi\bigl(-q^2x\bigr)\varphi\bigl(-\frac{\Lambda}{qx}\bigr)}qx
\Bor\frac{1}{\varphi(qx)\varphi\bigl(\frac{\Lambda}{qx}\bigr)}\\
&{}=(1+qx)f(\Lambda,x)-(qx+\Lambda)f(\Lambda,qx).\label{eq:fdiff2}
\end{align}
Comparing the $f(q \Lambda, q x)$ given by \eqref{eq:fdiff1} $|_{x\to qx}$ and \eqref{eq:fdiff2}, we have
\begin{equation}
f(\Lambda,q x)=f(\Lambda,x).
\end{equation}
Hence, from \eqref{eq:fdiff1} we have
\begin{equation}
f(q\Lambda,x)=(1-\Lambda) f(\Lambda,x).
\end{equation}
Since $f(0,x)=1$, we obtain
\begin{equation}
f(\Lambda, x)=\varphi(\Lambda)^{-1},
\end{equation}
as desired.
\end{proof}

\subsection{Relation to the pentagon identity}

There is yet another proof of Lemma \ref{Shamil-lemma},
\begin{equation}\label{eq:lemmaA1}
\Bor \frac{1}{\varphi(\alpha x)\varphi\bigl(\beta \frac{\Lambda}{x}\bigr)} x^n=
\frac{\varphi\bigl(-q^{n+1}\alpha x\bigr)\varphi\bigl(-q^{-n} \beta \frac{\Lambda}{x}\bigr)}{\varphi(\alpha\beta\Lambda)}q^{n(n+1)/2}x^n,
\end{equation}
and show its relation to the pentagon (quantum dilogarithm) identity.

(i) First we note that \eqref{eq:lemmaA1} follows from its special case $n=0$:
\begin{equation}\label{eq:lemmaA10}
\Bor \frac{1}{\varphi(\alpha x)\varphi\bigl(\beta \frac{\Lambda}{x}\bigr)} =
\frac{\varphi(-q \alpha x)\varphi\bigl(-\beta \frac{\Lambda}{x}\bigr)}{\varphi(\alpha\beta\Lambda)}.
\end{equation}
In fact, we have
\[
\Bor f(x) x^n=(px)^n \Bor f(x)=q^{n(n+1)/2}x^n p^n \Bor f(x)=q^{n(n+1)/2}x^n \Bor f(q^n x).
\]
Apply this to $f(x)=\frac{1}{\varphi(\alpha x)\varphi(\beta \frac{\Lambda}{x})}$ and
using (\ref{eq:lemmaA10}) we obtain (\ref{eq:lemmaA1}),
\begin{align*}
\Bor \frac{1}{\varphi(\alpha x)\varphi\bigl(\beta \frac{\Lambda}{x}\bigr)} x^n
&{}=q^{n(n+1)/2}x^n \Bor \frac{1}{\varphi(\alpha q^n x)\varphi\bigl(\beta \frac{q^{-n}\Lambda}{x}\bigr)} \\
&{}=q^{n(n+1)/2}x^n \frac{\varphi\bigl(-q^{n+1}\alpha x\bigr)\varphi\bigl(-q^{-n} \beta \frac{\Lambda}{x}\bigr)}{\varphi(\alpha\beta\Lambda)}.
\end{align*}

(ii) The equation \eqref{eq:lemmaA10} can be written as
\begin{align}
 \frac{\varphi(\alpha\beta\Lambda)}{\varphi(\alpha x)\varphi\bigl(\beta \frac{\Lambda}{x}\bigr)} &{} =
\Bor^{-1} \varphi(-q \alpha x)\varphi\biggl(-\beta \frac{\Lambda}{x}\biggr) \notag \\
&{}=\Bor^{-1} \sum_{k,l=0}^{\infty} \frac{q^{k(k+1)/2}(\alpha x)^k}{(q)_k}
\frac{q^{l(l-1)/2}\bigl(\beta \frac{\Lambda}{x}\bigr)^l}{(q)_l}
= \sum_{k,l=0}^{\infty} \frac{(\alpha x)^k}{(q)_k}
\frac{\bigl(\beta \frac{\Lambda}{x}\bigr)^l}{(q)_l} q^{k l},
\end{align}
hence it is enough to show
\begin{equation}\label{eq:lemmaA14}
 \frac{\varphi(\alpha\beta\Lambda)}{\varphi(\alpha x)\varphi\bigl(\beta \frac{\Lambda}{x}\bigr)}
= \sum_{k,l=0}^{\infty} \frac{(\alpha x)^k}{(q)_k}
\frac{\bigl(\beta \frac{\Lambda}{x}\bigr)^l}{(q)_l} q^{k l}.
\end{equation}
In terms of $q$-commutative variables $ba=q ab$
and $q$-exponential function $e_q(x)=\varphi(x)^{-1}$,
the equation~(\ref{eq:lemmaA14}) is equivalent to the quantum dilogarithm (pentagon) identity
\begin{equation}\label{eq:dilog}
e_q(b)e_q(a)=e_q(a)e_q(-ab)e_q(b).
\end{equation}
The identity \eqref{eq:dilog} can be easily derived as follows:
\begin{gather}
e_q(a)e_q(b)=e_q(a+b), \\
e_q(b)e_q(a)e_q(b)^{-1}
=e_q\bigl(e_q(b)a e_q(b)^{-1}\bigr)
=e_q\bigl(a e_q(q b)e_q(b)^{-1}\bigr)\notag \\ \hphantom{e_q(b)e_q(a)e_q(b)^{-1}}{}
=e_q(a (1-b))=e_q(a-a b)=e_q(a)e_q(-a b).
\end{gather}


\section[List of various R matrices and their characterization]{List of various $\boldsymbol{R}$ matrices and their characterization}
\label{App:List-R}
Since various $R$-matrices are used in the main text, we summarize them along with its characterization.

(1) The matrix $R^{\rm Sh}=(r_{i,j}(\Lambda))_{i,j=-n}^{m}$ arising from Shakirov's operator in \eqref{eq:SSr}:
\begin{align}\label{C1}
&q^{\frac{1}{2} i (i+1)} x^i \bigl(-d_1 q^{i-m}x\bigr)_{m-i} \biggl(-d_4 q^{-i-n}\frac{\Lambda}{x}\biggr)_{i+n} \nonumber \\
&\qquad{}=\sum_{j=-n}^m
q^{-\frac{1}{2} j (j+1)} x^j (-q^{-m}x)_{m-j} \biggl(-q^{-n}\frac{\Lambda}{x}\biggr)_{j+n}~r_{i,j}(\Lambda) .
\end{align}

(2) The matrix $R^{\rm HG}=\bigl(R^{\rm HG}_{ i,j}\bigr)_{i,j=0}^{\ell}$ defined by ${}_4\phi_{3}$ series in \eqref{eq:RinHGS}:
\begin{align}\label{C2}
\biggl(q^{\ell-1}\frac{x}{\beta};q^{-1}\biggr)_i \biggl(\frac{\alpha x}{z};q\biggr)_{\ell-i}
=\sum_{j=0}^\ell \biggl(q^{\ell-1}\frac{x}{z};q^{-1}\biggr)_{\ell-j}(x;q)_j~R^{\rm HG}_{i,j}.
\end{align}

(3) The matrix $R^{\rm Ito}=\bigl(R^{\rm Ito}_{i,j}\bigr)_{i,j=0}^{n}$ in Theorem \ref{thm:Ito}:
\begin{equation}
\bigl(q^{n-1}b_1 x; q^{-1}\bigr)_j\biggl(\frac{x}{a_2}; q\biggr)_{n-j}=
\sum_{i=0}^n \bigl(q^{n-1}b_2 x; q^{-1}\bigr)_{n-i}\biggl(\frac{x}{a_1}; q\biggr)_i R^{\rm Ito}_{i,j}.
\end{equation}

Actually, these $R$ matrices are not unrelated. When $\ell = n+m$,
the relation of $R^{\rm Sh}$ and $R^{\rm HG}$ is already proved in Section~\ref{subsec:r-matrix}, namely
we have \eqref{r=R}. One may confirm it by making a change of the ``dummy'' variable $x$ in~\eqref{C2}
to $- q^{-n} \Lambda/x$, which implies~\eqref{C1} up to a scaling constant.
On the other hand, with the identification
\begin{equation}\label{HG=Ito}
\ell = n; \qquad a_1 =1, \quad a_2 = \frac{z}{\alpha}, \quad b_1 = \frac{1}{\beta}, \quad b_2 = \frac{1}{z},
\end{equation}
we have $R^{\rm HG}_{i,j} = R^{\rm Ito}_{j,i}$.

(4) The matrix $A=A^{\rm Ito}=\bigl(A^{\rm Ito}_{ i,j}\bigr)_{i,j=0}^{n}$ in Theorem \ref{thm:Ito}:
\begin{align}
&\biggl[{n \atop i}\biggr]_q q^{\frac{i(i-1)}{2}}\biggl(-\frac{q^{1-n}}{z}\biggr)^i \biggl(\frac{x}{z}; q\biggr)_{n-i}\bigl(q^{n-i}x; q\bigr)_{i}\nonumber \\
&\qquad{}=\sum_{j=0}^n
\biggl[{n \atop j}\biggr]_q q^{\frac{(n-j)(n-j-1)}{2}}b_1^n\biggl(-\frac{q^{1-n}}{a_2b_1z}\biggr)^j
\bigl(q^{n-1}a_2b_2 x; q\bigr)_j\biggl(\frac{q^{j-n+1}x}{a_1b_1z}; q\biggr)_{n-j} A^{I}_{i,j}.
\end{align}

(5) The matrix ${\tilde A}=\bigl({\tilde A}_{i,j}\bigr)_{i,j=0}^{N}$ for the normalized solution in Section~\ref{sec3.2}:
\begin{align}
&\bigl(q^{N-j}a_2b_2x;q\bigr)_{j}\biggl(\frac{q^{1-N}}{a_1b_1}\frac{x}{\Lambda};q\biggr)_{N-j}
=\sum_{i=0}^N (x;q)_{i}\biggl(q^{i+1-N}\frac{x}{\Lambda};q\biggr)_{N-i}{\tilde A}_{i,j}.
\end{align}

(6) The matrix ${\bf R_{\mu,\nu}}=\bigl(\delta_{i+j,k+l}{\bf R_{\mu,\nu}}_{ i,j}^{k,l}\bigr)_{i,j,k,l\geq 0}$
for $\qg\bigl(A_1^{(1)}\bigr)$:
\begin{align}\label{eq:bRdef}
\biggl(\frac{b_{\nu}x}{z_{\nu}};q\biggr)_i \biggl(\frac{q^i x}{z_{\mu}};q\biggr)_{j}
=\sum_{k=0}^{n}\biggl(\frac{b_{\mu}x}{z_{\mu}};q\biggr)_{k} \biggl(\frac{q^{k} x}{z_{\nu}};q\biggr)_{n-k}{\bf R_{\mu,\nu}}_{i,j}^{k,n-k}.
\end{align}
This matrix ${\bf R_{\mu,\nu}}$ is related to the matrix $R^{\rm HG}$ by
\[
{{\bf R}_{\mu,\nu}}_{i,n-i}^{j,n-j}=\bigl[R^{\rm HG}\bigr]_{\ell \to n, z\to b_\mu
\frac{z_\nu}{z\mu}, \alpha\to b_\nu, \beta\to b_\mu} J, \qquad
J=(\delta_{i+j,n})_{i,j=0}^n.
\]
A significance of the matrix ${\bf R}$ is that it satisfies the Yang--Baxter relation.
The following direct proof using \eqref{eq:bRdef} may be instructive.
\begin{Proposition}
The matrix ${\bf R}_{\mu,\nu}$ satisfies the Yang--Baxter equation, namely for any $(i,j,k)$, $(i'',j'',k'') \in \bbZ^3$ such that $i+j+k=i''+j''+k''$,
we have
\begin{equation}\label{eq:YB-bR}
\sum_{i',j',k'}{{\bf R}_{1,2}}_{i',j'}^{i'',j''}{{\bf R}_{1,3}}_{i,k'}^{i',k''}{{\bf R}_{2,3}}_{j,k}^{j',k'}=
\sum_{i',j',k'}{{\bf R}_{2,3}}_{j',k''}^{j'',k''}{{\bf R}_{1,3}}_{i',k}^{i'',k'}{{\bf R}_{1,2}}_{i,j}^{i',j'}.
\end{equation}
\end{Proposition}
\begin{proof}
For $k_1, k_2, \dots, k_s \in \bbZ_{\geq 0}$ with fixed $n=k_1+k_2+\cdots+k_s$, define a rational function
in $t_1, \dots, t_n$ as
\begin{equation}
B_{k_1, k_2, \dots, k_s}={\mathcal S}\Biggl(\prod_{1 \leq i<j \leq n} \frac{q t_i-t_j}{t_i-t_j} \cdot
\prod_{a=1}^s \prod_{i_a=k_1+\cdots+k_{a-1}+1}^{k_1+\cdots+k_a}f_a(t_{i_a})
\Biggr),
\end{equation}
where
\[
f_a(t)=\prod_{i=1}^{a-1} \frac{1-\frac{\beta_i}{z_i}t}{1-\frac{t}{z_i}} \frac{1}{1-\frac{t}{z_a}},
\]
and ${\mathcal S}$ is the symmetrization on variables $t_1, \dots, t_n$.
The denominator of $B_{k_1, k_2, \dots, k_s}$ is
\[
\prod_{i=1}^n \prod_{\mu=1}^s\biggl(1-\frac{t_i}{z_\mu}\biggr)
\]
and the numerator is a symmetric polynomial in $t_1, \dots, t_n$ whose degree in each $t_i$ is $s-1$.
The functions $B_{k_1, \dots, k_s}$ form a basis of the linear space of such rational functions
of dimension $\bigl({n+s-1 \atop s-1}\bigr)$.
Define the action of a permutation $\sigma$ of $\{1,\dots, s\}$ on $\{z_a, \beta_a\}$ as
\[
s_\sigma=\bigl\{z_a \to z_{\sigma(a)},\, \beta_a \to \beta_{\sigma(a)}\bigr\},
\]
then we have the connection relation
\[
s_\sigma(B_{k_1, \dots, k_s})=\sum_{\{l\}} B_{l_1, \dots, l_s} {C_\sigma}_{\{k\}}^{\{l\}}.
\]
Note that the matrix $C_{(a,a+1)}$ is local in index $a$, i.e., it depends only on $k_a$, $k_{a+1}$, $l_a$, $l_{a+1}$ and~$z_a$,~$z_{a+1}$,~$\beta_a$,~$\beta_{a+1}$,
since $s_{(a,a+1)} (f_b(t))=f_b(t)$ for $a \neq b$, $b+1$. Then, due to the relation
$s_{(2,3)}s_{(1,2)}s_{(1,2)}B_{\{k\}}=s_{(1,2)}s_{(2,3)}s_{(1,2)}B_{\{k\}}$,
the matrices $C_{\sigma}$ satisfy the Yang--Baxter relation
\begin{equation}\label{eq:YB-Cmat}
C_{(2,3)}\bigl(s_{(2,3)}C_{(1,2)}\bigr)\bigl(s_{(2,3)}s_{(1,2)}C_{(2,3)}\bigr)
=C_{(1,2)}\bigl(s_{(1,2)}C_{(2,3)}\bigr)\bigl(s_{(1,2)}s_{(2,3)}C_{(1,2)}\bigr).
\end{equation}
Furthermore, by the locality, the matrix $C_{(a,a+1)}$ reduces to $C_{(1,2)}$, and we have
\begin{equation}\label{eq:C12}
C_{(a,a+1)}
=C_{(1,2)}|_{z_1\to z_a, z_2\to z_{a+1}, \beta_1\to \beta_a, \beta_2 \to \beta_{a+1}}.
\end{equation}
Thanks to the specialization
\[
B_{k_1, k_2}\big|_{t_i\to x q^{i-1}}=\frac{(q;q)_n}{(1-q)^{-n}} \frac{ \bigl( \frac{b_{1} x}{z_{1}};q\bigr)_{k_2}}{\bigl(\frac{x}{z_{1}};q\bigr)_{k_1+k_2}\bigl(\frac{x}{z_{2}};q\bigr)_{k_2}},
\]
the coefficient ${C_{(1,2)}}_{k_1,k_2}^{l_1,l_2}$ ($k_1+k_2=l_1+l_2$) is obtained from
\[
\frac{ \bigl( \frac{b_{2} x}{z_{2}};q\bigr)_{k_2}}{\bigl(\frac{x}{z_{2}};q\bigr)_{k_1+k_2}\bigl(\frac{x}{z_{1}};q\bigr)_{k_2}}=
\sum_{l_1, l_2}\frac{\bigl(\frac{b_{1} x}{z_{1}};q\bigr)_{l_2}}{\bigl(\frac{x}{z_{1}};q\bigr)_{l_1+l_2}\bigl(\frac{x}{z_{2}};q\bigr)_{l_2}}(C_{1,2})_{k_1,k_2}^{l_1,l_2},
\]
i.e.,
\begin{equation}\label{eq:C-con}
\biggl(\frac{b_{2} x}{z_{2}};q\biggr)_{k_2}\biggl(q^{k_2}\frac{x}{z_{1}};q\biggr)_{k_1}=
\sum_{l_1, l_2}\biggl(\frac{b_{1} x}{z_{1}};q\biggr)_{l_2}\biggl(q^{l_2}\frac{x}{z_{2}};q\biggr)_{l_1}(C_{1,2})_{k_1,k_2}^{l_1,l_2}.
\end{equation}
Equations \eqref{eq:YB-Cmat}--\eqref{eq:C-con} give the desired Yang--Baxter relation and defining equation for
the matrix $C$ (= ${\bf R}$).
\end{proof}

\section[Proof of RD\_2A=ARD\_2]{Proof of $\boldsymbol{R D_2 A=A R D_2}$}
\label{App:matrix-inversion}

\subsection[Recollection of Ito's R and A matrices]{Recollection of Ito's $\boldsymbol{R}$ and $\boldsymbol{A}$ matrices}
Let $n\in \mathbb{Z}_{\geq 0}$,\footnote{The non-negative integer $n$ is denoted by $N$ in the main text.}
and $a_1,a_2,b_1,b_2,q,\Lambda\in \mathbb{C}$ be
generic parameters.
Ito has introduced the matrices~$R$ and~$A$ (in \cite[Theorems 1.4 and 1.7]{MIto}) given as follows.
Remark that to have a consistent notation with the main body of our paper,
we change the notation as $t_{\rm Ito}\rightarrow q$, $(q^\alpha)_{\rm Ito}\rightarrow \Lambda$.
\begin{Definition}
Let $R$ be the $(n+1)$-dimensional square matrix given in the Gauss decomposed forms
\begin{gather*}
 R=L_R D_R U_R=U_R'D_R'L_R', \\
 L_R=\bigl(l^R_{ij}\bigr)_{0\leq i,j\leq n}, \qquad D_R=\bigl(\delta_{ij}d^R_{j}\bigr)_{0\leq i,j\leq n}, \qquad U_R=\bigl(u^R_{ij}\bigr)_{0\leq i,j\leq n}, \\
 U_R'=\bigl({u'}^R_{ij}\bigr)_{0\leq i,j\leq n}, \qquad D_R'=\bigl(\delta_{ij}{d'}^R_{j}\bigr)_{0\leq i,j\leq n}, \qquad L_R'=\bigl({l'}^R_{ij}\bigr)_{0\leq i,j\leq n},
\end{gather*}
where
\begin{align}
&
l^R_{ij}=\left[ n-j\atop n-i\right]_{q^{-1}}
{(-1)^{i-j} q^{-\left(i-j\atop 2\right)} \bigl(a_2b_2 q^j;q\bigr)_{i-j} \over
\bigl(a_1^{-1}a_2 q^{-(n-2j-1)};q\bigr)_{i-j}},\nonumber\\
&
d^R_j={\bigl(a_1a_2^{-1} q^{-j};q\bigr)_{n-j} (a_2 b_1;q)_j \over
(a_1b_2;q)_{n-j} \bigl(a_1^{-1}a_2 q^{-(n-j)};q\bigr)_j },\nonumber\\
&
u^R_{ij}=
\left[ j\atop i\right]_{q^{-1}}
{\bigl(a_1b_1 q^{n-j};q\bigr)_{j-i} \over
\bigl(a_1a_2^{-1} q^{n-i-j};q\bigr)_{j-i}},\nonumber\\
&
{u'}^R_{ij}=
\left[ j\atop i\right]_{q}
{(-1)^{j-i} q^{\left(j-i\atop 2\right)} \bigl(a_1^{-1}b_1^{-1} q^{-(n-i-1)};q\bigr)_{j-i} \over
\bigl(b_1^{-1}b_2 q^{i+j-n};q\bigr)_{j-i}},\nonumber\\
&
{d'}^R_j={\bigl(b_1b_2^{-1} q^{n-2j+1};q\bigr)_{j} \bigl(a_2^{-1} b_1^{-1}q^{-(n-j-1)};q\bigr)_{n-j} \over
\bigl(a_1^{-1}b_2^{-1}q^{-(j-1)};q\bigr)_{j} \bigl(b_1^{-1}b_2 q^{-(n-2j-1)};q\bigr)_{n-j} },\nonumber\\
&
{l'}^R_{ij}=\left[ n-j\atop n-i\right]_{q}
{\bigl(a_2^{-1}b_2^{-1} q^{-(i-1\bigr)};q)_{i-j} \over
\bigl(b_1b_2^{-1} q^{n-2i+1};q\bigr)_{i-j}}.\label{Ito-R}
\end{align}
\end{Definition}

\begin{Definition}
Let $A$ be the $(n+1)$-dimensional square matrix given in the Gauss decomposed forms
\begin{gather*}
A=L_A D_A U_A=U_A'D_A'L_A', \\
L_A=\bigl(l^A_{ij}\bigr)_{0\leq i,j\leq n}, \qquad D_A=\bigl(\delta_{ij}d^A_{j}\bigr)_{0\leq i,j\leq n}, \qquad U_A=\bigl(u^A_{ij}\bigr)_{0\leq i,j\leq n},\\
U_A'=\bigl({u'}^A_{ij}\bigr)_{0\leq i,j\leq n}, \qquad D_A'=\bigl(\delta_{ij}{d'}^A_{j}\bigr)_{0\leq i,j\leq n}, \qquad L_A'=\bigl({l'}^A_{ij}\bigr)_{0\leq i,j\leq n},
\end{gather*}
where
\begin{align}
&
l^A_{ij}=
(-1)^{i-j}q^{\left(n-i\atop 2\right)-\left(n-j\atop 2\right)}
\left[ n-j\atop n-i\right]_{q}
{ \bigl(a_2b_2 q^j;q\bigr)_{i-j} \over
\bigl(\Lambda a_2b_2 q^{2j};q\bigr)_{i-j}},\nonumber\\
&
d^A_j=
a_1^{n-j} a_2^j q^{\left(j\atop 2\right)+\left(n-j\atop 2\right)}
{(\Lambda;q)_{j} \bigl(\Lambda a_2 b_2 q^{2j};q\bigr)_{n-j} \over
\bigl(\Lambda a_2b_2q^{j-1};q\bigr)_{j} \bigl(\Lambda a_1a_2 b_1b_2 q^{n+j-1};q\bigr)_{n-j} },\nonumber\\
&
u^A_{ij}=
\bigl(-\Lambda a_1^{-1}a_2\bigr)^{j-i}
q^{\left(j\atop 2\right)-\left(i\atop 2\right)}
\left[ j\atop i\right]_{q}
{\bigl(a_1b_1 q^{n-j};q\bigr)_{j-i} \over
\bigl(\Lambda a_2b_2 q^{2i};q\bigr)_{j-i}},\nonumber\\
&
{u'}^A_{ij}=
(-\Lambda)^{j-i} q^{\left(n-i\atop 2\right)-\left(n-j\atop 2\right)}
\left[ j\atop i\right]_{q}
{ \bigl(a_1b_1 q^{n-j};q\bigr)_{j-i} \over
\bigl(\Lambda a_1 b_1 q^{2(n-j)};q\bigr)_{j-i}},\nonumber\\
&
{d'}^A_j=
a_1^{n-j} a_2^j
q^{\left(j\atop 2\right)+\left(n-j\atop 2\right)}
{\bigl(\Lambda a_1b_1 q^{2(n-j)};q\bigr)_{j} (\Lambda ;q)_{n-j} \over
\bigl(\Lambda a_1a_2b_1b_2 q^{2n-j-1};q\bigr)_{j} \bigl(\Lambda a_1b_1 q^{n-j-1};q\bigr)_{n-j} },\nonumber\\
&
{l'}^A_{ij}=
(-a_1a_2^{-1})^{i-j}
q^{\left(j\atop 2\right)-\left(i\atop 2\right)}
\left[ n-j\atop n-i\right]_{q}
{\bigl(a_2b_2 q^{j};q\bigr)_{i-j} \over
\bigl(\Lambda a_1b_1 q^{2(n-i)};q\bigr)_{i-j}}.\label{Ito-A}
\end{align}
\end{Definition}

It was shown that a certain system of Jackson integrals satisfies two different types of $q$-KZ equations
described by two matrices $R$ and $A$
(see \cite[Proposition~1.2 (Matsuo) and Theorem~1.7]{MIto}),
which one may regard as a consequence of
the so-called base-fiber duality of the gauge theory.
Hence, the compatibility of the two dual systems
guarantees the commutativity of the pertinent matrices.
\begin{Theorem}\label{thm-TBP}
Let $R$ and $A$ be Ito's matrices given as above.
Set
$D_2=\bigl(\bigl(\Lambda q^{n-1}\bigr)^i \delta_{ij}\bigr)_{i,j=0}^n$.
Then we have $R D_2 A=A R D_2$.
\end{Theorem}
In this appendix, we give an alternative direct proof of this commutativity.

\subsection{Notations}
We use the standard notations for the basic hypergeometric ${}_{r+1}\phi_r$ series
\begin{align*}
&{}_{r+1}\phi_r\biggl[{a_1,a_2,\dots,a_{r+1}\atop b_1,\dots,b_r};q,z \biggr]=
\sum_{n=0}^\infty
{(a_1;q)_n (a_2;q)_n \cdots (a_{r+1};q)_n \over
(q;q)_n (b_1;q)_n \cdots (b_{r};q)_n }z^n.
\end{align*}
A ${}_{r+1}\phi_r$ series is called balanced (or Saalsch\"utzian) if $b_1b_2\cdots b_r=q a_1a_2\cdots a_{r+1}$ and $z=q$.
We use the shorthand notation ${}_{r+1}W_r$ for the very-well-poised ${}_{r+1}\phi_r$ series
\begin{align*}
{}_{r+1}W_r(a_1;a_4,a_5,\dots,a_{r+1};q,z)
&{}={}_{r+1}\phi_r\left[{a_1,q a_1^{1/2},-q a_1^{1/2},a_4,\dots,a_{r+1}\atop
a_1^{1/2},-a_1^{1/2},qa_1/a_4,\dots,qa_1/a_{r+1}};q,z \right] \\
&{} = \sum_{n=0}^\infty
{(a_1;q)_n (a_4;q)_n \cdots (a_{r+1};q)_n \over
(q;q)_n (q a_1/a_4;q)_n \cdots (q a_1/a_{r+1};q)_n }{1-a_1 q^{2n}\over 1-a_1}z^n.
\end{align*}
As for the detail, see \cite{hypergeometric}.

\subsection{Andrews' and Bressoud's matrix inversion formulas}

Let ${\mathcal A}(a)=({\mathcal A}_{ij}(a))_{i,j=0}^\infty$ be Andrews' infinite-dimensional lower triangular matrix \cite{A},
\begin{align}
{\mathcal A}_{ij}(a)={1\over (q;q)_{i-j}(aq;q)_{i+j}}.\label{Andrews-1}
\end{align}
Then the inverse matrix ${\mathcal A}^{-1}(a)=\bigl({\mathcal A}^{-1}_{ij}(a)\bigr)_{i,j=0}^\infty$ is given by
\begin{align}
{\mathcal A}^{-1}_{ij}(a)=
{ \bigl(1-a q^{2i} \bigr)(a;q)_{i+j} (-1)^{i-j} q^{({i-j\atop 2})}
\over(1-a)(q;q)_{i-j}}.\label{Andrews-2}
\end{align}

Let ${\mathcal D}(a,b)=({\mathcal D}_{ij}(a,b))_{i,j=0}^\infty$ be Bressoud's infinite-dimensional lower triangular matrix \cite{B}
in the original form,
\begin{align}
{\mathcal D}_{ij}(a,b)={\bigl(1-a q^{2j}\bigr) (b;q)_{i+j} \bigl(ba^{-1};q\bigr)_{i-j} \bigl(ba^{-1}\bigr)^j
\over (1-a) (a q;q)_{i+j} (q;q)_{i-j}}.
\end{align}
Then the inverse matrix is given by ${\mathcal D}^{-1}(a,b)=({\mathcal D}_{ij}(b,a))_{i,j=0}^\infty$.

Closely following the line in \cite{AAB},
one can state the interrelations between Andrews' and Bressoud's matrices
in the following manner, aiming at (\ref{Cor1}) and (\ref{Cor2}) in Corollary~\ref{Cor:Bressoud} below.
Consider the gauge transformation ${\mathcal B}(a,b)=({\mathcal B}_{ij}(a,b))_{i,j=0}^\infty$ of Bressoud's matrix
\begin{align}
{\mathcal B}_{ij}(a,b)&=
{(a q;q)_{2i} \over (a ;q)_{2i} }b^i \cdot {\mathcal D}_{ij}(b,a)\cdot {(b;q)_{2j} \over (bq ;q)_{2j} }a^{-j}=
{\bigl(1-a q^{2i}\bigr) (a;q)_{i+j} (a/b;q)_{i-j}
\over (1-a) (b q;q)_{i+j} (q;q)_{i-j}} b^{i-j}. \label{Bressoud:dfn}
\end{align}

\begin{Remark}
Our choice of the gauge transformation is different from the one in \cite{AAB}.
\end{Remark}

\begin{Proposition}[\cite{AAB}]
We have the initial condition ${\mathcal B}(a,a)=(\delta_{ij})_{i,j=0}^\infty$ and the transition property
\begin{align}
{\mathcal B}(a,b){\mathcal B}(b,c)={\mathcal B}(a,c),\label{Bressoud}
\end{align}
which in the particular case $a=c$ means Bressoud's matrix inversion ${\mathcal B}^{-1}(a,b)={\mathcal B}(b,a)$,
or ${\mathcal D}^{-1}(a,b)={\mathcal D}(b,a)$.
\end{Proposition}

For the readers' convenience, we reproduce the proof in \cite{AAB}.
\begin{proof} Let $i\geq j$. We need to calculate
$\sum_{k=j}^i {\mathcal B}_{ik}(a,b){\mathcal B}_{kj}(b,c)$.
In the following summation, by using the definition
 (\ref{Bressoud:dfn}), one finds the
summable very-well-poised ${}_6W_5$ series \cite[formula~(2.4.2)]{hypergeometric} as
\begin{align*}
\sum_{l=0}^{i-j}
{ {\mathcal B}_{i,j+l}(a,b){\mathcal B}_{j+l,j}(b,c) \over
{\mathcal B}_{ij}(a,b){\mathcal B}_{jj}(b,c)
}&{}=
{}_6W_5\bigl(b q^{2j}; a q^{i+j}, b/c,q^{-i+j};q,cq/a \bigr)\\
&{}=
{\bigl(a/c,b q^{2j+1};q\bigr)_{i-j} \over \bigl(a/b,c q^{2j+1};q\bigr)_{i-j} } (c/b)^{i-j}.
\end{align*}
Then simplification of the factors shows that
\begin{align*}
\mathcal{B}_{ij}(a,b){\mathcal B}_{jj}(b,c)
{\bigl(a/c,b q^{2j+1};q\bigr)_{i-j} \over \bigl(a/b,c q^{2j+1};q\bigr)_{i-j} } (c/b)^{i-j}=\mathcal{B}_{ij}(a,c).
\end{align*}
Hence, we have (\ref{Bressoud}).
\end{proof}

\begin{Corollary}\label{Cor:Bressoud}
By taking the limits $b\rightarrow 0$ or $\infty$ in \eqref{Bressoud}, we have
\begin{align}
&\sum_{k} a^i
{\mathcal A}^{-1}_{ik}(a) a^{-k} c^k {\mathcal A}_{kj}(c) c^{-j}=
{\mathcal B}_{ij}(a,c), \label{Cor1}\\
&\sum_{k} q^{-i^2} {\mathcal A}^{-1}_{ik}(a) {\mathcal A}_{kj}(c) q^{j^2} =
{\mathcal B}_{ij}(a,c), \label{Cor2}
\end{align}
which reproduce Andrews' matrix inversion by setting $a=c$.
\end{Corollary}
\begin{proof}
It follows from
\begin{alignat*}{3}
&\lim_{b\rightarrow 0} {\mathcal B}_{ik}(a,b)=a^{i-k}{\mathcal A}^{-1}_{ik}(a) ,
&&
\lim_{b\rightarrow 0} {\mathcal B}_{kj}(b,c)=c^{k-j}{\mathcal A}_{kj}(c) ,&\\
&\lim_{b\rightarrow \infty} b^{2k}{\mathcal B}_{ik}(a,b)=q^{-i^2-k(k+1)}{\mathcal A}^{-1}_{ik}(a) , \qquad
&&
\lim_{b\rightarrow \infty} b^{-2k} {\mathcal B}_{kj}(b,c)=q^{k(k+1)+j^2}{\mathcal A}_{kj}(c) .&
\tag*{\qed}
\end{alignat*}
\renewcommand{\qed}{}
\end{proof}

\subsection[Expressions for L\_X, U\_X, U'\_X, L'\_X (X=R,A) in terms of Andrews' triangular matrices A\^{}\{pm 1\}(a)]{Expressions for $\boldsymbol{L_X}$, $\boldsymbol{U_X}$, $\boldsymbol{U'_X}$, $\boldsymbol{L'_X}$ ($\boldsymbol{X=R,A}$) in terms \\ of Andrews' triangular matrices $\boldsymbol{\mathcal{A}^{\pm 1}(a)}$}

\begin{Definition}
Set
\begin{align}
\alpha=1/a_2b_2,\qquad \beta=1/a_1b_1,\qquad z=1/a_1b_2,\qquad w=\Lambda q^{n-1}.\label{parameters}
\end{align}
\end{Definition}

\begin{Definition}
Define the gauge factors $g^L_i$, $g^U_i$, $g'^L_i$, $g'^U_i$
and $h^R_i$, $h^A_i$, $h'^R_i$, $h'^A_i$ by
\begin{alignat}{3}
&g^L_i=q^i (1/\alpha;q)_i(q^{-n};q)_i,\qquad && g^U_i=(z/\alpha\beta)^i\bigl(q^{-n+1}\beta ;q\bigr)_i(q;q)_i, &\nonumber\\
&g'^L_i= q^i (z/\alpha\beta)^{-i} (1/\alpha;q)_i(q^{-n};q)_i,\qquad&& g'^U_i=\bigl(q^{-n+1}\beta ;q\bigr)_i(q;q)_i,&\nonumber\\
&h^R_i=\bigl(q^{-n+1} z/\alpha;q\bigr)_{2i},\qquad&& h^A_i=\bigl(q^{-n+1} w/\alpha;q\bigr)_{2i},&\nonumber\\
&h'^R_i=\bigl(q^{-n+1} \beta/z;q;q\bigr)_{2i},\qquad&& h'^A_i= \bigl(q^{-n+1} \beta/w;q\bigr)_{2i}.&\label{gauge}
\end{alignat}
\end{Definition}

\begin{Proposition}
We have
\begin{alignat}{3}
&
l^R_{ij}=g^L_i {\mathcal A}_{ij}(q^{-n} z/\alpha)
{h^R_j \over g^L_j},
\qquad &&
l^A_{ij}= g^L_i {\mathcal A}_{ij}(q^{-n} w/\alpha)
{ h^A_j\over
 g^L_j},&\label{lr-prp-1} \\
&
u^R_{ij}=
{h^R_i\over g^U_i}
{\mathcal A}_{ji}(q^{-n} z/\alpha)
g^U_j,
\qquad&&
u^A_{ij}=
{h^A_i\over
w^{i}g^U_i}
{\mathcal A}_{ji}(q^{-n} w/\alpha)
w^{j}g^U_j,&\label{lr-prp-2}\\
&{u'}^R_{ij}=
{1\over g'^U_i}
 {\mathcal A}^{-1}_{ji}(q^{-n}\beta/z)
 {g'^U_j\over
 h'^R_j },
 \qquad &&
 {u'}^A_{ij}=
 {1\over g'^U_i}
 {\mathcal A}^{-1}_{ji}(q^{-n}\beta/w)
 {g'^U_j\over
h'^A_j },&\label{lr-prp-3}\\
 &{l'}^R_{ij}=
 {g'^L_i\over
h'^R_i }
 {\mathcal A}^{-1}_{ij}(q^{-n}\beta/z)
 {1\over g'^L_j },
\qquad&&
 {l'}^A_{ij}=
 { g'^L_i \over
w^{i} h'^A_i }
 {\mathcal A}^{-1}_{ij}(q^{-n}\beta/w)
 {w^{j}\over g'^L_j}.&\label{lr-prp-4}
\end{alignat}
\end{Proposition}

\begin{proof}
Straightforward calculation by using the definitions in (\ref{Ito-R}) and (\ref{Ito-A}),
the parametrization~(\ref{parameters}),
the gauge factors (\ref{gauge}), and
Andrews' matrices (\ref{Andrews-1}) and (\ref{Andrews-2}).
\end{proof}

\subsection[Matrices R and A and Watson's transformation formula for terminating very-well-poised \_8phi\_7 series]{Matrices $\boldsymbol{R}$ and $\boldsymbol{A}$ and Watson's transformation formula \\ for terminating very-well-poised $\boldsymbol{{}_8\phi_7}$ series}
\begin{Proposition}
We have
\begin{gather}
 d^R_j =
{( \alpha/z ;q)_n\over (1/z ;q)_n}\cdot
q^{j^2} (q^{n}\alpha )^{-j}
{(z/\alpha\beta;q)_j (q z/\alpha ;q)_j (q^{-n} z/\alpha ;q)_j \bigl(q^{-n+1} z;q\bigr)_j \over
(q^{-n} z/\alpha ;q)_{2j} \bigl(q^{-n+1} z/\alpha ;q\bigr)_{2j}},\label{prp-d-1}\\
 d'^R_{j} =
\alpha^{n} {(z/\alpha\beta ;q)_n\over (z/\beta ;q)_n}\cdot
q^{-j^2} (q^{n}/\beta )^{j}
{
(q^{-n}\beta/z;q)_{2j}\bigl(q^{-n+1}\beta/z ;q\bigr)_{2j} \over
(1/z;q)_j (q\beta/z ;q)_j (q^{-n}\beta/z;q)_j \bigl(q^{-n+1}\alpha \beta/z ;q\bigr)_j },\label{prp-d-2}\\
 d^A_j =
(-a_1w/\alpha)^n {(1/w;q)_n\over (w/\alpha\beta;q)_n}
 \cdot
{( \alpha/w ;q)_n\over (1/w ;q)_n}\cdot
z^j q^{j^2} (q^{n}\alpha )^{-j} \nonumber\\ \hphantom{d^A_j =}{}
\times {(w/\alpha\beta;q)_j (q w/\alpha ;q)_j (q^{-n} w/\alpha ;q)_j \bigl(q^{-n+1} w;q\bigr)_j \over
(q^{-n} w/\alpha ;q)_{2j} \bigl(q^{-n+1} w/\alpha ;q\bigr)_{2j}}, \label{prp-d-3}\\
 d'^A_{j} =
(-a_1w/\alpha)^n {(1/w;q)_n\over (w/\alpha\beta;q)_n}
 \cdot
\alpha^{n} {(w/\alpha\beta ;q)_n\over (w/\beta ;q)_n}\cdot
z^j q^{-j^2} (q^{n}/\beta )^{j}\nonumber\\ \hphantom{d'^A_{j} =}{}
\times {(q^{-n}\beta/w;q)_{2j}\bigl(q^{-n+1}\beta/w ;q\bigr)_{2j} \over
(1/w;q)_j (q\beta/w ;q)_j (q^{-n}\beta/w;q)_j \bigl(q^{-n+1}\alpha \beta/w ;q\bigr)_j }.\label{prp-d-4}
\end{gather}
\end{Proposition}

\begin{proof}
Straightforward calculation by using the definitions in (\ref{Ito-R}) and (\ref{Ito-A}),
and the para\-metri\-za\-tion~(\ref{parameters}).
\end{proof}

\begin{Definition}
Set
\begin{align}
\mathsf{R}^n_{ij}(\alpha,\beta|z)&{}=
\beta^{-i}
{(q;q)_n (\alpha/z;q)_{n-j} (1/\beta ;q)_{n-i} (\beta/z;q)_i \over
(q;q)_i(q;q)_{n-i} (1/z;q)_n(1/\beta;q)_{n-j}}\nonumber\\
&\quad{}\times
{}_4 \phi_3 \biggl[{ q^{-i}, q^{-n+j},z/\alpha\beta, q^{-n+1}z \atop
q^{-n},q^{-i+1}z/\beta ,q^{-n+j+1}z/\alpha};q,q \biggr].
\end{align}
\end{Definition}

\begin{Proposition}\label{prp-R-A}
We have
\begin{align}
&R_{ij}=\mathsf{R}^n_{ij}(\alpha,\beta|z),\label{R-4phi3}\\
&A_{ij}=
(-a_1w/\alpha)^n {(1/w;q)_n\over (w/\alpha\beta;q)_n}\cdot \mathsf{R}^n_{ij}(\alpha,\beta|w) \cdot z^j.\label{A-4phi3}
\end{align}
\end{Proposition}

Recalling that $D_2=\bigl(w^i \delta_{ij}\bigr)_{i,j=0}^n$, the commutativity $RD_2A=ARD_2$
(in Theorem \ref{thm-TBP})
is recast as follows.
\begin{Theorem}\label{TBP-2}
We have
\begin{align}
\sum_{k=0}^n \mathsf{R}^n_{ik}(\alpha,\beta|z) w^k \mathsf{R}^n_{kj}(\alpha,\beta|w) z^j
=
\sum_{k=0}^n \mathsf{R}^n_{ik}(\alpha,\beta|w) z^k \mathsf{R}^n_{kj}(\alpha,\beta|z) w^j .
\end{align}
\end{Theorem}

\begin{proof}[Proof of Proposition \ref{prp-R-A}]
Starting from the Gauss decomposition $R=L_R D_R U_R$,
one can proceed as follows. We have
\begin{align*}
R_{ij}=\sum_k l^R_{ik} d^R_k u^R_{kj}=
l^R_{i0} d^R_0 u^R_{0j} \sum_{k=0}^{\min(i,j)} {l^R_{ik}\over l^R_{i0}} {d^R_k\over d^R_0} {u^R_{kj}\over u^R_{0j}}
=l^R_{i0} d^R_0 u^R_{0j}
\cdot X,
\end{align*}
where the summation $X$ can be calculated by using (\ref{lr-prp-1}), (\ref{lr-prp-2}), and (\ref{lemma-A-2}), (\ref{gauge-R-1}), (\ref{lem-d-2}) below
as
\begin{align*}
X&{}=\sum_{k=0}^{\min(i,j)}
{\mathcal{A}_{ik}(q^{-n}z/\alpha)\over \mathcal{A}_{i0}(q^{-n}z/\alpha)}
{h^R_k\over g^L_k} {d^R_k\over d^R_0}
{h^R_k\over g^U_k}
{\mathcal{A}_{jk}(q^{-n}z/\alpha)\over \mathcal{A}_{j0}(q^{-n}z/\alpha)}
\\
&{}={}_8 W_7 \bigl( q^{-n}z/\alpha;q^{-i},q^{-j},z/\alpha\beta, qz/\alpha, q^{-n+1}z;q, q^{-n+i+j}\beta/z\bigr)\\
&{}=
{\bigl(q^{-n+1}z/\alpha;q\bigr)_i (\beta/z;q)_i
\over
\bigl(q^{-n+1}\beta ;q\bigr)_i (1/\alpha;q)_i }
{}_4 \phi_3 \biggl[{ q^{-i}, q^{-n+j},z/\alpha\beta, q^{-n+1}z \atop
q^{-n},q^{-i+1}z/\beta,q^{-n+j+1}z/\alpha};q,q \biggr].
\end{align*}
Here, we have used
Watson's transformation formula for
terminating very-well-poised ${}_8\phi_7$ series~\cite[formula~(2.5.1)]{hypergeometric}.

From (\ref{gauge}), (\ref{lr-prp-1}), (\ref{lr-prp-2}), (\ref{lemma-A-1}), and (\ref{lem-d-1}) below, we have
\begin{align*}
l^R_{i0} d^R_0 u^R_{0j}&{}=g^L_i \mathcal{A}_{i0}(q^{-n}z/\alpha)d^R_0 \mathcal{A}_{j0}(q^{-n}z/\alpha)g^U_j\\
&{}=q^i (z/\alpha\beta)^j
{(1/\alpha;q)_i (q^{-n};q)_i\over \bigl(q^{-n+1}z/\alpha;q\bigr)_i (q;q)_i}
{
\bigl(q^{-n+1}\beta;q\bigr)_j \over \bigl(q^{-n+1}z/\alpha;q\bigr)_j}
{
(\alpha/z;q)_n \over (1/z;q)_n
}.
\end{align*}
Hence, it holds that
\begin{align*}
R_{ij}&{}=q^i (z/\alpha\beta)^j
{(\beta/z;q)_i (q^{-n};q)_i\over \bigl(q^{-n+1}\beta;q\bigr)_i (q;q)_i}
{
\bigl(q^{-n+1}\beta;q\bigr)_j \over \bigl(q^{-n+1}z/\alpha;q\bigr)_j}
{
(\alpha/z;q)_n \over (1/z;q)_n
}\\
&\quad\times {}_4 \phi_3 \biggl[{ q^{-i}, q^{-n+j},z/\alpha\beta, q^{-n+1}z \atop
q^{-n},q^{-i+1}z/\beta,q^{-n+j+1}z/\alpha};q,q \biggr].
\end{align*}
Simplifying the prefactor, we have (\ref{R-4phi3}).

The formula (\ref{A-4phi3}) for the matrix $A$ immediately follows from
the one (\ref{R-4phi3}) for $R$,
by making a comparison of the matrix elements given
in (\ref{lr-prp-1})--(\ref{lr-prp-4}),
and (\ref{prp-d-1})--(\ref{prp-d-4}).
\end{proof}

\begin{Remark}
If we use the opposite Gauss decomposition $R=U'_R D'_R L'_R$,
we have
\begin{align*}
R_{ij}=\sum_k u'^R_{ik} d'^R_k l'^R_{kj}=
u'^R_{in} d'^R_n l'^R_{nj}\sum_{k=\max(i,j)}^n
{u'^R_{ik}\over u'^R_{in}} {d'^R_k\over d'^R_n} { l'^R_{kj}\over l'^R_{nj}}
=u'^R_{in} d'^R_n l'^R_{nj}\cdot X',
\end{align*}
where $X$ can be calculated by using (\ref{lr-prp-3}), (\ref{lr-prp-4}), and (\ref{lemma-A-4}), (\ref{gauge-R-2}) and (\ref{lem-d-4}) below,
\begin{align*}
X'&=\sum_{l=0}^{\min(i,j)}
{\mathcal{A}^{-1}_{n-l,i}(q^{-n}\beta/z)\over \mathcal{A}^{-1}_{n,i}(q^{-n}\beta/z)}
{g'^U_{n-l} h'^R_n\over g'^U_{n}h'^R_{n-l}}
{d'^R_{n-l} \over d'^R_{n} }
{g'^L_{n-l} h'^R_n\over g'^L_{n}h'^R_{n-l}}
{\mathcal{A}^{-1}_{n-l,j}(q^{-n}\beta/z)\over \mathcal{A}^{-1}_{n,j}(q^{-n}\beta/z)}
\\
&={}_8 W_7 \bigl( q^{-n}z/\beta;q^{-n+i},q^{-n+j},z/\alpha\beta, qz/\beta, q^{-n+1}z;q, q^{-n+i+j}\alpha/z\bigr)\\
&=
{\bigl(q^{-n+1}z/\beta;q\bigr)_{n-j}(\alpha/z;q)_{n-j}
\over
\bigl(q^{-n+1}\alpha ;q\bigr)_{n-j} (1/\beta;q)_{n-j} }
{}_4 \phi_3 \biggl[{ q^{-i}, q^{-n+j},z/\alpha\beta, q^{-n+1}z \atop
q^{-n},q^{-i+1}z/\beta,q^{-n+j+1}z/\alpha};q,q \biggr].
\end{align*}
Note that from (\ref{gauge}), (\ref{lr-prp-3}), (\ref{lr-prp-4}), and (\ref{lemma-A-3}), (\ref{lem-d-3}) below,
\begin{align*}
u'^R_{in} d'^R_n l'^R_{nj}&{}=
{1\over g'^U_i} \mathcal{A}^{-1}_{n,i}(q^{-n}\beta/z)
{g'^U_n\over h'^R_n} d'^R_n {g'^L_n\over h'^R_n}
 \mathcal{A}^{-1}_{n,j}(q^{-n}\beta/z){1\over g'^L_j} \\
 &{}=q^i (z/\alpha\beta)^{j-n}
 {(\beta/z;q)_i (q^{-n};q)_i
 \over \bigl(q^{-n+1}\beta;q\bigr)_i (q;q)_i
 }
 {(\beta/z;q)_j\over (1/\alpha;q)_j}
 {(1/\beta;q)_n \over (1/z;q)_n}.
\end{align*}
Simplifying the prefactor, we also have (\ref{R-4phi3}).
\end{Remark}

\begin{Lemma}
We have
\begin{align}
&\mathcal{A}_{i0}(a)=
{1\over (aq;q)_i(q;q)_i},\label{lemma-A-1}\\
&
{\mathcal{A}_{ik}(a) \over \mathcal{A}_{i0}(a) }
=q^{-\left(k\atop 2 \right) }\bigl(-q^i\bigr)^k
{ \bigl(q^{-i};q\bigr)_k\over (a q^{i+1};q)_k },\label{lemma-A-2}\\
&
\mathcal{A}^{-1}_{n,i}(a) =
(-1)^n q^{\left(n\atop 2 \right)}
{(aq;q)_{2n}\over (a;q)_{2n}}
{(a;q)_{n}\over (q;q)_{n}}\cdot q^i
(a q^n;q)_i (q^{-n};q)_i,\label{lemma-A-3}\\
&
{\mathcal{A}^{-1}_{n-l,i}(a) \over \mathcal{A}^{-1}_{n,i}(a) }
=
q^{\left(l\atop 2 \right)}
\bigl(-q^{i+n} a\bigr)^{-l}
{1-a^{-1} q^{-2n+2l}\over 1-a^{-1} q^{-2n}}
{\bigl(q^{-n+i};q\bigr)_l\over \bigl(a^{-1} q^{1-n-i};q\bigr)_l }.\label{lemma-A-4}
\end{align}
\end{Lemma}

\begin{Lemma}
We have
\begin{align}
&
{h^R_k\over g^L_k}{h^R_k\over g^U_k}=
(qz/\alpha\beta)^{-k}
{\bigl(q^{-n+1}z/\alpha;q\bigr)_{2k} \bigl(q^{-n+1}z/\alpha;q\bigr)_{2k} \over
(1/\alpha;q)_k (q^{-n};q)_k\bigl(q^{-n+1}\beta;q\bigr)_k(q;q)_k}, \label{gauge-R-1}\\
&{g'^U_{n-l} h'^R_n\over g'^U_{n}h'^R_{n-l}}
{g'^L_{n-l} h'^R_n\over g'^L_{n}h'^R_{n-l}}=
q^{-2l^2}\bigl(q^{2n+1}\beta^2/z^3\bigr)^l
{(q^{-n}z/\beta;q)_{2l}(q^{-n}z/\beta;q)_{2l} \over
(1/\beta;q)_l(q^{-n};q)_l\bigl(q^{-n+1}\alpha;q\bigr)_l(q;q)_l}.\label{gauge-R-2}
\end{align}

\end{Lemma}

\begin{Lemma}
We have
\begin{align}
&d^R_0={( \alpha/z ;q)_n\over (1/z ;q)_n},\label{lem-d-1}\\
&{d^R_k \over d^R_0}=
q^{k^2} (q^{n}\alpha )^{-k}
{(q z/\alpha ;q)_k \bigl(q^{-n+1} z;q\bigr)_k (z/\alpha\beta;q)_k (q^{-n} z/\alpha ;q)_k\over
\bigl(q^{-n+1} z/\alpha ;q\bigr)_{2k}(q^{-n} z/\alpha ;q)_{2k}},\label{lem-d-2}\\
&d'^R_n=\beta^{-n} {(\beta/z ;q)_n\over (1/z ;q)_n},\label{lem-d-3}\\
&{d'^R_{n-l} \over d'^R_n}=
q^{l^2} (q^{-n}\alpha )^{l}
{
(q z/\beta;q)_l \bigl(q^{-n+1}z ;q\bigr)_l (z/\alpha\beta;q)_l (q^{-n}z/\beta ;q)_l \over
\bigl(q^{-n+1}z/\beta;q\bigr)_{2l}(q^{-n}z/\beta ;q)_{2l}
}.\label{lem-d-4}
\end{align}
\end{Lemma}

\subsection[Products of triangular matrices L\_A\^{}\{-1\}L\_R, U\_RD\_2U'\_A, U\_AU'\_R, and L'\_RD\_2L'\_A\^{}\{-1\} in terms of Bressoud's triangular matrix B(a,b)]{Products of triangular matrices $\boldsymbol{L_{A}^{-1}L_R}$, $\boldsymbol{U_{R}D_2 U'_A}$, $\boldsymbol{U_{A}U'_R}$, and $\boldsymbol{L'_RD_2{L'_{A}}^{\!\!-1}}$ \\ in terms of Bressoud's triangular matrix $\boldsymbol{\mathcal{B}(a,b)}$}
For the convenience of our study on the commutativity $R D_2 A=A R D_2$,
we write the matrices~$R$ and $A$ in Gauss decomposed forms and
investigate
\begin{align}
\bigl(L_RD_RU_R\bigr)D_2 \bigl(U'_AD'_AL'_A\bigr)=
\bigl(L_AD_AU_A\bigr) \bigl(U'_RD'_RL'_R\bigr)D_2,
\end{align}
which is equivalent to
\begin{align}
\mathsf{L}^1\cdot D_R\cdot \mathsf{U}^1 =
D_A\cdot \mathsf{U}^2 \cdot D'_R \cdot \mathsf{L}^2\cdot {D'_A}^{\!\!-1}, \label{TBP}
\end{align}
where
\begin{align}
&\mathsf{L}^1=L_{A}^{-1}L_R,\qquad \mathsf{U}^1=U_RD_2 U'_A,\qquad
\mathsf{U}^2=U_AU'_R,\qquad\,\, \mathsf{L}^2=L'_RD_2{L'_{A}}^{\!\!-1}.\label{L-U}
\end{align}

Note from (\ref{lr-prp-1}) and (\ref{lr-prp-4}), we have that
\begin{align}
&\bigl(L_{A}^{-1}\bigr)_{ij}=
{ g^L_i\over h^A_i} {\mathcal A}^{-1}_{ij}(q^{-n} w/\alpha )
 {1\over g^L_j} ,\label{lr-5}\\
 &\bigl({L'_{A}}^{\!\!-1}\bigr)_{ij}=
{g'^L_i \over w^{i} }
 {\mathcal A}_{ij}(q^{-n} \beta/w )
 {
 w^{j} h'^A_j\over
g'^L_j }.\label{lr-6}
\end{align}
One finds that all the products $\mathsf{L}^1$, $\mathsf{U}^1$, $\mathsf{U}^2$ and
$\mathsf{L}^2$ in (\ref{L-U}) can be treated by
the summation formulas (\ref{Cor1}) and (\ref{Cor2}) in Corollary
\ref{Cor:Bressoud}, which are recast as
\begin{align}
\sum_{k} {\mathcal A}^{-1}_{ik}(a) (b/a)^k {\mathcal A}_{kj}(b) =
a^{-i} {\mathcal B}_{ij}(a,b) b^j, \qquad
\sum_{k} {\mathcal A}^{-1}_{ik}(a) {\mathcal A}_{kj}(b)=
q^{i^2}{\mathcal B}_{ij}(a,b) q^{-j^2} . \label{A-to-B}
\end{align}

\begin{Proposition}
We have
\begin{align}
&\mathsf{L}^1_{ij}
={q^{i^2}g^L_i \over
h^A_i}
{\mathcal B}_{ij}(q^{-n} w/\alpha ,q^{-n} z/\alpha )
{ h^R_j\over q^{j^2}g^L_j},\label{prp-LU-1}\\
&\mathsf{U}^1_{ij}
={(q^{-n} z/\alpha)^{i} h^R_i \over
 g^U_i}
{\mathcal B}_{ji}( q^{-n} \beta/w,q^{-n} z/\alpha )
{g'^U_j\over
(q^{-n} \beta/w)^{j} h'^A_j },\label{prp-LU-2}\\
&\mathsf{U}^2_{ij}
=
{(q^{-n} w/\alpha)^{i} h^A_i \over
w^i g^U_i}
{\mathcal B}_{ji}(q^{-n}\beta/z ,q^{-n} w/\alpha )
{g'^U_j\over
(q^{-n}\beta/z)^{j} h'^R_j },\label{prp-LU-3}\\
&\mathsf{L}^2_{ij}
=
{q^{i^2} g'^L_i\over
h'^R_i}
{\mathcal B}_{ij}(q^{-n}\beta/z ,q^{-n} \beta/w )
{w^{j} h'^A_j \over
 q^{j^2} g'^L_j}.\label{prp-LU-4}
\end{align}

\end{Proposition}

\begin{proof}
Straightforward calculation by using $D_2=\bigl(w^i \delta_{ij}\bigr)_{i,j=0}^n$, (\ref{gauge}), (\ref{lr-prp-1})--(\ref{lr-prp-4}), (\ref{lr-5}), (\ref{lr-6}),
and (\ref{A-to-B}).
\end{proof}

\subsection[Products L\^{}1 cdot D\_R cdot U\^{}1 and U\^{}2 cdot D'\_R cdot L\^{}2 in terms of terminating very-well-poised balanced \_\{10\}phi\_9 series]{Products $\boldsymbol{\mathsf{L}^1\cdot D_R\cdot \mathsf{U}^1}$ and $\boldsymbol{\mathsf{U}^2 \cdot D'_R \cdot \mathsf{L}^2}$ in terms of terminating \\ very-well-poised balanced $\boldsymbol{{}_{10}\phi_9}$ series}

\begin{Proposition}\label{prp-LDU-UDL-1}
We have
\begin{gather}
{1\over \mathsf{L}^1_{i0}\cdot d^R_{0}\cdot \mathsf{U}^1_{0j}}
\bigl(\mathsf{L}^1\cdot D_R\cdot \mathsf{U}^1\bigr)_{ij}\nonumber\\
\qquad{}={}_{10}W_9
\bigl(q^{-n}z/\alpha;q^{-i},q^{-j},q^{-n+i}w/\alpha,q^{-n+j}\beta/w,z/\alpha\beta,qz/\alpha,q^{-n+1}z;q,q \bigr),\label{final-1}\\
{1\over \mathsf{U}^2_{in} \cdot d'^R_n \cdot \mathsf{L}^2_{nj}}
\bigl(\mathsf{U}^2 \cdot D'_R \cdot \mathsf{L}^2\bigr)_{ij}\nonumber\\
\qquad{}={}_{10}W_9
\bigl(q^{-n}z/\beta;q^{-n+i},q^{-n+j},q^{-i}\alpha/w,q^{-j}w/\beta,z/\alpha\beta,qz/\beta,q^{-n+1}z;q,q \bigr),\label{final-2}
\end{gather}
and
\begin{align}
{ \mathsf{U}^2_{in} \cdot d'^R_n \cdot \mathsf{L}^2_{nj} \over
\mathsf{L}^1_{i0}\cdot d^R_{0}\cdot \mathsf{U}^1_{0j}}
\cdot { d^A_i \over d'^A_j}
&{}=
{
(1/\alpha,1/\beta,w/z,\alpha\beta/zw;q)_n
\over
(\alpha/z,\beta/z,1/w,w/\alpha\beta;q)_n}
{
\bigl(\beta/z,w/\alpha\beta,q^{-n+1}z/\alpha,q^{-n+1}w;q\bigr)_i
\over
\bigl(1/\alpha,w/z,q^{-n+1}\beta,q^{-n+1}zw/\alpha\beta;q\bigr)_i
}\nonumber\\
&\quad{}\times
{
\bigl(\beta/z,1/w,q^{-n+1}z/\alpha,q^{-n+1}\alpha\beta/w;q\bigr)_j
\over
\bigl(1/\alpha,\alpha \beta/zw,q^{-n+1}\beta,q^{-n+1}z/w;q\bigr)_j
}. \label{final-3}
\end{align}
\end{Proposition}

\begin{proof}
For (\ref{final-1}), we have from (\ref{lem-d-2}), (\ref{prp-LU-1}), (\ref{prp-LU-2}), and (\ref{lem-B-2}), (\ref{lem-gauge-2}) below that
\begin{align*}
&{1\over \mathsf{L}^1_{i0}\cdot d^R_{0}\cdot \mathsf{U}^1_{0j}}
\bigl(\mathsf{L}^1\cdot D_R\cdot \mathsf{U}^1\bigr)_{ij}
=\sum_{k=0}^{\min(i,j)}
{\mathsf{L}^1_{ik}\cdot d^R_{k}\cdot \mathsf{U}^1_{kj}\over \mathsf{L}^1_{i0}\cdot d^R_{0}\cdot \mathsf{U}^1_{0j}}\\
&\qquad{}=
\sum_{k=0}^{\min(i,j)}
{{\mathcal B}_{ik}(q^{-n} w/\alpha ,q^{-n} z/\alpha ) \over
{\mathcal B}_{i0}(q^{-n} w/\alpha ,q^{-n} z/\alpha )}
{ h^R_k\over q^{k^2}g^L_k} {d^R_{k} \over d^R_{0} }
{(q^{-n} z/\alpha)^{k} h^R_k \over
 g^U_k}
{{\mathcal B}_{jk}( q^{-n} \beta/w,q^{-n} z/\alpha ) \over
{\mathcal B}_{j0}( q^{-n} \beta/w,q^{-n} z/\alpha )}\\
&\qquad{}={}_{10}W_9
\bigl(q^{-n}z/\alpha;q^{-i},q^{-j},q^{-n+i}w/\alpha,q^{-n+j}\beta/w,z/\alpha\beta,qz/\alpha,q^{-n+1}z;q,q \bigr).
\end{align*}
For (\ref{final-2}), in a similar manner, we have from (\ref{lem-d-4}), (\ref{prp-LU-3}), (\ref{prp-LU-4}), and (\ref{lem-B-4}), (\ref{lem-gauge-4}) below that
\begin{align*}
&{1\over \mathsf{U}^2_{in} \cdot d'^R_n \cdot \mathsf{L}^2_{nj}}
\bigl(\mathsf{U}^2 \cdot D'_R \cdot \mathsf{L}^2\bigr)_{ij}
=\sum_{k=\max(i,j)}^n
{\mathsf{U}^2_{ik} \cdot d'^R_k \cdot \mathsf{L}^2_{kj}\over
\mathsf{U}^2_{in} \cdot d'^R_n \cdot \mathsf{L}^2_{nj}}
=
\sum_{l=0}^{\min(i,j)}
{\mathsf{U}^2_{i,n-l} \cdot d'^R_{n-l} \cdot \mathsf{L}^2_{n-l,j}\over
\mathsf{U}^2_{in} \cdot d'^R_n \cdot \mathsf{L}^2_{nj}}\\
&\qquad{}=
\sum_{l=0}^{\min(i,j)}
{\mathcal{B}_{n-l,i}(q^{-n} \beta/z,q^{-n}w/\alpha) \over
\mathcal{B}_{ni}(q^{-n} \beta/z,q^{-n}w/\alpha)}
{g'^U_{n-l}\over g'^U_{n}}
{(q^{-n}\beta/z)^n h'^R_n \over (q^{-n}\beta/z)^{n-l} h'^R_{n-l} }
{d'^R_{n-l}\over d'^R_{n}}\\
&\qquad\qquad\qquad{}\times
{q^{(n-l)^2}g'^L_{n-l} \over q^{n^2}g'^L_{n} }
{h'^R_n \over h'^R_{n-l} }
{\mathcal{B}_{n-l,j}(q^{-n} \beta/z,q^{-n}\beta/w) \over
\mathcal{B}_{nj}(q^{-n} \beta/z,q^{-n}\beta/w)}\\
&\qquad{}={}_{10}W_9
\bigl(q^{-n}z/\beta;q^{-n+i},q^{-n+j},q^{-i}\alpha/w,q^{-j}w/\beta,z/\alpha\beta,qz/\beta,q^{-n+1}z;q,q \bigr).
\end{align*}
Concerning (\ref{final-3}), one can proceed as
\begin{align*}
&{ \mathsf{U}^2_{in} \cdot d'^R_n \cdot \mathsf{L}^2_{nj} \over
\mathsf{L}^1_{i0}\cdot d^R_{0}\cdot \mathsf{U}^1_{0j}}
\cdot {d^A_i \over d'^A_0}\\&\qquad{}=
\mathcal{B}_{n0}(q^{-n}\beta/z,q^{-n}w/\alpha)\mathcal{B}_{n0}(q^{-n}\beta/z,q^{-n}\beta/w) {d^A_0\over d'^A_0}
{d'^R_n \over d^R_0}
{g'^U_n\over (q^{-n}\beta/z)^n h'^R_n}
{q^{n^2} g'^L_n\over h'^R_n}
\\
&\qquad\quad{}\times
{\mathcal{B}_{ni}(q^{-n}\beta/z,q^{-n}w/\alpha)\over \mathcal{B}_{n0} (q^{-n}\beta/z,q^{-n}w/\alpha)}
{1\over
\mathcal{B}_{i0} (q^{-n}w/\alpha,q^{-n}z/\alpha)}
{(q^{-n}w/\alpha)^i h^A_i h^A_i \over q^{i^2}g^L_i w^i g^U_i}
{d^A_i\over d^A_0}
\\
&\qquad\quad{}\times
{\mathcal{B}_{nj}(q^{-n}\beta/z,q^{-n}\beta/w)\over\mathcal{B}_{n0}(q^{-n}\beta/z,q^{-n}\beta/w)}
{1\over
\mathcal{B}_{j0} (q^{-n}\beta/w,q^{-n}z/\alpha)}
{(q^{-n}\beta/w)^j h'^A_j w^j h'^A_j \over
g'^U_j q^{j^2}g'^L_j} {d'^A_0 \over d'^A_j}.
\end{align*}
Then from (\ref{lem-n-term}), (\ref{lem-i-term}), and (\ref{lem-j-term}) below, we obtain (\ref{final-3}).
\end{proof}

\begin{Lemma}
We have
\begin{align}
&
{\mathcal{B}_{ik}(a,b) \over \mathcal{B}_{i0}(a,b) }
=\bigl(a^{-1}q\bigr)^k {\bigl(a q^i;q\bigr)_k \bigl(q^{-i};q\bigr)_k\over \bigl(b q^{i+1};q\bigr)_k \bigl(a^{-1} b q^{1-i};q\bigr)_k},\label{lem-B-2}\\
&
{\mathcal{B}_{n-l,i}(a,b) \over \mathcal{B}_{n,i}(a,b) }
=\bigl(a^{-2}b\bigr)^l
{\bigl(a^{-1} q^{1-2n};q\bigr)_{2l}\over \bigl(a^{-1} q^{-2n};q\bigr)_{2l}}
 {\bigl(b^{-1} q^{-n-i};q\bigr)_l \bigl(q^{-n+i};q\bigr)_l\over \bigl(a^{-1} q^{1-n-i};q\bigr)_l \bigl(a^{-1} b q^{1-n+i};q\bigr)_l},\label{lem-B-4}\\
 &\mathcal{B}_{i0}(a,b)=b^i
{(aq;q)_{2i}\over (a;q)_{2i}}
{(a;q)_i(a/b;q)_i\over (bq;q)_i(q;q)_i},\label{lem-B-1}\\
&
{\mathcal{B}_{n,i}(a,b) \over \mathcal{B}_{n,0}(a,b) }=
 \bigl(a^{-1}q\bigr)^i
{(a q^n;q)_i (q^{-n};q)_i\over \bigl(b q^{1+n};q\bigr)_i \bigl(a^{-1} b q^{1-n};q\bigr)_i}.\label{lem-B-3}
\end{align}
\end{Lemma}

\begin{Lemma}
We have
\begin{align}
&
{h^R_k \over q^{k^2} g^L_k}
{( q^{-n} z/\alpha)^k h^R_k \over g^U_k}=
q^{-k^2} \bigl( q^{-n-1}\beta \bigr)^{k}
{
\bigl( q^{-n+1}z/\alpha ;q\bigr)_{2k} \bigl(q^{-n+1}z/\alpha;q\bigr)_{2k} \over
(1/\alpha;q)_k (q^{-n};q)_k \bigl(q^{-n+1} \beta ;q\bigr)_k(q;q)_k},\label{lem-gauge-2}\\
&
{g'^U_{n-l}\over g'^U_{n}}
{(q^{-n}\beta/z)^n h'^R_n \over (q^{-n}\beta/z)^{n-l} h'^R_{n-l} }
{q^{(n-l)^2}g'^L_{n-l} \over q^{n^2}g'^L_{n} }
{h'^R_n \over h'^R_{n-l} }\nonumber\\
&\qquad{}=
q^{-l^2} (q^{-n+1} \beta^3/z^4 )^l
{
(q^{-n}z/\beta;q)_{2l} (q^{-n}z/\beta;q)_{2l} \over
(1/\beta;q)_l (q^{-n};q)_l \bigl(q^{-n+1}\alpha ;q\bigr)_l(q;q)_l}.\label{lem-gauge-4}
\end{align}
\end{Lemma}

\begin{Lemma}
We have
\begin{align}
&\mathcal{B}_{n0}(q^{-n}\beta/z,q^{-n}w/\alpha)\mathcal{B}_{n0}(q^{-n}\beta/z,q^{-n}\beta/w) {d^A_0\over d'^A_0}
{d'^R_n \over d^R_0}
{g'^U_n\over (q^{-n}\beta/z)^n h'^R_n}
{q^{n^2} g'^L_n\over h'^R_n}\nonumber\\
&\qquad{}={
(1/\alpha,1/\beta,w/z,\alpha\beta/zw;q)_n
\over
(\alpha/z,\beta/z,1/w,w/\alpha\beta;q)_n},\label{lem-n-term}\\
&
{\mathcal{B}_{ni}(q^{-n}\beta/z,q^{-n}w/\alpha)\over \mathcal{B}_{n0} (q^{-n}\beta/z,q^{-n}w/\alpha)}
{1\over
\mathcal{B}_{i0} (q^{-n}w/\alpha,q^{-n}z/\alpha)}
{(q^{-n}w/\alpha)^i h^A_i h^A_i \over q^{i^2}g^L_i w^i g^U_i}
{d^A_i\over d^A_0}\nonumber\\
&\qquad{}=
{
\bigl(\beta/z,w/\alpha\beta,q^{-n+1}z/\alpha,q^{-n+1}w;q\bigr)_i
\over
\bigl(1/\alpha,w/z,q^{-n+1}\beta,q^{-n+1}zw/\alpha\beta;q\bigr)_i},\label{lem-i-term}\\
&
{\mathcal{B}_{nj}(q^{-n}\beta/z,q^{-n}\beta/w)\over\mathcal{B}_{n0}(q^{-n}\beta/z,q^{-n}\beta/w)}
{1\over
\mathcal{B}_{j0} (q^{-n}\beta/w,q^{-n}z/\alpha)}
{(q^{-n}\beta/w)^j h'^A_j w^j h'^A_j \over
g'^U_j q^{j^2}g'^L_j} {d'^A_0 \over d'^A_j}\nonumber\\
&\qquad{}=
{
\bigl(\beta/z,1/w,q^{-n+1}z/\alpha,q^{-n+1}\alpha\beta/w;q\bigr)_j
\over
\bigl(1/\alpha,\alpha \beta/zw,q^{-n+1}\beta,q^{-n+1}z/w;q\bigr)_j}.\label{lem-j-term}
\end{align}
\end{Lemma}

\subsection[Final step of proof with Bailey's transformation formula for terminating very-well-poised balanced \_\{10\}phi\_9 series]{Final step of proof with Bailey's transformation formula \\ for terminating very-well-poised balanced $\boldsymbol{{}_{10}\phi_9}$ series}
Recall that Bailey's transformation formula for
terminating very-well-poised balanced ${}_{10}\phi_9$ series \cite[formulas~(2.9.4) and (2.9.5)]{hypergeometric} reads
\begin{align}
&{}_{10}W_9(a;b,c,d,e,f,g,h;q,q)\nonumber\\
&\qquad{}=
{(aq,aq/ef,aq/eg,aq/eh,aq/fg,aq/fh,aq/gh,aq/efgh;q)_\infty \over
(aq/e,aq/f,aq/g,aq/h,aq/efg,aq/efh,aq/egh,aq/fgh;q)_\infty}\nonumber\\
&\qquad\quad{}\times
{}_{10}W_9\bigl(qa^2/bcd;aq/bc,aq/bd,aq/cd,e,f,g,h;q,q\bigr),
\end{align}
where at least one of the parameters $e$, $f$, $g$, $h$ is of the form $q^{-n}$, $n=0,1,2,\dots$,
and the balancing condition is satisfied, namely
\begin{align}
q^2 a^3=bcdefgh. \label{cond}
\end{align}
In particular, by setting $h=q^{-n}$, we have
\begin{align}
&{}_{10}W_9(a;b,c,d,e,f,g,q^{-n};q,q)\label{Bailey}\\
&\qquad{}=
{(aq,aq/ef,aq/eg,aq/fg;q)_n\over
(aq/e,aq/f,aq/g,aq/efg;q)_n}
{}_{10}W_9\bigl(qa^2/bcd;aq/bc,aq/bd,aq/cd,e,f,g,q^{-n};q,q\bigr).\nonumber
\end{align}

\begin{Proposition}\label{10W9}
Let $n,i,j\in\mathbb{Z}_{\geq 0}$ and $i,j\leq n$.
Applying Bailey's transformation twice,
we have
\begin{align}
&
{}_{10}W_9
\bigl(q^{-n}z/\alpha;q^{-i},q^{-j},q^{-n+i}w/\alpha,q^{-n+j}\beta/w,z/\alpha\beta,qz/\alpha,q^{-n+1}z;q,q \bigr)\nonumber\\
&\qquad{}=
{
(1/\alpha,1/\beta,w/z,\alpha\beta/zw;q)_n
\over
(\alpha/z,\beta/z,1/w,w/\alpha\beta;q)_n}
{
\bigl(\beta/z,w/\alpha\beta,q^{-n+1}z/\alpha,q^{-n+1}w;q\bigr)_i
\over
\bigl(1/\alpha,w/z,q^{-n+1}\beta,q^{-n+1}zw/\alpha\beta;q\bigr)_i
}\nonumber\\
&\qquad\quad{}\times
{
\bigl(\beta/z,1/w,q^{-n+1}z/\alpha,q^{-n+1}\alpha\beta/w;q\bigr)_j
\over
\bigl(1/\alpha,\alpha \beta/zw,q^{-n+1}\beta,q^{-n+1}z/w;q\bigr)_j
}\nonumber\\
&\qquad\quad{}\times {}_{10}W_9
\bigl(q^{-n}z/\beta;q^{-n+i},q^{-n+j},q^{-i}\alpha/w,q^{-j}w/\beta,z/\alpha\beta,qz/\beta,q^{-n+1}z;q,q \bigr).\label{Bailey-2}
\end{align}
\end{Proposition}

\begin{proof}
Observe that the balancing condition as in (\ref{cond}) is satisfied
\begin{align*}
q^2 \biggl( {q^{-n}z\over\alpha} \biggr)^3=
q^{-i}q^{-j}
{ q^{-n+i}w\over \alpha}{q^{-n+j}\beta\over w}{z\over \alpha\beta}{qz\over \alpha}{q^{-n+1}z}.
\end{align*}
Hence, Bailey's transformation formula (\ref{Bailey}) applies and we have
\begin{align*}
&
{}_{10}W_9
\bigl(q^{-n}z/\alpha;q^{-i},q^{-j},q^{-n+i}w/\alpha,q^{-n+j}\beta/w,z/\alpha\beta,qz/\alpha,q^{-n+1}z;q,q \bigr)\\
&\qquad{}=
{}_{10}W_9
\bigl(q^{-n}z/\alpha;q^{-j},q^{-n+j}\beta/w,qz/\alpha,q^{-n+i}w/\alpha,z/\alpha\beta,q^{-n+1}z,q^{-i};q,q \bigr)\\
&\qquad{}=
{\bigl(q^{-n+1}z/\alpha, q^{-i+1}\alpha\beta/w, q^{-i+n}/w,\beta/z;q\bigr)_i
\over
\bigl(q^{-i+1}z/w,q^{-n+1}\beta,1/\alpha,q^{-i+n}\alpha\beta/zw;q\bigr)_i } \\
&\qquad\quad{}\times
{}_{10}W_9
\bigl(q^{-n}zw/\alpha\beta;qzw/\alpha\beta,q^{-n+j}, q^{-j}w/\beta, q^{-n+i}w/\alpha, z/\alpha\beta, q^{-n+1}z, q^{-i};q,q \bigr).
\end{align*}
Applying Bailey's transformation once again, we have
\begin{align*}
&{}_{10}W_9
\bigl(q^{-n}zw/\alpha\beta; qzw/\alpha\beta,q^{-n+j}, q^{-j}w/\beta, q^{-n+i}w/\alpha, z/\alpha\beta, q^{-n+1}z, q^{-i};q,q \bigr)\\
&\qquad{}={}_{10}W_9
\bigl(q^{-n}zw/\alpha\beta; q^{-i}, q^{-n+i}w/\alpha, qzw/\alpha\beta, q^{-j}w/\beta, z/\alpha\beta, q^{-n+1}z,q^{-n+j} ;q,q \bigr)\\
&\qquad{}=
{\bigl(q^{-n+1}zw/\alpha\beta,q^{-n+j+1}\beta,q^{j}/\alpha,w/z;q\bigr)_{n-j}
\over \bigl(q^{-n+j+1}z/\alpha,q^{-n+1} w,w/\alpha\beta,q^j \beta/z;q\bigr)_{n-j}
} \\
&\qquad\quad{}\times
{}_{10}W_9
\bigl(q^{-n}z/\beta; qz/\beta, q^{-n+i}, q^{-i}\alpha/w, q^{-j}w/\beta, z/\alpha\beta, q^{-n+1}z,q^{-n+j} ;q,q \bigr).
\end{align*}
Simplifying the prefactor we have (\ref{Bailey-2}).
\end{proof}

Now we are ready to state our proof of Theorem \ref{thm-TBP}, therewith its reformulation Theorem~\ref{TBP-2}.

\begin{proof}[Proof of Theorem \ref{thm-TBP}]
{}In view of (\ref{final-1})--(\ref{final-3}) in Proposition \ref{prp-LDU-UDL-1}, and
(\ref{Bailey-2}) in Proposition \ref{10W9}, we have the equality of matrices
\begin{align*}
{1\over \mathsf{L}^1_{i0}\cdot d^R_{0}\cdot \mathsf{U}^1_{0j}}
\mathsf{L}^1\cdot D_R\cdot \mathsf{U}^1 =
{1\over \mathsf{L}^1_{i0}\cdot d^R_{0}\cdot \mathsf{U}^1_{0j}}
D_A\cdot \mathsf{U}^2 \cdot D'_R \cdot \mathsf{L}^2\cdot {D'_A}^{\!\!-1},
\end{align*}
which is equivalent to the commutativity $RD_2A=A R D_2$.
\end{proof}

\section{Truncation by tuning mass parameters }
\label{App:Tuning}

In \cite{Awata:2022idl}, we employed the following orbifolded Nekrasov factor for the affine Laumon space:
\begin{align}\label{Macdonaldform}
&\Nk^{(k|N)}_{\lambda,\mu}(u|q,\kappa)=
\Nk^{(k)}_{\lambda,\mu}(u|q,\kappa) \\
&=
 \prod_{j\geq i\geq 1 \atop j-i \equiv k \,\,({\rm mod}\,N)}
\bigl[u q^{-\mu_i+\lambda_{j+1}} \kappa^{-i+j};q\bigr]_{\lambda_j-\lambda_{j+1}}
\cdot
\prod_{\beta\geq \alpha \geq 1 \atop \beta-\alpha \equiv -k-1 \,\,({\rm mod}\,N)}
\bigl[u q^{\lambda_{\alpha}-\mu_\beta} \kappa^{\alpha-\beta-1};q\bigr]_{\mu_{\beta}-\mu_{\beta+1}}, \nonumber
\end{align}
with $[u;q]_n:=u^{-n/2}q^{-n(n-1)/4} (u;q)_n$.
To work out the tuning condition of mass parameters required for the truncation of the Young diagrams,
it is convenient to recast \eqref{Macdonaldform} to the ``Nakajima'' form,
which is expressed as a product over the boxes of the Young diagrams $(\lambda, \mu)$.
In the following, we temporally replace $[u;q]_n$ by $(u;q)_n$, since we can neglect the monomial factors for our purpose.

By taking the product over $\Nk^{(0)}$ to $\Nk^{(N-1)}$, we can remove the selection rule,
\begin{align}
\Nk_{\lambda,\mu}(u|q,\kappa):={}& \prod_{k=0}^{N-1} \Nk^{(k)}_{\lambda,\mu}(u|q,\kappa) \nonumber \\
={}& \!\prod_{j\geq i\geq 1}\!\bigl(u q^{-\mu_i+\lambda_{j+1}} \kappa^{-i+j};q\bigr)_{\lambda_j-\lambda_{j+1}}
\cdot\! \prod_{\beta\geq \alpha \geq 1} \!\bigl(u q^{\lambda_{\alpha}-\mu_\beta} \kappa^{\alpha-\beta-1};q\bigr)_{\mu_{\beta}-\mu_{\beta+1}}.\!
\end{align}
Namely, we can obtain $\Nk^{(k|N)}_{\lambda,\mu}(u|q,\kappa)$ by factorizing
the total Nekrasov factor $\Nk_{\lambda,\mu}(u|q,\kappa)$ according to the power of $\kappa$.
Now we can transform the total Nekrasov factor as follows:
\begin{align}
\Nk_{\lambda,\mu}(u|q,\kappa)
={}& \prod_{j\geq i\geq 1} \frac{\bigl(u q^{-\mu_i+\lambda_{j+1}} \kappa^{-i+j};q\bigr)_\infty}
{\bigl(u q^{-\mu_i+\lambda_{j}} \kappa^{-i+j};q\bigr)_\infty}
\cdot \prod_{i\geq j \geq 1}\frac{\bigl(u q^{\lambda_{j}-\mu_i} \kappa^{j-i-1};q\bigr)_\infty}
{\bigl(u q^{\lambda_{j}-\mu_{i+1}} \kappa^{j-i-1};q\bigr)_\infty} \nonumber \\
={}& \prod_{i,j =1}^\infty \frac{\bigl(u q^{\lambda_j - \mu_i} \kappa^{j-i-1};q\bigr)_\infty}
{\bigl(u q^{\lambda_j - \mu_i} \kappa^{j-i};q\bigr)_\infty}
\cdot \frac{\bigl(u \kappa^{j-i};q\bigr)_\infty}{\bigl(u \kappa^{j-i-1};q\bigr)_\infty} \nonumber \\
={}& \prod_{(i,j) \in \lambda} \bigl(1- u q^{\lambda_i-j} \kappa^{-\mu^{\vee}_j + i-1}\bigr) \cdot
\prod_{(i,j) \in \mu} \bigl(1- u q^{-\mu_i+j-1} \kappa^{\lambda^{\vee}_j -i}\bigr).
\end{align}
where we have used the combinatorial identity \big(cf.\ proposition in \cite{Awata:2008ed}, $(2.12)=(2.9)$ with $\kappa=t^{-1}$\big)
for the last equality.
When one of the partitions is empty as in the case of fundamental matter contributions, the Nekrasov factor is simplified
considerably,
\begin{gather}
\Nk_{\lambda,\varnothing}(u|q,\kappa) = \prod_{(i,j) \in \lambda} \bigl(1- u q^{\lambda_i-j} \kappa^{i-1}\bigr),
\nonumber\\
\Nk_{\varnothing,\mu}(u|q,\kappa)= \prod_{(i,j) \in \mu} \bigl(1- u q^{-\mu_i+j-1} \kappa^{-i}\bigr).
\end{gather}
When $N=2$, the factors from the first rows of $\lambda$ and $\mu$ are contained in \smash{$\Nk_{\lambda,\varnothing}^{(0)}(u|q,\kappa)$}
and \smash{$ \Nk_{\varnothing,\mu}^{(1)}(u|q,\kappa)$}, respectively, which are explicitly
\begin{equation}
\prod_{k=0}^{\lambda_1-1}\bigl(1-uq^k\bigr), \qquad \prod_{k=1}^{\mu_1} \bigl(1- u q^{-k} \kappa^{-1}\bigr).
\end{equation}
Hence, we can see that if $u=q^{-m}$, \smash{$\Nk_{\lambda,\varnothing}^{(0)}(u|q,\kappa)$} vanishes for $m \leq \lambda_1-1$.
Similarly, if~${u \kappa^{-1}\! =q^{n+1}}$, \smash{$\Nk_{\varnothing,\mu}^{(1)}(u|q,\kappa)$} vanishes for $n+1 \leq \mu_1$.

In the localization formula of the partition function,
the numerator of the contribution from a fixed point \smash{$\bigl(\lambda^{(1)},\lambda^{(2)}\bigr)$}, which is identified
with the matter contribution is explicitly given by%
\begin{align}
&\prod_{i,j=1}^2 \mathsf{N}^{(j-i|2)}_{\varnothing,\lambda^{(j)}} (u_i/v_j |q,\kappa)
\mathsf{N}^{(j-i|2)}_{\lambda^{(i)},\varnothing} (v_i/w_j |q,\kappa) \nonumber \\
&\quad{}=\mathsf{N}^{(0)}_{\varnothing,\lambda^{(1)}} (u_1/v_1 |q,\kappa)
\mathsf{N}^{(1)}_{\varnothing,\lambda^{(2)}} (u_1/v_2 |q,\kappa)
\mathsf{N}^{(1)}_{\varnothing,\lambda^{(1)}} (u_2/v_1 |q,\kappa)
\mathsf{N}^{(0)}_{\varnothing,\lambda^{(2)}} (u_2/v_2 |q,\kappa) \nonumber \\
&\qquad{}\times
\mathsf{N}^{(0)}_{\lambda^{(1)},\varnothing} (v_1/w_1 |q,\kappa)
\mathsf{N}^{(1)}_{\lambda^{(1)},\varnothing} (v_1/w_2|q,\kappa)
\mathsf{N}^{(1)}_{\lambda^{(2)},\varnothing} (v_2/w_1|q,\kappa)
\mathsf{N}^{(0)}_{\lambda^{(2)},\varnothing} (v_2/w_2|q,\kappa).
\end{align}
Let us look at the mass parameter tuning conditions such that
the columns of the Young diagrams of $\lambda^{(1)}$ and
$\lambda^{(2)}$ are at most $m$ and $n$, respectively.
We need a zero coming from
\smash{$\mathsf{N}^{(0)}_{\lambda^{(1)},\varnothing} (v_1/w_1|q,\kappa)$}
or \smash{$\mathsf{N}^{(1)}_{\varnothing,\lambda^{(1)}} (u_2/v_1 |q,\kappa)$}, namely,
\beq\label{pit1}
(v_1/w_1)= q^{-m}, \qquad (u_2/v_1)\kappa^{-1} = q^{m+1}.
\eeq
Similarly, we need a zero coming from
\smash{$\mathsf{N}^{(0)}_{\lambda^{(2)},\varnothing} (v_2/w_2|q,\kappa)$}
or \smash{$\mathsf{N}^{(1)}_{\varnothing,\lambda^{(2)}} (u_1/v_2 |q,\kappa)$}, namely,
\beq\label{pit2}
(v_2/w_2) = q^{-n}, \qquad (u_1/v_2) \kappa^{-1} = q^{n+1}.
\eeq

Recall that we use the following specialization (see \eqref{Yamada}) in this paper:
\beq
u_1 = \frac{qQ}{d_3}, \quad u_2 = \frac{\kappa q}{d_1}, \qquad
v_1 = 1, \quad v_2 = \frac{Q}{\kappa}, \qquad
w_1 = \frac{1}{d_2}, \quad w_2 = \frac{Q}{d_4 \kappa}.
\eeq
Hence, the above conditions \eqref{pit1} and \eqref{pit2} are translated to
\beq
d_2 = q^{-m}, \qquad d_1 = q^{-m},
\eeq
and
\beq
d_4 = q^{-n}, \qquad d_3 = q^{-n}.
\eeq
The conditions for $d_1$ and $d_2$ are the same and those for $d_3$ and $d_4$ are also the same.
Thus, there are four equivalent choices for the truncation of the Young diagrams.
In the main text of the paper, we choose $d_2=q^{-m}$ and $d_3=q^{-n}$.


\section{Affine Laumon space and orbifolded Nekrasov factor}
\label{App:Laumon}

The affine Laumon space is the moduli space of parabolic torsion free sheaves
on $\mathbb{P}^1 \times \mathbb{P}^1$. Originally in \cite{Laumon1,Laumon2},
a partial compactification of the space of quasi-maps from $\mathbb{P}^1$ to the flag
variety of ${\rm GL}_N$ was defined. The geometric representation theory of the compactification is
controlled by $\mathfrak{gl}_N$. There is an action of the $\mathfrak{gl}_N$-Yangian
on the equivariant cohomology of the Laumon space~\cite{FFNR}. The affine Laumon space is an affine analogue,
where the base manifold is replaced by $\mathbb{P}^1 \times \mathbb{P}^1$
and we have an action of the affine Yangian of $\mathfrak{gl}_N$.
In \cite{Negut2}, it was shown that~$\qg\big(\widehat{\mathfrak{gl}}_n\big)$ acts on the equivariant $K$-theory of the affine Laumon space,
which cannot be irrelevant to the present work.
In algebraic geometry such ``affine'' quasi-maps are described by torsion free sheaves on $\mathbb{P}^1 \times \mathbb{P}^1$
with a canonical framing at $\{\infty\} \times \{\infty\}$ and a parabolic structure on $\mathbb{P}^1 \times \{ 0 \}$.

In the context of four-dimensional $\mathcal{N}=2$ supersymmetric gauge theories
the affine Laumon space was first featured in \cite{Braverman:2004vv,Braverman:2004cr},
where the conjecture on the relation of the Nekrasov partition function and the Seiberg--Witten
prepotential was proved. Then as a generalization of the AGT correspondence for Virasoro and $\mathcal{W}_N$ conformal blocks,
Alday--Tachikawa \cite{Alday:2010vg} pointed out that we could employ the affine Laumon space to describe
the theory with a surface defect whose classification by the embedding of $\mathfrak{sl}_2$ subalgebra into $\mathfrak{sl}_N$
is exactly the same as that of the parabolic structure on $\mathbb{P}^1 \times \{ 0 \}$.
In~\cite{Alday:2010vg}, it is conjectured when the defect is of the full type the instanton partition function
gives a conformal block for the $\mathfrak{sl}_N$ current algebra.\footnote{For the codimension two surface defect
of general type the conjecture is that the corresponding chiral algebra is the $\mathcal{W}$ algebra obtained
from Drinfeld--Sokolov reduction of $\widehat{\mathfrak{sl}}_N$ current algebra, see, for example,~\cite{Kanno:2011fw}.}

It is known that the parabolic sheaves are equally described by the sheaves with the $\mathbb{Z}_M$ orbifold action
$\mathbb{P}^1 \times \mathbb{P}^1 \ni (z,w) \to (z,\omega w)$ with $\omega^M =1$.
Moreover a natural toric action on $\mathbb{P}^1 \times \mathbb{P}^1$ induces that on the affine Laumon space.
The fixed points of the toric action are labelled by $N$-tuples of Young diagrams, which allows us
to use the technique of the equivariant localization. In practice the relevant quiver is called
the chain-saw quiver \cite{FR}. In particular, the equivariant character of the Yangian action on the equivariant cohomology
of the affine Laumon space was computed in \cite{FFNR}. It is straightforward to up-lift the formula in \cite{FFNR}
to the $K$ theory version, which we use in the present paper.
On the method of orbifolding for the computation of the instanton partition function with a surface defect,
see, for example, \cite{Bullimore:2014awa,Kanno:2011fw,Nekrasov:2017rqy,Nekrasov:2017gzb,Nekrasov:2021tik}.

In \cite{BFS}, the relation of the asymptotically free Macdonald polynomials and the geometric representation theory of the Laumon space was clarified.
In an attempt at generalizing this result to an affine version, a formula of the non-stationary Ruijsenaars function
is obtained in~\cite{Shi} based on the affine screening operators.
The non-stationary Ruijsenaars function proposed in~\cite{Shi}
agrees with the instanton partition function obtained from the equivariant Euler character of 
the affine Laumon space.\footnote{Physically this corresponds to the gauge theory with adjoint hypermultiplet,
sometimes called $\mathcal{N}=2^{*}$ theory.} For the relation of the formula in \cite{Shi}, which we employ in the present paper,
and the formula of the equivariant character in \cite{FFNR}, see, for example, \cite[Appendix~B]{Awata:2019isq}.
See also \cite{Negut:2011aa} on the relations of the geometry of affine Laumon space, intertwines of the quantum toroidal algebra
and associated integrable systems.

As we have seen in Appendix \ref{App:Tuning}, the total Nekrasov factor
\begin{align}
\Nk_{\lambda,\mu}(u|q,\kappa)
= & \prod_{(i,j) \in \lambda} \bigl(1- u q^{\lambda_i-j} \kappa^{-\mu^{\vee}_j + i-1}\bigr) \cdot
\prod_{(i,j) \in \mu} \bigl(1- u q^{-\mu_i+j-1} \kappa^{\lambda^{\vee}_j -i}\bigr), \nonumber \\
= & \prod_{i,j =1}^\infty \frac{\bigl(u q^{\lambda_j - \mu_i} \kappa^{j-i-1};q\bigr)_\infty}
{\bigl(u q^{\lambda_j - \mu_i} \kappa^{j-i};q\bigr)_\infty}
\cdot \frac{\bigl(u \kappa^{j-i};q\bigr)_\infty}{\bigl(u \kappa^{j-i-1};q\bigr)_\infty}\label{q-base}
\end{align}
can be factorized into $\Nk^{(k\vert n)} _{\lambda,\mu}(u|q,\kappa)$, $0 \leq k \leq n-1$,
according to the power of $\kappa$. The combinatorial identities in \cite[(2.8) and (2.13)]{Awata:2008ed} also show
the following ``transposed'' formula for the total
Nekrasov factor:
\begin{align}\label{t-base}
\Nk_{\lambda,\mu}(u|q,t=\kappa^{-1})
= & \prod_{(i,j) \in \mu} \bigl(1- u q^{\lambda_i-j} t^{\mu^{\vee}_j - i+1}\bigr) \cdot
\prod_{(i,j) \in \lambda} \bigl(1- u q^{-\mu_i+j-1} t^{-\lambda^{\vee}_j +i}\bigr) \nonumber \\
= & \prod_{i,j =1}^\infty \frac{\bigl(u t^{\mu_i^{\vee} - \lambda_j^{\vee} +1} q^{j-i};t\bigr)_\infty}
{\bigl(u t^{\mu_i^{\vee} - \lambda_j^{\vee} +1} q^{j-i-1};t\bigr)_\infty}
\cdot \frac{\bigl(u t q^{j-i-1};t\bigr)_\infty}{\bigl(u t q^{j-i};t\bigr)_\infty}.
\end{align}
Compared with \eqref{q-base}, the partitions $\lambda$ and $\mu$ in the formula \eqref{t-base} are
exchanged in the product over the boxes of Young diagram, while the form of each factor is kept intact.
Factorization of~\eqref{t-base} according to the power of $\kappa$ gives the following formula
of the orbifolded Nekrasov factor:
\begin{gather}
\mathsf{N}_{\lambda \mu}^{(k\vert n)}(u|q,\kappa) =
\prod_{i \leq j} \prod_{\ell=0 \atop \lambda_{j+1}^\vee - \mu_i^\vee + \ell \equiv k}^{\lambda_{j}^{\vee} - \lambda_{j+1}^\vee -1}
\bigl[u q^{j-i} \kappa^{\lambda_{j+1}^\vee - \mu_i^\vee + \ell}\bigr] \nonumber\\ \hphantom{\mathsf{N}_{\lambda \mu}^{(k\vert n)}(u|q,\kappa) =}{}
\times
\prod_{i \leq j} \prod_{\ell=0 \atop \lambda_{i}^\vee - \mu_j^\vee + \ell \equiv k}^{\mu_{j}^{\vee} - \mu_{j+1}^\vee -1}
\bigl[u q^{i-j-1} \kappa^{\lambda_{i}^\vee - \mu_{j}^\vee + \ell}\bigr],\label{awata}
\end{gather}
where $[u]=u^{-1/2}-u^{1/2}$.
%
Remark that usually the Nekrasov factor
(such as $\Nk_{\lambda,\mu}(u|q,\kappa)$ as above) is expressed in terms of $(u;q)_n$. But
we find in the case of the affine Laumon space it is more appropriate to use
\begin{equation}
[u;q]_n=u^{-n/2}q^{-n(n-1)/4} (u;q)_n
= [u][q u]\cdots \bigl[q^{n-1}u\bigr].
\end{equation}

When $n \to \infty$, we have to regularize the monomial factor which connects $[u;q]_n$ and $(u;q)_n$.
One of the ways of the regularization is to define $[u;q]_\infty$ by
\beq\label{sinh-reg}
[u;q]_\infty := \frac{(u;q)_\infty}{\vartheta_{q^{1/2}}\bigl(- u^{1/2}\bigr)}
= \frac{\bigl(u^{1/2}; q^{1/2}\bigr)_\infty}{\bigl(-q^{1/2}u^{-1/2}; q^{1/2}\bigr)_\infty},
\eeq
where $\vartheta_p(z):= (z;p)_\infty \bigl(pz^{-1};p\bigr)_\infty$.
One can check
\begin{align}\label{sinh-ratio}
\frac{[u;q]_\infty}{[q^n u ; q]_\infty} &=
\frac{\bigl(u^{1/2} ; q^{1/2}\bigr)_\infty}{\bigl(-q^{1/2} u^{-1/2}; q^{1/2}\bigr)_\infty}
\frac{\bigl(-q^{(1-n)/2} u^{-1/2}; q^{1/2}\bigr)_\infty}{\bigl(q^{n/2} u^{1/2} ; q^{1/2}\bigr)_\infty}
\nonumber \\
&= \bigl(u^{1/2}; q^{1/2}\bigr)_n \bigl(-q^{\bigl(1-n\bigr)/2} u^{-1/2}; q^{1/2}\bigr)_n = [u; q]_n.
\end{align}
In particular, we have $[u;q]_\infty = \bigl(u^{-1/2} - u^{1/2}\bigr) [qu;q]_\infty$.
We also find
\begin{equation}
\frac{[u;q]_\infty}{[q/u;q]_\infty}
= \frac{\bigl(u^{1/2};q^{1/2}\bigr)_\infty}{\bigl(-q^{1/2}u^{-1/2};q^{1/2}\bigr)_\infty}
\frac{\bigl(-u^{1/2};q^{1/2}\bigr)_\infty}{\bigl(q^{1/2}u^{-1/2};q^{1/2}\bigr)_\infty}
= \frac{(u;q)_\infty}{(q/u;q)_\infty}.
\end{equation}

It is possible to express the result of the selection rule in the orbifolded Nekrasov factor \eqref{awata}
in terms of the floor function $\floor{\bullet}$.
In the following manipulations, we use the formulas for the floor function,
\beq\label{floor-formula}
\biggl\lfloor\frac{\ell}{n}\biggr\rfloor + 1 = - \biggl\lfloor\frac{-\ell-1}{n}\biggr\rfloor, \qquad
\biggl\lfloor\frac{\ell+m}{n}\biggr\rfloor = \biggl\lfloor\frac{\ell}{n}\biggr\rfloor + \biggl\lfloor\frac{(\ell)+m}{n}\biggr\rfloor, \qquad \ell, m \in \bbZ,
\eeq
where $(\ell)$ denotes the residue of an integer $\ell$ modulo $n$.

\begin{Proposition}\label{Shiraishi}
The orbifolded Nekrasov factor \eqref{awata} is given by the following formula:
\begin{align}\label{eq:Shiraishi}
\mathsf{N}_{\lambda \mu}^{(k\vert n)} &{}=
\prod_{i \leq j} \Bigl[q^{j-i} \kappa^{k + n\floor{\frac{\lambda_{j+1}^\vee + n -1 -k -\bigl(\mu_i^\vee\bigr)}{n}} -n \floor{\frac{\mu_i^\vee}{n}}}
; \kappa^n\Bigr]_{\floor{\frac{\lambda_{j}^\vee + n -1 -k - \bigl(\mu_i^\vee\bigr)}{n}} - \floor{\frac{\lambda_{j+1}^\vee + n -1 -k - \bigl(\mu_i^\vee\bigr)}{n}}} \nonumber \\
&\quad{}\times
\prod_{i \leq j} \Bigl[q^{i-j-1} \kappa^{k + n\floor{\frac{\lambda_{i}^\vee + n -1 -k-(\mu_j^\vee)}{n}} -n \floor{\frac{\mu_j^\vee}{n}}}
; \kappa^n\Bigr]_{\floor{\frac{\mu_{j}^\vee + k + (-\lambda_i^\vee)}{n}} - \floor{\frac{\mu_{j+1}^\vee + k + (-\lambda_i^\vee)}{n}}} \nonumber \\
&{}=
\prod_{i \leq j} \Bigl[q^{j-i} \kappa^{k + n\floor{\frac{\lambda_{j+1}^\vee + n -1 -k -\bigl(\mu_i^\vee\bigr)}{n}} -n \floor{\frac{\mu_i^\vee}{n}}}
; \kappa^n\Bigr]_{\floor{\frac{\lambda_{j}^\vee -1 -k - \bigl(\mu_i^\vee\bigr)}{n}} - \floor{\frac{\lambda_{j+1}^\vee -1 -k - \bigl(\mu_i^\vee\bigr)}{n}}} \nonumber \\
&\quad{}\times
\prod_{i \leq j} \Bigl[q^{i-j-1} \kappa^{k + n\floor{\frac{\lambda_{i}^\vee + n -1 -k-(\mu_j^\vee)}{n}} -n \floor{\frac{\mu_j^\vee}{n}}}
; \kappa^n\Bigr]_{\floor{\frac{\mu_{j}^\vee + k + (-\lambda_i^\vee)}{n}} - \floor{\frac{\mu_{j+1}^\vee + k + (-\lambda_i^\vee)}{n}}}.
\end{align}
\end{Proposition}

\begin{proof}
Since the selection rule on $\ell$ is imposed by $\equiv$ (the equality mod $n$),
$\mathsf{N}_{\lambda \mu}^{(k\vert n)}$ can be rewritten in terms of $\kappa^n$-shifted factorial $(z_0 ; \kappa^n)_m$.
Let us check that the initial value $z_0$ and the number $m$ of factors for each shifted factorial agree with those in the formula
\eqref{eq:Shiraishi}. To compute them for the first shifted factorial, let us assume
\beq
\mu_i^\vee = q_1 n + r_1, \qquad \lambda_{j+1}^\vee = q_2 n + r_2, \qquad \lambda_i^\vee = q_3 n + r_3, \qquad 0 \leq r_i\leq n-1.
\eeq
Namely, we may write, for example, $q_1 = \big\lfloor\frac{\mu_i^{\vee}}{n}\bigr\rfloor$ and $r_1 = \bigl(\mu_i^{\vee}\bigr)$. But in the following
we will use~$q_i$ and~$r_i$ for simplicity of notations.
The selection rule in \eqref{awata} tells
\beq
\ell \equiv k + r_1 - r_2.
\eeq
Since $0 \leq k, r_1, r_2 \leq n-1$, we have $1-n \leq k+r_1 - r_2 \leq 2n-2$.
Accordingly, we can write the initial value of $\ell$ uniformly as
\beq\label{initial}
\ell_0 = k + r_1 - r_2 - n \biggl\lfloor\frac{k + r_1 - r_2}{n}\biggr\rfloor.
\eeq
Hence, we see
\beq
z_0 = q^{j-i} \kappa^{\lambda_{j+1}^\vee - \mu_i^\vee + \ell_0}
= q^{j-i} \kappa^{k + n\floor{\frac{\lambda_{j+1}^\vee - \mu_i^\vee}{n}}- n \floor{\frac{k + r_1 - r_2}{n}}}.
\eeq
Now, let us count the number of factors in the first shifted factorial.
Since the initial value of $\ell$ is given by \eqref{initial} and the upper bound is $\lambda_{j}^{\vee} - \lambda_{j+1}^\vee -1$,
we see
\begin{align}
m &{}= \biggl\lfloor\frac{\lambda_{j}^{\vee} - \lambda_{j+1}^\vee -1 -\ell_0}{n}\biggr\rfloor +1 \nonumber \\
&{}= (q_3 - q_2) + 1 + \biggl\lfloor\frac{r_3 - r_1 -k -1}{n}\biggr\rfloor + \biggl\lfloor\frac{k + r_1 - r_2}{n}\biggr\rfloor \nonumber \\
&{}= \biggl\lfloor\frac{\lambda_j^\vee - r_1 -k -1}{n}\biggr\rfloor + \biggl\lfloor\frac{k + r_1 - \lambda_{j+1}^\vee}{n}\biggr\rfloor +1
\end{align}
Since
\beq
\biggl\lfloor\frac{\mu_i^\vee}{n}\biggr\rfloor = q_1,
\eeq
and
\beq
\biggl\lfloor\frac{\lambda_{j+1}^\vee + n -1 -k -\bigl(\mu_i^\vee\bigr)}{n}\biggr\rfloor
=
\begin{cases}
q_2 +1, & 0 \leq r_2 - r_1 -1 -k,
\\
q_2, & -n \leq r_2 - r_1 -1 -k < 0,
\\
q_2 -1, & r_2 - r_1 -1 -k < -n,
\end{cases}
\eeq
the power of $\kappa$ is
\beq
\begin{cases}
\lambda_{j+1}^\vee - \mu_{i}^\vee + k + r_1 - r_2 +n , & 1 \leq r_2 - r_1 -k,
\\
\lambda_{j+1}^\vee - \mu_{i}^\vee + k + r_1 - r_2, & 1-n \leq r_2 - r_1 -k < 1,
\\
\lambda_{j+1}^\vee - \mu_{i}^\vee + k + r_1 - r_2 -n , & r_2 - r_1 -k < 1-n.
\end{cases}
\eeq
Comparing this with \eqref{initial}, we can confirm an agreement.
For the number of factors, what we have to check is
\beq \label{final}
\biggl\lfloor\frac{k + r_1 - \lambda_{j+1}^\vee}{n}\biggr\rfloor +1 = - \biggl\lfloor\frac{\lambda_{j+1}^\vee -1 -k - \bigl(\mu_i^\vee\bigr)}{n}\biggr\rfloor.
\eeq
But this is a special case of the first formula of \eqref{floor-formula}. We have checked a complete agreement of the first factorial.

Let us proceed to the second factorial. Concerning the initial condition, we can see the formula for the second is obtained from the first
simply by the replacement,
\beq
\bigl(\mu_i^\vee, \lambda_{j+1}^\vee\bigr) \longrightarrow \bigl(\mu_j^\vee, \lambda_{i}^\vee\bigr).
\eeq
Thus an agreement is proved by appropriate changes of variables. However, counting the number of factors is more involved.
Let us assume
\beq
\lambda_i^\vee = q_4 n + r_4, \qquad \mu_{j}^\vee = q_5 n + r_5, \qquad \mu_{j+1}^\vee = q_6 n + r_6, \qquad 0 \leq r_i\leq n-1.
\eeq
Then a similar computation as before gives the number of factors in the formula \eqref{awata}
\beq
m' = \biggl\lfloor\frac{\mu_{j}^{\vee} - \mu_{j+1}^\vee -1 -\ell_0'}{n}\biggr\rfloor +1,
\eeq
where
\beq
\ell_0' = k + r_5 - r_4 - n \biggl\lfloor\frac{k + r_5 - r_4}{n}\biggr\rfloor.
\eeq
Hence, \begin{align}
m' & = (q_5 - q_6) + 1 + \biggl\lfloor\frac{r_4 - r_6 -k -1}{n}\biggr\rfloor + \biggl\lfloor\frac{k + r_5 - r_4}{n}\biggr\rfloor \nonumber \\
& = \biggr\lfloor\frac{r_4 - \mu_{j+1}^\vee -k -1}{n}\biggr\rfloor + \biggl\lfloor\frac{\mu_{j}^\vee + k - r_4 }{n}\biggr\rfloor +1.
\end{align}
We note $\bigl(-\lambda_i^\vee\bigr) = n-r_4$ for $r_4 \neq 0$, but $\bigl(-\lambda_i^\vee\bigr) = -r_4 = 0$ for $r_4 \neq 0$.
But such a difference does not matter, since we take a difference of the floor function.
Hence, we can safely use $(-\lambda_i^\vee) = -r_4$ and what we have to check is
\beq
\biggl\lfloor\frac{r_4 - \mu_{j+1}^\vee -k -1}{n}\biggr\rfloor + 1 = - \biggl\lfloor\frac{\mu_{j+1}^\vee + k + (-\lambda_i^\vee)}{n}\biggr\rfloor,
\eeq
which follows from the first formula of \eqref{floor-formula}.
\end{proof}

We can also write the formula of Proposition \ref{Shiraishi}
in terms of the ratio of the regularized infinite products $[u;q]_\infty$ defined by \eqref{sinh-reg}.
Using the formulas \eqref{sinh-ratio} and \eqref{floor-formula}, we find
\begin{align}
\mathsf{N}_{\lambda, \mu}^{(k\vert n)} (u \vert q, \kappa)
&{}=
\frac{\displaystyle{\prod_{j > i \geq 1}} \Bigl[uq^{j-i-1}\kappa^{k -n \floor{\frac{\mu_i^\vee}{n}}
+n\floor{\frac{\lambda_{j}^\vee}{n}} -n\floor{\frac{-(\lambda_{j}^\vee) +k + (\mu_i^\vee )}{n}}}; \kappa^n\Bigr]_\infty}
{\displaystyle{\prod_{j \geq i \geq 1}} \Bigl[uq^{j-i}\kappa^{k -n \floor{\frac{\mu_i^\vee}{n}}
+n\floor{\frac{\lambda_{j}^\vee}{n}} - n\floor{\frac{-(\lambda_{j}^\vee) +k + (\mu_i^\vee )}{n}}}; \kappa^n\Bigr]_\infty}
\nonumber \\
&\quad{}\times\frac{\displaystyle{\prod_{i \geq j \geq 1}} \Bigl[uq^{j-i-1}\kappa^{k + n \floor{\frac{\lambda_j^\vee}{n}}
-n\floor{\frac{\mu_{i}^\vee}{n}}
-n\floor{\frac{(\mu_{i}^\vee) + k -(\lambda_j^\vee)}{n}}} ; \kappa^n\Bigr]_\infty}
{\displaystyle{\prod_{i > j \geq 1}}\Bigl[uq^{j-i}\kappa^{k + n \floor{\frac{\lambda_j^\vee}{n}}
-n\floor{\frac{\mu_{i}^\vee}{n}}
-n\floor{\frac{(\mu_{i}^\vee) + k -(\lambda_j^\vee)}{n}}} ; \kappa^n\Bigr]_\infty} \nonumber \\
&{}= \prod_{i,j=1}^\infty
\frac{\Bigl[uq^{j-i-1}\kappa^{k -n\floor{\frac{\mu_{i}^\vee + k - \lambda_j^\vee}{n}}} ; \kappa^n\Bigr]_\infty}
{\Bigl[uq^{j-i}\kappa^{k -n\floor{\frac{\mu_{i}^\vee + k -\lambda_j^\vee}{n}}} ; \kappa^n\Bigr]_\infty}
\frac{\bigl[uq^{j-i}\kappa^{k}; \kappa^n\bigr]_\infty}{\bigl[uq^{j-i-1}\kappa^{k}; \kappa^n\bigr]_\infty}.\label{kappa-form}
\end{align}

For later convenience, let us express the orbifolded Nekrasov factor in terms of $t=\kappa^{-n}$-shifted
factorials. We have
\begin{align}
&{}\mathsf{N}_{\lambda, \mu}^{(k\vert n)} (u \vert q, t)\nonumber\\
&\qquad{}=
\prod_{j \geq i \geq 1} \Bigl[uq^{j-i} t^{-\frac{k}{n} +\floor{\frac{\mu_i^\vee}{n}}
-\floor{\frac{\lambda_{j+1}^\vee + n-1 -k - (\mu_i^\vee )}{n}}} ; t^{-1}\Bigr]
_{\floor{\frac{\lambda_{j}^\vee + n-1 -k - (\mu_i^\vee )}{n}}- \floor{\frac{\lambda_{j+1}^\vee + n-1 -k - (\mu_i^\vee )}{n}}} \nonumber \\
&\qquad\quad{}\times \prod_{j \geq i \geq 1} \Bigl[uq^{i-j-1} t^{-\frac{k}{n} - \floor{\frac{\lambda_i^\vee}{n}}
+ \floor{\frac{\mu_j^\vee + k -(\lambda_i^\vee)}{n}}} ; t^{-1}\Bigr]
_{\floor{\frac{\mu_j^\vee + k -(\lambda_i^\vee)}{n}}-\floor{\frac{\mu_{j+1}^\vee + k -(\lambda_i^\vee)}{n}}} \nonumber \\
&\qquad{}=
\prod_{j \geq i \geq 1} \Bigl[uq^{j-i} t^{1 -\frac{k}{n} +\floor{\frac{\mu_i^\vee}{n}}
-\floor{\frac{\lambda_{j}^\vee + n-1 -k - (\mu_i^\vee )}{n}}} ; t\Bigr]
_{\floor{\frac{\lambda_{j}^\vee + n-1 -k - (\mu_i^\vee )}{n}}- \floor{\frac{\lambda_{j+1}^\vee + n-1 -k - (\mu_i^\vee )}{n}}} \nonumber \\
&\qquad\quad{}\times \prod_{j \geq i \geq 1} \Bigl[uq^{i-j-1} t^{1 -\frac{k}{n} - \floor{\frac{\lambda_i^\vee}{n}}
+ \floor{\frac{\mu_{j+1}^\vee + k -(\lambda_i^\vee)}{n}}} ; t\Bigr]
_{\floor{\frac{\mu_j^\vee + k -(\lambda_i^\vee)}{n}}-\floor{\frac{\mu_{j+1}^\vee + k -(\lambda_i^\vee)}{n}}} \nonumber \\
&\qquad{}=
\prod_{j \geq i \geq 1} \frac{\Bigl[uq^{j-i} t^{1 -\frac{k}{n} +\floor{\frac{\mu_i^\vee}{n}}
-\floor{\frac{\lambda_{j}^\vee + n-1 -k - (\mu_i^\vee )}{n}}} ; t\Bigr]_\infty}
{\Bigl[uq^{j-i} t^{1 -\frac{k}{n} +\floor{\frac{\mu_i^\vee}{n}}
-\floor{\frac{\lambda_{j+1}^\vee + n-1 -k - (\mu_i^\vee )}{n}}} ; t\Bigr]_\infty} \nonumber \\
&\qquad\quad{}\times \prod_{j \geq i \geq 1} \frac{\Bigl[uq^{i-j-1} t^{1 -\frac{k}{n} - \floor{\frac{\lambda_i^\vee}{n}}
+ \floor{\frac{\mu_{j+1}^\vee + k -(\lambda_i^\vee)}{n}}} ; t\Bigr]_\infty}
{\Bigl[uq^{i-j-1} t^{1 -\frac{k}{n} - \floor{\frac{\lambda_i^\vee}{n}}
+ \floor{\frac{\mu_{j}^\vee + k -(\lambda_i^\vee)}{n}}} ; t\Bigr]_\infty},
\end{align}
where we have used
\beq\label{q-shifted-inversion}
\bigl[u;t^{-1}\bigr]_n = \bigl[u t^{-n+1} ; t\bigr]_n, \qquad n>0.
\eeq
Finally, the same manipulation as the last equality of \eqref{kappa-form} gives
\begin{gather}\label{t-formula}
\mathsf{N}_{\lambda, \mu}^{(k\vert n)} (u \vert q,t)
=
\prod_{i,j=1}^\infty
\frac{\Bigl[uq^{j-i} t^{1 -\frac{k}{n} + \floor{\frac{\mu_{i}^\vee + k -\lambda_j^\vee}{n}}} ; t\Bigr]_\infty}
{\Bigl[uq^{j-i-1} t^{1 -\frac{k}{n} + \floor{\frac{\mu_{i}^\vee + k -\lambda_j^\vee}{n}}} ; t\Bigr]_\infty}
\frac{\bigl[uq^{j-i-1} t^{1 -\frac{k}{n}}; t]_\infty} {[uq^{j-i} t^{1 -\frac{k}{n}}; t\bigr]_\infty}.
\end{gather}

When one of the partitions is empty, $\mathsf{N}_{\lambda, \mu}^{(k\vert n)}$ simplifies as follows:
\begin{gather}
\mathsf{N}_{\lambda, \varnothing}^{(k\vert n)} (u \vert q, \kappa)
=
\frac{\displaystyle{\prod_{j \geq i \geq 1}} \bigl[uq^{j-i}\kappa^{k}; \kappa^n\bigr]
_{\floor{\frac{\lambda_{j}^\vee + n-1 -k}{n}}}}
{\displaystyle{\prod_{j \geq i \geq 1}} \bigl[uq^{(j+1)-(i+1)}\kappa^{k}; \kappa^n\bigr]
_{\floor{\frac{\lambda_{j+1}^\vee + n-1 -k}{n}}}} \nonumber\\ \hphantom{\mathsf{N}_{\lambda, \varnothing}^{(k\vert n)} (u \vert q, \kappa)}{}
= \prod_{i \geq 1} \bigl[uq^{i-1}\kappa^{k}; \kappa^n\bigr]
_{\floor{\frac{\lambda_{i}^\vee + n-1 -k}{n}}},
\\
\mathsf{N}_{\varnothing, \mu}^{(k\vert n)} (u \vert q, \kappa)
=
\frac{\displaystyle{\prod_{j \geq i \geq 1}} \Bigl[uq^{i-j -1}\kappa^{k -n\floor{\frac{\mu_j^\vee + k }{n}}} ; \kappa^n\Bigr]
_{\floor{\frac{\mu_j^\vee + k }{n}}}}
{\displaystyle{\prod_{j \geq i \geq 1}} \Bigl[uq^{i-(j+1)}\kappa^{k -n\floor{\frac{\mu_{j+1}^\vee + k}{n}}} ; \kappa^n\Bigr]
_{\floor{\frac{\mu_{j+1}^\vee + k}{n}}}} \nonumber\\ \hphantom{\mathsf{N}_{\varnothing, \mu}^{(k\vert n)} (u \vert q, \kappa)}{}
= \prod_{i \geq 1}\Bigl[uq^{-i}\kappa^{k -n\floor{\frac{\mu_i^\vee + k }{n}}} ; \kappa^n\Bigr]
_{\floor{\frac{\mu_i^\vee + k }{n}}}.
\end{gather}


\section{(Anti-)symmetrization in a factorized form}
\label{App:symmetrization}

A basis of the space of cohomology classes for the $N$-tuple Jackson integral is defined by using
the anti-symmetrization of the variables $z=(z_1, \dots, z_N)$ with the weight function $\Delta(q,z)$ (see~\eqref{Delta-q}).
In general, the result of (anti-)symmetrization is expanded in terms of symmetric polynomials (divided by the Vandermonde determinant
for anti-symmetrization). However, from the viewpoint of the correspondence of the Jackson integral and the Bethe vector for quantum integrable system
\cite{Reshe-Bethe}, what we need is the (anti-)symmetrization that keeps the factorization.
One can rephrase it as a summation over partitions of the index set $[N] := \{ 1,2, \dots, N \}$.
We give a proof for the following well-known formulas for completeness.

\begin{Proposition}\label{Shuffle}
Let $I_1 \sqcup \cdots \sqcup I_s$ be a partition of $\{1,2,\dots, N\}$ with $|I_a|=k_a \geq 1$
and ${\mathcal A}$ be the anti-symmetrization operator of $z=(z_1, \dots, z_N)$.
For any functions $f_a(x)$, $1 \leq a \leq s$, we have
\begin{gather}
\dfrac{1}{\Delta(1,z)}{\mathcal A}\Biggl(\prod_{i_1=1}^{k_1}f_1(z_{i_1}) \prod_{i_2=1}^{k_2}f_2(z_{k_1+i_2})
\cdots
\prod_{i_s=1}^{k_s}f_s(z_{k_1+\cdots+k_{s-1}+i_s}) \Delta(q,z) \Biggr)\nonumber\\
\qquad{}=\prod_{a=1}^s [k_a]_{q^{-1}}! \sum_{I_1 \sqcup \cdots \sqcup I_s}\Biggl\{
\prod_{a=1}^s \prod_{i_a \in I_a}f_a(z_{i_a})
\prod_{1\leq a<b\leq s}\prod_{i \in I_a}\prod_{j \in I_b} \dfrac{z_i-q^{-1}z_j}{z_i-z_j} \Biggr\},\label{eq:Asym}
\end{gather}
where the sum $\sum_{I_1 \sqcup \cdots \sqcup I_s}$ on the right-hand side
is taken over the partitions $I_1 \sqcup \cdots \sqcup I_s$.
\end{Proposition}

\begin{proof}
For a permutation $\sigma \in S_N$, we define $I_a=\{\sigma(k_1+\cdots+k_{a-1}+i) \mid 1\leq i\leq k_a\}$.
Since
\begin{equation*}
\dfrac{1}{\Delta(1,z)}{\mathcal A}(F(z) \Delta(q,z) )
=\sum_{\sigma \in S_N} \sigma\biggl(F(z) \dfrac{\Delta(q,z)}{\Delta(1,z)} \biggr),
\end{equation*}
for any $F(z)=F(z_1, \dots, z_N)$, we can recast the left-hand side of \eqref{eq:Asym} to
\begin{equation*}
\sum_{\sigma \in S_N}\Biggl(\prod_{a=1}^s \prod_{i \in I_a}f_a(z_i)\dfrac{\Delta(q, z_{I_a})}{\Delta(z_{I_a})}
\prod_{1\leq a<b\leq s}\prod_{i \in I_a}\prod_{j \in I_b} \dfrac{z_i-q^{-1}z_j}{z_i-z_j}\Biggr).
\end{equation*}
Let $S_I$ be the symmetric group of a set $I$.
Then according to the coset decomposition $S_{N}=
\mathop{\sqcup }_{I_1 \sqcup \cdots \sqcup I_s} S_{I_1} \times \cdots \times S_{I_s}$,
we take the sum in the following manner:
\[
\sum_{\sigma \in S_N}=\sum_{I_1 \sqcup \cdots \sqcup I_s}
\sum_{w_1 \in S_{I_1}}\cdots \sum_{w_s \in S_{I_s}}.
\]
By applying the formula \cite[Section~III, equation~(1.4)]{MacD}
\[
\sum_{w_a \in S_{I_a}} w_a\biggl(\dfrac{\Delta(q,z_{I_a})}{\Delta(1,z_{I_a})}\biggr)=[k_a]_{q^{-1}}!
\]
for each sum over $S_{I_a}$, we obtain the desired result \eqref{eq:Asym}.
\end{proof}

\section{Shakirov's equation as a coupled system}
\label{App:coupled}

Consider Shakirov's equation in the form
\begin{align}
&\psi=\SS T_{t,\Lambda}^{-1}T_{qtQ,x}^{-1} \psi, \label{eq:pHp}\\
&\SS = \dfrac{1}{\varphi(qx)\varphi(\frac{\Lambda}{x})} \Bor
\dfrac{\varphi(\Lambda)\varphi\bigl(q^{-1} d_1d_2d_3d_4\Lambda\bigr)}{\varphi(-d_1x)\varphi(-d_2x)\varphi\bigl(-d_3 \frac{\Lambda}{x}\bigr)\varphi\bigl(-d_4 \frac{\Lambda}{x}\bigr)}
\Bor \dfrac{1}{\varphi\bigl(\frac{d_1d_2}{q}x\bigr)\varphi\bigl(d_3d_4 \frac{\Lambda}{x}\bigr)}.\nonumber
\end{align}
It has a unique formal series solution normalized as
$\psi=1+{\mathcal O}\bigl(x, \frac{\Lambda}{x}\bigr)$.
In terms of the affine Laumon partition function $\mathcal{Z}_{\mathrm{AL}}$, the solution is given by
\begin{align}\label{eq:Fuvw}
\psi(\Lambda,x) ={}&\mathcal{Z}_{\mathrm{AL}}
\left( \left.\left.\begin{array}{c}u_1,u_2 \\v_1,v_2\\w_1,w_2\end{array} \right| x_1, \frac{\Lambda_1}{x_1} \right|q,t\right) \nonumber \\
={}&\mathcal{Z}_{\mathrm{AL}}
\left( \left.\left.\begin{array}{c} \frac{qQ}{d_3},\frac{q\kappa}{d_1} \\ 1,\frac{Q}{\kappa} \\ \frac{1}{d_2},\frac{Q}{\kappa d_4} \end{array} \right|
-\frac{\sqrt{Qd_1d_2}}{\kappa}x, - \sqrt{\frac{d_3d_4}{q^2Q}}\frac{\Lambda}{x}\right| q,\kappa^{-2} \right).
\end{align}
%

We will give a coupled form of Shakirov's equation and its solution.
To do this, we define operators $T$ and $K$ by\footnote{One may use $d_1$ (or $d_3$) instead of $d_2$ (or $d_4$).}
\begin{align}\label{eq:TKdef}
T={}&\biggl\{d_2 \mapsto \dfrac{q}{t Q d_2}, \,
d_4 \mapsto \dfrac{qQ}{d_4}, \,
\Lambda \mapsto \dfrac{d_2 d_4 \Lambda}{q},\,
x \mapsto -\dfrac{d_2 x}{q} \biggr\},\nonumber \\
K={}&\dfrac{1}{\varphi(qx)\varphi\bigl(\frac{\Lambda}{x}\bigr)} \Bor
\dfrac{1}{\varphi(-d_1x)\varphi\bigl(-d_3 \frac{\Lambda}{x}\bigr)}.
\end{align}
On the parameters $u_i$, $v_i$, $w_i$, $x_1$, $\Lambda_1$, the operator $T$ acts as
\begin{equation}
T= \bigg\{ w_i \mapsto \frac{t Q}{q w_i}, \,
x_1 \mapsto -\frac{x_1}{\sqrt{qt Q}}, \,
\Lambda_1 \mapsto \frac{\Lambda_1}{\sqrt{t}} \bigg\}.
\end{equation}
In our previous paper \cite{Awata:2022idl}, we used the renormalized $q$-Borel transformation
\smash{\raisebox{0.5pt}{$\widetilde{\Bor}:= T^{-1}_{(qt^{1/2}Q)^{1/2}, x} \Bor$}}
in writing down the coupled system in order to make the correspondence with the $qq$-Pain\-lev\'e~VI equation clear.
If this is taken into account, the action of $T$ on the expansion parameters becomes $x_1=x \mapsto - t^{-1/4} x_1$
and hence $x_2 = \Lambda_1/x_1 \mapsto - t^{-1/4} x_2$. The $T$ action is the same for~$x_1$ and~$x_2$ and
coincides with the action of the shift operator
$\widetilde{T}_{\mathsf{p}, b}$ with $\mathsf{p}=t^{1/4}$, which is regarded as a square root of the time evolution
of the parameters $b_i$ in the $qq$-Painlev\'e VI equation (see~\mbox{\cite[equation~(6.3)]{Awata:2022idl}}).

\begin{Proposition}
The normalized solution $\psi$ of \eqref{eq:pHp} and its parameter transform $\chi:=T \psi$ satisfy the following coupled system of equations:
\begin{equation}\label{coupled:new}
\psi=g K \chi, \qquad
\chi=T g K T \psi,
\end{equation}
where
\begin{equation}
g=\dfrac{\bigl(\frac{td_2d_4}{q} \Lambda,d_1d_3 \Lambda;q,t\bigr)_{\infty}}{\bigl(t\Lambda,\frac{td_1d_2d_3d_4}{q}\Lambda;q,t\bigr)_{\infty}}.
\end{equation}
Conversely the equation \eqref{eq:pHp} follows from the coupled system \eqref{coupled:new}.
\end{Proposition}

\begin{proof}
From \eqref{eq:pHp} and \eqref{eq:TKdef},
one can verify the following relations:
\begin{align}
&T^2=T_{t,\Lambda}^{-1} T_{q Qt,x}^{-1}, \label{eq:T2TT}\\
&TK =
\dfrac{1}{\varphi(-d_2 x)\varphi\bigl(-d_4\frac{\Lambda}{x}\bigr)} \Bor
\dfrac{1}{\varphi\bigl(\frac{d_1d_2}{q}x\bigr)\varphi\bigl(d_3 d_4 \frac{\Lambda}{x}\bigr)}T,\\
& KTKT =
K\dfrac{1}{\varphi(-d_2 x)\varphi\bigl(-d_4\frac{\Lambda}{x}\bigr)} \Bor
\dfrac{1}{\varphi\bigl(\frac{d_1d_2}{q}x\bigr)\varphi\bigl(d_3 d_4 \frac{\Lambda}{x}\bigr)}T^2
=\dfrac{1}{\varphi(\Lambda)\varphi\bigl(\frac{d_1d_2d_3d_4}{q}\Lambda\bigr)}\SS T^2, \\
& K T \SS T^2
= \varphi\biggl(\frac{d_2 d_4}{q} \Lambda\biggr) \varphi\biggl(\frac{d_1 d_3}{t} \Lambda\biggr) (KT)^3
=\dfrac{\varphi\bigl(\frac{d_2d_4}{q}\Lambda\bigr)\varphi\bigl(\frac{d_1d_3}{t}\Lambda\bigr)}{\varphi(\Lambda)\varphi\bigl(\frac{d_1d_2d_3d_4}{q}\Lambda\bigr)}\SS T^2 KT.
\label{eq:KTH}
\end{align}
From \eqref{eq:T2TT}, \eqref{eq:KTH} and \eqref{eq:pHp}, we have
\begin{align}
KT \psi=K T \SS T^2 \psi
=\dfrac{\varphi\bigl(\frac{d_2d_4}{q}\Lambda\bigr)\varphi\bigl(\frac{d_1d_3}{t}\Lambda\bigr)}{\varphi(\Lambda)\varphi\bigl(\frac{d_1d_2d_3d_4}{q}\Lambda\bigr)}\SS T^2 K T \psi.
\end{align}
Noting that
\[
T_{t,\Lambda}^{-1}(g)=\frac{\varphi\bigl(\frac{d_2d_4}{q}\Lambda\bigr)\varphi\bigl(\frac{d_1d_3}{t}\Lambda\bigr)}{\varphi(\Lambda)\varphi\bigl(\frac{d_1d_2d_3d_4}{q}\Lambda\bigr)}g,
\]
we see the function $g KT \psi$ satisfies the same equation as \eqref{eq:pHp} with the same initial condition as $\psi$.
By the uniqueness of the solution, we obtain the first relation $\psi=g KT \psi=g K \chi$.
The second relation is the $T$-transform of the first one.
To check the converse, we note that
\beq
g \cdot T\cdot g = \frac{\bigl(\Lambda, \frac{d_1 d_2 d_3 d_4}{q} \Lambda ; q,t\bigr)_\infty}
{\bigl(t \Lambda, \frac{t d_1 d_2 d_3 d_4}{q} \Lambda ; q,t\bigr)_\infty} \cdot T
= \varphi(\Lambda) \varphi\biggl(\frac{d_1 d_2 d_3 d_4}{q}\Lambda\biggr) \cdot T.
\eeq
Hence,
\beq
\psi = g \cdot KT \cdot g \cdot KT \psi = \varphi(\Lambda) \varphi\biggl(\frac{d_1 d_2 d_3 d_4}{q}\Lambda\biggr) \cdot KTKT \psi
= \SS T^2 \psi.
\eeq
This completes the proof.
\end{proof}

\subsection*{Acknowledgements}
We would like to thank S.~Arthamonov, M.~Bershtein, P.~Gavrylenko, M.~Ito,
M.~Noumi, M.~Schlosser and G.~Shibukawa for useful discussions.
Our work is supported in part by Grants-in-Aid for Scientific Research (Kakenhi):
18K03274 (H.K.), 23K03087 (H.K.), 21K03180 (R.O.), 19K03512 (J.S.), 19K03530 (J.S.) and 22H01116 (Y.Y.).
The work of R.O. was partly supported by Osaka Central Advanced Mathematical
Institute: MEXT Joint Usage/Research Center on Mathematics and
Theoretical Physics JPMXP0619217849, and the Research Institute for Mathematical Sciences,
an International Joint Usage/Research Center located in Kyoto University.


\begin{thebibliography}{99}
\footnotesize\itemsep=0pt

\bibitem{AOF}
Aganagic M., Frenkel E., Okounkov A., Quantum {$q$}-{L}anglands correspondence,
 \href{https://doi.org/10.1090/mosc/278}{\textit{Trans. Moscow Math. Soc.}} \textbf{79} (2018), 1--83,
 \href{https://arxiv.org/abs/1701.03146}{arXiv:1701.03146}.

\bibitem{aganagic2011knot}
Aganagic M., Shakirov {\relax{Sh}}., Knot homology and refined
 {C}hern--{S}imons index, \href{https://doi.org/10.1007/s00220-014-2197-4}{\textit{Comm. Math. Phys.}} \textbf{333} (2015),
 187--228, \href{https://arxiv.org/abs/1105.5117}{arXiv:1105.5117}.

\bibitem{AAB}
Agarwal A.K., Andrews G.E., Bressoud D.M., The {B}ailey lattice,
 \textit{J.~Indian Math. Soc. (N.S.)} \textbf{51} (1987), 57--73.

\bibitem{Alday:2009fs}
Alday L.F., Gaiotto D., Gukov S., Tachikawa Y., Verlinde H., Loop and surface
 operators in {${\mathcal N}=2$} gauge theory and {L}iouville modular
 geometry, \href{https://doi.org/10.1007/JHEP01(2010)113}{\textit{J.~High Energy Phys.}} \textbf{2010} (2010), no.~1, 113,
 50~pages, \href{https://arxiv.org/abs/0909.0945}{arXiv:0909.0945}.

\bibitem{Alday:2009aq}
Alday L.F., Gaiotto D., Tachikawa Y., Liouville correlation functions from
 four-dimensional gauge theories, \href{https://doi.org/10.1007/s11005-010-0369-5}{\textit{Lett. Math. Phys.}} \textbf{91}
 (2010), 167--197, \href{https://arxiv.org/abs/0906.3219}{arXiv:0906.3219}.

\bibitem{Alday:2010vg}
Alday L.F., Tachikawa Y., Affine~{${\rm SL}(2)$} conformal blocks from 4d gauge
 theories, \href{https://doi.org/10.1007/s11005-010-0422-4}{\textit{Lett. Math. Phys.}} \textbf{94} (2010), 87--114,
 \href{https://arxiv.org/abs/1005.4469}{arXiv:1005.4469}.

\bibitem{A}
Andrews G.E., Connection coefficient problems and partitions, in Relations
 Between Combinatorics and Other Parts of Mathematics ({P}roc. {S}ympos.
 {P}ure {M}ath., {O}hio {S}tate {U}niv., {C}olumbus, {O}hio, 1978),
 \textit{Proc. Sympos. Pure Math.}, Vol.~34, American Mathematical Society,
 Providence, RI, 1979, 1--24.

\bibitem{AK}
Aomoto K., Kato Y., Gauss decomposition of connection matrices for symmetric
 {$A$}-type {J}ackson integrals, \href{https://doi.org/10.1007/BF01587906}{\textit{Selecta Math.~(N.S.)}} \textbf{1}
 (1995), 623--666.

\bibitem{AFKMY}
Awata H., Fuji H., Kanno H., Manabe M., Yamada Y., Localization with a~surface
 operator, irregular conformal blocks and open topological string,
 \href{https://doi.org/10.4310/ATMP.2012.v16.n3.a1}{\textit{Adv. Theor. Math. Phys.}} \textbf{16} (2012), 725--804,
 \href{https://arxiv.org/abs/1008.0574}{arXiv:1008.0574}.

\bibitem{Awata:2022idl}
Awata H., Hasegawa K., Kanno H., Ohkawa R., Shakirov S., Shiraishi J., Yamada
 Y., Non-stationary difference equation and affine {L}aumon space:
 quantization of discrete {P}ainlev\'e equation, \href{https://doi.org/10.3842/SIGMA.2023.089}{\textit{SIGMA}} \textbf{19}
 (2023), 089, 47~pages, \href{https://arxiv.org/abs/2211.16772}{arXiv:2211.16772}.

\bibitem{Awata:2008ed}
Awata H., Kanno H., Refined {BPS} state counting from {N}ekrasov's formula and
 {M}acdonald functions, \href{https://doi.org/10.1142/S0217751X09043006}{\textit{Internat. J.~Modern Phys.~A}} \textbf{24}
 (2009), 2253--2306, \href{https://arxiv.org/abs/0805.0191}{arXiv:0805.0191}.

\bibitem{Awata:2019isq}
Awata H., Kanno H., Mironov A., Morozov A., On a~complete solution of the
 quantum {D}ell system, \href{https://doi.org/10.1007/jhep04(2020)212}{\textit{J.~High Energy Phys.}} \textbf{2020} (2020),
 no.~4, 212, 30~pages, \href{https://arxiv.org/abs/1912.12897}{arXiv:1912.12897}.

\bibitem{Awata:2009ur}
Awata H., Yamada Y., Five-dimensional {AGT} conjecture and the deformed
 {V}irasoro algebra, \href{https://doi.org/10.1007/JHEP01(2010)125}{\textit{J.~High Energy Phys.}} \textbf{2010} (2010),
 no.~1, 125, 11~pages, \href{https://arxiv.org/abs/0910.4431}{arXiv:0910.4431}.

\bibitem{Bosnjak:2016oze}
Bosnjak G., Mangazeev V.V., Construction of {$R$}-matrices for symmetric tensor
 representations related to {$U_q(\widehat{sl_n})$}, \href{https://doi.org/10.1088/1751-8113/49/49/495204}{\textit{J.~Phys.~A}}
 \textbf{49} (2016), 495204, 19~pages, \href{https://arxiv.org/abs/1607.07968}{arXiv:1607.07968}.

\bibitem{Braverman:2004vv}
Braverman A., Instanton counting via affine {L}ie algebras.~{I}. {E}quivariant
 {$J$}-functions of (affine) flag manifolds and {W}hittaker vectors, in
 Algebraic Structures and Moduli Spaces, \textit{CRM Proc. Lecture Notes},
 Vol.~38, \href{https://doi.org/10.1090/crmp/038/04}{American Mathematical Society}, Providence, RI, 2004, 113--132,
 \href{https://arxiv.org/abs/math.AG/0401409}{arXiv:math.AG/0401409}.

\bibitem{Braverman:2004cr}
Braverman A., Etingof P., Instanton counting via affine {L}ie algebras.~{II}.
 {F}rom {W}hittaker vectors to the {S}eiberg--{W}itten prepotential, in
 Studies in {L}ie Theory, \textit{Progr. Math.}, Vol. 243, \href{https://doi.org/10.1007/0-8176-4478-4_5}{Birkh\"auser
 Boston}, Boston, MA, 2006, 61--78, \href{https://arxiv.org/abs/math.AG/0409441}{arXiv:math.AG/0409441}.

\bibitem{BFS}
Braverman A., Finkelberg M., Shiraishi J., Macdonald polynomials, {L}aumon
 spaces and perverse coherent sheaves, in Perspectives in Representation
 Theory, \textit{Contemp. Math.}, Vol. 610, \href{https://doi.org/10.1090/conm/610/12130}{American Mathematical Society},
 Providence, RI, 2014, 23--41, \href{https://arxiv.org/abs/1206.3131}{arXiv:1206.3131}.

\bibitem{B}
Bressoud D.M., A~matrix inverse, \href{https://doi.org/10.2307/2044991}{\textit{Proc. Amer. Math. Soc.}} \textbf{88}
 (1983), 446--448.

\bibitem{Bullimore:2014awa}
Bullimore M., Kim H.-C., Koroteev P., Defects and quantum {S}eiberg--{W}itten
 geometry, \href{https://doi.org/10.1007/JHEP05(2015)095}{\textit{J.~High Energy Phys.}} (2015), no.~5, 095, 78~pages,
 \href{https://arxiv.org/abs/1412.6081}{arXiv:1412.6081}.

\bibitem{CV}
Cotti G., Varchenko A., Equivariant quantum differential equation and q{KZ}
 equations for a~projective space: {S}tokes bases as exceptional collections,
 {S}tokes matrices as {G}ram matrices, and B-theorem, in Integrability,
 Quantization, and Geometry.~{I}. {I}ntegrable Systems, \textit{Proc. Sympos.
 Pure Math.}, Vol. 103, \href{https://doi.org/10.1090/pspum/103.1/01833}{American Mathematical Society}, Providence, RI, 2021,
 101--170, \href{https://arxiv.org/abs/1909.06582}{arXiv:1909.06582}.

\bibitem{DFK1}
Di~Francesco P., Kedem R., Macdonald duality and the proof of the quantum
 {Q}-system conjecture, \href{https://doi.org/10.1007/s00029-023-00909-z}{\textit{Selecta Math.~(N.S.)}} \textbf{30} (2024), 23,
 100~pages, \href{https://arxiv.org/abs/2112.09798}{arXiv:2112.09798}.

\bibitem{DFK2}
Di~Francesco P., Kedem R., Duality and {M}acdonald difference operators,
 \href{https://arxiv.org/abs/2303.04276}{arXiv:2303.04276}.

\bibitem{EFK}
Etingof P.I., Frenkel I.B., Kirillov Jr. A.A., Lectures on representation
 theory and {K}nizhnik--{Z}amolodchikov equations, \textit{Math. Surveys
 Monogr.}, Vol.~58, \href{https://doi.org/10.1090/surv/058}{American Mathematical Society}, Providence, RI, 1998.

\bibitem{FZ}
Fateev V.A., Zamolodchikov A.B., Operator algebra and correlation functions in
 the two-dimensional ${\rm SU}(2) \times {\rm SU}(2)$ chiral {W}ess--{Z}umino
 model, \textit{Sov. J.~Nuclear Phys.} \textbf{43} (1986), 657--664.

\bibitem{FFNR}
Feigin B., Finkelberg M., Negut A., Rybnikov L., Yangians and cohomology rings
 of {L}aumon spaces, \href{https://doi.org/10.1007/s00029-011-0059-x}{\textit{Selecta Math.~(N.S.)}} \textbf{17} (2011),
 573--607, \href{https://arxiv.org/abs/0812.4656}{arXiv:0812.4656}.

\bibitem{FR}
Finkelberg M., Rybnikov L., Quantization of {D}rinfeld {Z}astava in type~{A},
 \href{https://arxiv.org/abs/1009.0676}{arXiv:1009.0676}.

\bibitem{FR-q-KZ}
Frenkel I.B., Reshetikhin N.{\relax{Yu}}., Quantum affine algebras and
 holonomic difference equations, \href{https://doi.org/10.1007/BF02099206}{\textit{Comm. Math. Phys.}} \textbf{146}
 (1992), 1--60.

\bibitem{hypergeometric}
Gasper G., Rahman M., Basic hypergeometric series, \textit{Encyclopedia Math.
 Appl.}, Vol.~96, 2nd ed., \href{https://doi.org/10.1017/CBO9780511526251}{Cambridge University Press}, Cambridge, 2004.

\bibitem{MIto}
Ito M., {$q$}-difference systems for the {J}ackson integral of symmetric
 {S}elberg type, \href{https://doi.org/10.3842/SIGMA.2020.113}{\textit{SIGMA}} \textbf{16} (2020), 113, 31~pages,
 \href{https://arxiv.org/abs/1910.08393}{arXiv:1910.08393}.

\bibitem{MIto1997}
Ito M., Gauss decomposition and {$q$}-difference equations for {J}ackson
 integrals of symmetric {S}elberg type, \textit{Ryukyu Math.~J.} \textbf{36}
 (2023), 1--47, \href{https://arxiv.org/abs/2309.17181}{arXiv:2309.17181}.

\bibitem{ItoForrester}
Ito M., Forrester P.J., A~bilateral extension of the {$q$}-{S}elberg integral,
 \href{https://doi.org/10.1090/tran/6851}{\textit{Trans. Amer. Math. Soc.}} \textbf{369} (2017), 2843--2878,
 \href{https://arxiv.org/abs/1309.0001}{arXiv:1309.0001}.

\bibitem{ItoNoumi}
Ito M., Noumi M., Connection formula for the {J}ackson integral of type~{$A_n$}
 and elliptic {L}agrange interpolation, \href{https://doi.org/10.3842/SIGMA.2018.077}{\textit{SIGMA}} \textbf{14} (2018),
 077, 42~pages, \href{https://arxiv.org/abs/1801.07041}{arXiv:1801.07041}.

\bibitem{Kanno:2011fw}
Kanno H., Tachikawa Y., Instanton counting with a~surface operator and the
 chain-saw quiver, \href{https://doi.org/10.1007/JHEP06(2011)119}{\textit{J.~High Energy Phys.}} \textbf{2011} (2011), no.~6,
 119, 24~pages, \href{https://arxiv.org/abs/1105.0357}{arXiv:1105.0357}.

\bibitem{langmann2020basic}
Langmann E., Noumi M., Shiraishi J., Basic properties of non-stationary
 {R}uijsenaars functions, \href{https://doi.org/10.3842/SIGMA.2020.105}{\textit{SIGMA}} \textbf{16} (2020), 105, 26~pages,
 \href{https://arxiv.org/abs/2006.07171}{arXiv:2006.07171}.

\bibitem{Laumon1}
Laumon G., Un analogue global du c\^one nilpotent, \href{https://doi.org/10.1215/S0012-7094-88-05729-8}{\textit{Duke Math.~J.}}
 \textbf{57} (1988), 647--671.

\bibitem{Laumon2}
Laumon G., Faisceaux automorphes li\'es aux s\'eries d'{E}isenstein, in
 Automorphic Forms, {S}himura Varieties, and {$L$}-functions, {V}ol.~{I}
 ({A}nn {A}rbor, {MI}, 1988), \textit{Perspect. Math.}, Vol.~10, Academic
 Press, Boston, MA, 1990, 227--281.

\bibitem{MacD}
Macdonald I.G., Symmetric functions and {H}all polynomials, 2nd ed., Oxford
 Math. Monogr., Oxford University Press, New York, 1995.

\bibitem{Mangazeev:2014gwa}
Mangazeev V.V., On the {Y}ang--{B}axter equation for the six-vertex model,
 \href{https://doi.org/10.1016/j.nuclphysb.2014.02.019}{\textit{Nuclear Phys.~B}} \textbf{882} (2014), 70--96, \href{https://arxiv.org/abs/1401.6494}{arXiv:1401.6494}.

\bibitem{AMatsuo}
Matsuo A., Jackson integrals of {J}ordan--{P}ochhammer type and quantum
 {K}nizhnik--{Z}amolodchikov equations, \href{https://doi.org/10.1007/BF02096769}{\textit{Comm. Math. Phys.}}
 \textbf{151} (1993), 263--273.

\bibitem{AMatsuo2}
Matsuo A., Quantum algebra structure of certain {J}ackson integrals,
 \href{https://doi.org/10.1007/BF02096880}{\textit{Comm. Math. Phys.}} \textbf{157} (1993), 479--498.

\bibitem{Mimachi}
Mimachi K., Holonomic {$q$}-difference system of the first order associated
 with a~{J}ackson integral of {S}elberg type, \href{https://doi.org/10.1215/S0012-7094-94-07319-5}{\textit{Duke Math.~J.}}
 \textbf{73} (1994), 453--468.

\bibitem{Nagoya}
Nagoya H., Hypergeometric solutions to {S}chr\"odinger equations for the
 quantum {P}ainlev\'e equations, \href{https://doi.org/10.1063/1.3620412}{\textit{J.~Math. Phys.}} \textbf{52} (2011),
 083509, 16~pages, \href{https://arxiv.org/abs/1109.1645}{arXiv:1109.1645}.

\bibitem{Negut:2011aa}
Negu\c{t} A., Affine Laumon spaces and integrable systems, \href{https://arxiv.org/abs/1112.1756}{arXiv:1112.1756}.

\bibitem{Negut2}
Negu\c{t} A., Affine {L}aumon spaces and a~conjecture of {K}uznetsov,
 \href{https://doi.org/10.24033/asens.2505}{\textit{Ann. Sci. \'Ec. Norm. Sup\'er.}} \textbf{55} (2022), 739--789,
 \href{https://arxiv.org/abs/1811.01011}{arXiv:1811.01011}.

\bibitem{Nekrasov:2017rqy}
Nekrasov N., B{PS}/{CFT} correspondence~{IV}: sigma models and defects in gauge
 theory, \href{https://doi.org/10.1007/s11005-018-1115-7}{\textit{Lett. Math. Phys.}} \textbf{109} (2019), 579--622,
 \href{https://arxiv.org/abs/1711.11011}{arXiv:1711.11011}.

\bibitem{Nekrasov:2017gzb}
Nekrasov N., {BPS}/{CFT} correspondence {V}: {BPZ} and {KZ} equations from
 {$qq$}-characters, \href{https://arxiv.org/abs/1711.11582}{arXiv:1711.11582}.

\bibitem{Nekrasov:2021tik}
Nekrasov N., Tsymbaliuk A., Surface defects in gauge theory and {KZ} equation,
 \href{https://doi.org/10.1007/s11005-022-01511-8}{\textit{Lett. Math. Phys.}} \textbf{112} (2022), 28, 53~pages,
 \href{https://arxiv.org/abs/2103.12611}{arXiv:2103.12611}.

\bibitem{Reshe-Bethe}
Reshetikhin N., Jackson-type integrals, {B}ethe vectors, and solutions to a
 difference analog of the {K}nizhnik--{Z}amolodchikov system, \href{https://doi.org/10.1007/BF00420749}{\textit{Lett.
 Math. Phys.}} \textbf{26} (1992), 153--165.

\bibitem{Reshe}
Reshetikhin N., The {K}nizhnik--{Z}amolodchikov system as a~deformation of the
 isomonodromy problem, \href{https://doi.org/10.1007/BF00420750}{\textit{Lett. Math. Phys.}} \textbf{26} (1992),
 167--177.

\bibitem{Rosengren:2003nee}
Rosengren H., An elementary approach to {$6j$}-symbols (classical, quantum,
 rational, trigonometric, and elliptic), \href{https://doi.org/10.1007/s11139-006-0245-1}{\textit{Ramanujan~J.}} \textbf{13}
 (2007), 131--166, \href{https://arxiv.org/abs/math.CA/0312310}{arXiv:math.CA/0312310}.

\bibitem{Shakirov:2021krl}
Shakirov {\relax{Sh}}., Non-stationary difference equation for $q$-{V}irasoro
 conformal blocks, \textit{Lett. Math. Phys.}, to appear, \href{https://arxiv.org/abs/2111.07939}{arXiv:2111.07939}.

\bibitem{Shi}
Shiraishi J., Affine screening operators, affine {L}aumon spaces and
 conjectures concerning non-stationary {R}uijsenaars functions,
 \href{https://doi.org/10.1093/integr/xyz010}{\textit{J.~Integrable Syst.}} \textbf{4} (2019), xyz010, 30~pages,
 \href{https://arxiv.org/abs/1903.07495}{arXiv:1903.07495}.

\bibitem{Shiraishi:1995rp}
Shiraishi J., Kubo H., Awata H., Odake S., A~quantum deformation of the
 {V}irasoro algebra and the {M}acdonald symmetric functions, \href{https://doi.org/10.1007/BF00398297}{\textit{Lett.
 Math. Phys.}} \textbf{38} (1996), 33--51, \href{https://arxiv.org/abs/q-alg/9507034}{arXiv:q-alg/9507034}.

\bibitem{TV}
Tarasov V., Varchenko A., Landau--{G}inzburg mirror, quantum differential
 equations and {$q$KZ} difference equations for a~partial flag variety,
 \href{https://doi.org/10.1016/j.geomphys.2022.104711}{\textit{J.~Geom. Phys.}} \textbf{184} (2023), 104711, 58~pages,
 \href{https://arxiv.org/abs/2203.03039}{arXiv:2203.03039}.

\bibitem{VarchenkoCMP}
Varchenko A., Quantized {K}nizhnik--{Z}amolodchikov equations, quantum
 {Y}ang--{B}axter equation, and difference equations for {$q$}-hypergeometric
 functions, \href{https://doi.org/10.1007/BF02101745}{\textit{Comm. Math. Phys.}} \textbf{162} (1994), 499--528.

\end{thebibliography}


\pdfbookmark[1]{References}{ref}
\LastPageEnding

\end{document}